\RequirePackage{colortbl}
\documentclass[bj]{imsart}

\usepackage{color}
\usepackage[dvipsnames]{xcolor}

\startlocaldefs
\theoremstyle{plain}

\newtheorem{theorem}{Theorem}[section]
\newtheorem{lemma}[theorem]{Lemma}
\newtheorem{proposition}[theorem]{Proposition}
\newtheorem{corollary}[theorem]{Corollary}

\newtheorem{definition}[theorem]{Definition}

\usepackage{tikz}
\usetikzlibrary{arrows.meta}
\tikzset{>={Latex[width=3mm,length=2mm]}}
\usetikzlibrary{positioning}
\tikzstyle{observed}=[draw, circle, minimum size=20pt, inner sep=0pt, fill=black!10]

\usetikzlibrary{shapes, calc, shapes, arrows, positioning}
\tikzstyle{qedge}=[->,thick,black]
\tikzstyle{neuron}=[draw,circle,minimum size=26pt,inner sep=0pt, fill=black!10]
\tikzstyle{hidden}=[draw,circle,minimum size=26pt,inner sep=0pt, fill=white]
\usetikzlibrary{arrows.meta}
\tikzset{>={Latex[width=3mm,length=2mm]}}
\tikzstyle{arr}=[->, thick, black]
\usetikzlibrary{decorations.text}
\usetikzlibrary{shadows,shadings,shapes.symbols}
\tikzset{
    double color fill/.code 2 args={
        \pgfdeclareverticalshading[%
            tikz@axis@top,tikz@axis@middle,tikz@axis@bottom%
        ]{diagonalfill}{100bp}{%
            color(0bp)=(tikz@axis@bottom);
            color(50bp)=(tikz@axis@bottom);
            color(50bp)=(tikz@axis@middle);
            color(50bp)=(tikz@axis@top);
            color(100bp)=(tikz@axis@top)
        }
        \tikzset{shade, left color=#1, right color=#2, shading=diagonalfill}
    }
}
\usepackage{booktabs}
\usepackage{amsfonts}
\usepackage{mathtools}
\DeclarePairedDelimiter{\ceil}{\lceil}{\rceil}
\usepackage{bbm}
\usepackage{multirow}

\newcommand{\bo}{\boldsymbol }
\newcommand{\aaa}{\boldsymbol \alpha}
\newcommand{\btheta}{\boldsymbol \theta}
\newcommand{\nnu}{\boldsymbol\nu}
\newcommand{\pp}{\boldsymbol p}
\newcommand{\ee}{\boldsymbol e}

\newcommand{\yy}{\boldsymbol y}
\newcommand{\uu}{\boldsymbol u}
\newcommand{\cg}{\boldsymbol g}
\newcommand{\mca}{\mathcal{A}}
\newcommand{\mcb}{\mathcal{B}}
\newcommand{\one}{\mathbf 1}
\newcommand{\zero}{\mathbf 0}
\newcommand{\G}{\mathbf G}
\newcommand{\Q}{\mathbf Q}
\newcommand{\I}{\mathbf I}
\newcommand{\MP}{\mathbb P}

\newcommand{\vect}{\text{\normalfont{vec}}}
\newcommand{\rank}{\text{\normalfont{rank}}}

\newcommand{\ch}{\text{\normalfont{Child}}}

\newcommand{\paren}[1]{\left(#1\right)}
\newcommand{\bcdot}{\boldsymbol\cdot}

\def\tilde{\widetilde}

\newcommand{\bcolon}{\boldsymbol{:}}
\newcommand{\ov}{\overline }
\newcommand{\indep}{\perp\!\!\!\perp}
\newcommand{\notindep}{\not\! \perp\!\!\!\perp}
\usepackage[plain,noend]{algorithm2e}

\endlocaldefs

\begin{document}

\begin{frontmatter}
\title{Blessing of dependence: identifiability and geometry of
discrete models with multiple binary latent variables}
\runtitle{Blessing of dependence}

\begin{aug}
\author[A]{\inits{s}\fnms{Yuqi}~~\snm{Gu}~~{yuqi.gu@columbia.edu}}

\address[A]{Department of Statistics, Columbia University, New York, NY, USA}
\end{aug}

\begin{abstract}
Identifiability of discrete statistical models with latent variables is known to be challenging to study, yet crucial to a model's interpretability and reliability. This work presents a general algebraic technique to investigate identifiability of discrete models with latent and graphical components. Specifically, motivated by diagnostic tests collecting multivariate categorical data, we focus on discrete models with multiple binary latent variables. We consider the BLESS model, in which the latent variables can have arbitrary dependencies among themselves while the latent-to-observed measurement graph takes a ``star-forest'' shape. We establish necessary and sufficient graphical criteria for identifiability, and reveal an interesting and perhaps surprising geometry of blessing-of-dependence: under the minimal conditions for generic identifiability, the parameters are identifiable if and only if the latent variables are not statistically independent. Thanks to this theory, we can perform formal hypothesis tests of identifiability in the boundary case by testing marginal independence of the observed variables. {In addition to the BLESS model, we also use the technique to show identifiability and the blessing-of-dependence geometry for a more flexible model, which has a general measurement graph beyond a start forest.} Our results give new understanding of statistical properties of graphical models with latent variables. They also entail useful implications for designing diagnostic tests or surveys that measure binary latent traits.
\end{abstract}


\end{frontmatter}

\section{Introduction}

Discrete statistical models with latent variables and graphical components are widely used across many disciplines, such as
Noisy-Or Bayesian networks in medical diagnosis \citep{shwe1991qmrdt, halpern2013noisyor}, binary latent skill models in cognitive diagnosis \citep{chen2015qmat, xu2017rlcm, gu2023jmle}, and restricted Boltzmann machines and their variants in machine learning \citep{hinton2006fast, goodfellow2016deep}.
Incorporating latent variables into graphical models can greatly enhance the flexibility of a model. But such flexibility comes at a cost of  increasing model complexity and statistical subtlety, including identifiability as a fundamental and challenging issue.
In many applications, the latent variables carry substantive meaning such as specific diseases in medical settings and certain skills in educational settings, so uniquely identifying the model parameters and latent structure is of paramount practical importance to ensure valid interpretation \citep[e.g.][]{bing2020overlap, bing2020detecting}.
This work presents a general algebraic technique to investigate identifiability of discrete models with latent and graphical components, characterize the minimal identifiability requirements for a class of such models motivated by diagnostic test applications, and along the way reveal a new geometry about multidimensional latent structures -- the blessing of dependence for identifiability.

A set of parameters for a family of models are said to be identifiable, if distinct values of the parameters correspond to distinct distributions of the observed variables.
Identifiability is a fundamental prerequisite for valid statistical inference. 
Identifiability of discrete statistical models with latent variables is known to be challenging to study, partly due to their inherent nonlinearity.
Latent class models \citep[LCMs;][]{lazarsfeld1968latent} are the simplest form of discrete latent structure models, which assumes a univariate discrete latent variable renders the multivariate categorical responses conditional independent.
Despite the seemingly simple structure and the popularity of LCMs in various applications, their identifiability issues eluded researchers for decades.
\cite{goodman1974} investigated several specific small-dimensional LCMs, some being identifiable and some not.
\cite{gyllenberg1994non} proved LCMs with binary responses are not \textit{strictly identifiable}. \cite{carreira2000practical} empirically showed the so-called practical identifiability of LCMs using simulations. 
And finally, \cite{allman2009} provided a rigorous statement about the \textit{generic identifiability} of LCMs, whose proof leveraged Kruskal's Theorem from \cite{kruskal1977three} on the uniqueness of three-way tensor decompositions.

To be concrete, \textit{strict identifiability} means model parameters are identifiable everywhere in some parameter space $\mathcal T$.
A slightly weaker notion, \textit{generic identifiability} proposed by \cite{allman2009}, {is defined as the situation where identifiability occurs except for a subset $\mathcal N$ of the parameter space, with $\mathcal N$ being the zero-set of nonzero polynomials in the model parameters. In parametric settings with a finite number of parameters, a zero-set of polynomials $\mathcal N$ is either the whole parameter space, or a lower-dimensional subset of it and thus occupying Lebesgue measure zero in the parameter space.} 
In some cases, these measure-zero subsets may be trivial, such as simply being the boundary of the parameter space. In some other cases, however, these subsets may be embedded in the interior of the parameter space, or even carries rather nontrivial geometry and interesting statistical interpretation (as is the case in this work under minimal conditions for generic identifiability).
A precise characterization of the measure-zero subset where identifiability breaks down is essential to performing correct statistical analysis and hypothesis testing \citep{drton2009lrt}. But it is often hard to obtain a complete understanding of such sets or to derive sharp conditions for identifiability in complicated latent variable models. {These issues become more challenging when there exist sparse graphs in a latent variable model, which would induce many additional constraints on the model parameters.}

{In the literature, \cite{allman2008covarion} first used Kruskal's Theorem \citep{kruskal1977three} to prove the identifiability of covarion models in phylogenetics.
Later in a seminal paper, \cite{allman2009} established identifiability for various latent structure models by laying out a general framework of leveraging and transforming Kruskal's Theorem. Their proof strategy has been extended to show identifiability in a variety of settings including  stochastic blockmodels, nonparametric hidden Markov models, and psychometric models
\citep[e.g.,][]{allman2011parameter, gassiat2016inference, fang2019, culpepper2019ordinal, chen2020slcm, fang2020bifactor}.}
These identifiability proofs using Kruskal's Theorem often rely on certain global rank conditions of the tensor formulated under the model.
Instead, we characterize a useful transformation property of the Khatri-Rao tensor products of arbitrary discrete variables' probability tables.
We then use this property to investigate how any specific parameter impacts the zero set of polynomials induced by the latent and graphical constraints.
This general technique covers as a special case a result in \cite{xu2017rlcm} for restricted latent class models with binary responses.
Our approach will allow us to study identifiability at the {finest} possible scale (rather than checking global rank conditions of tensors), and hence help 
characterize the aforementioned measure-zero non-identifiable sets.

We provide an overview of our results.
Motivated by epidemiological and educational diagnostic tests, we focus on discrete models with multiple binary latent variables, where the latent-to-observed measurement graph is a forest of star trees. Namely, each latent variable can have several observed noisy proxy variables as children. 
We allow the binary latent variables to have any possible dependencies among themselves. 
Call this model the \textit{Binary Latent cliquE Star foreSt (BLESS)} model. 
We characterize the necessary and sufficient graphical criteria for strict and generic identifiability, respectively, of the BLESS model; this includes identifying both the discrete star-forest structure and the continuous parameters. 
Under the minimal conditions for generic identifiability that each latent variable has \emph{exactly two} observed children, we show that the measure-zero set $\mathcal N$ in which identifiability breaks down is the independence model of the latent variables.
That is, our identifiability condition delivers a geometry of \textit{blessing-of-dependence} -- the statistical dependence between latent variables can help restore identifiability.
Building on the blessing of dependence, we propose a formal statistical hypothesis test of identifiability in the boundary case. In this case, testing identifiability amounts to testing the marginal dependence of the latent variables' observed children.

{We point out that the blessing-of-dependence is not a new concept in the literature, in that it has been discovered for some other latent variable models.} For example, in the traditional factor analysis model with \emph{continuous Gaussian} latent and observed variables, it is known that if two latent factors are each measured by two observed variables, then the parameters are identifiable if and only if the two latent factors are correlated \citep[see, e.g. Chapter 7 in][]{bollen1989sem}. 
{As another example, independent nonparametric mixture models are not identifiable in
general; however, \cite{gassiat2016inference} established that hidden Markov models with nonparametric components are identifiable. 
This result implies that the latent dependence in the form of a latent Markov model helps with identifiability.
Also, \cite{gassiat2016nonpahmm} proved that, for a family of translation mixture models, identifiability holds without any assumption on the translated distribution provided that the latent variables are indeed not independent.}
On the other hand, for discrete non-Gaussian graphical models with latent variables, the identifiability issue can be more complicated because the observed distributions cannot be simply summarized as covariance matrices but rather take the form of higher-order tensors subject to graphical constraints. 
To this end, this work contributes a generally useful technique to study identifiability and reveal new geometry for such discrete models.

The rest of this paper is organized as follows.
Section \ref{sec-setup} introduces the formal setup of the BLESS model and several relevant identifiability notions.
Section \ref{sec-main} presents the main theoretical results of identifiability and overviews our general proof technique.
Section \ref{sec-extend} extends beyond the BLESS model and shows identifiability and the blessing-of-dependence geometry in two more complex model setups: one with a higher-order latent structure, and the other with a general measurement graph beyond a star forest.
Section \ref{sec-test} proposes a statistical hypothesis test of identifiability. 
Section \ref{sec-prac} presents a real-world example. 
Section \ref{sec-disc} provides further discussions and concludes the paper.
The Supplementary Material \cite{supplement} contains the technical proofs of all theoretical results, details of the algorithms, and an additional real-world example.

\section{Model setup and identifiability notions}\label{sec-setup}

\subsection{Binary Latent cliquE Star foreSt (BLESS) model}
We next introduce the setup of the BLESS model, the focus of this study.
For an integer $m$, denote $[m]=\{1,\ldots,m\}$. For a $K$-dimensional vector $\bo x=(x_1,\ldots, x_K)$ and some index $k\in[K]$, denote the $(K-1)$-dimensional vector by $\bo x_{-k} = (x_1,\ldots,x_{k-1},x_{k+1},\ldots,x_K)$.
Consider discrete statistical models with $K$ binary latent variables $a_1,\ldots,a_K 
 \in \{0,1\}$ and $p$ categorical observed variables $y_1,\ldots,y_p$ with $y_j\in[d]$.
It is possible to extend our identifiability results to the case where $y_j\in[d_j]$ with different $d_1,d_2,\ldots,d_p$, but for ease of exposition, we focus on the case of a common number of categories across all observed variables.
{Allowing $d\geq 2$ covers both the binary response case ($d=2$) and the polytomous response case ($d>2$).}
Both the latent vector $\bo a = (a_1,\ldots,a_K) \in \{0,1\}^K$ and the observed vector $\bo y=(y_1,\ldots,y_p) \in [d]^p$ are subject-specific, and have their realizations for each subject $i$ in a random sample.
For two random vectors (or variables) $\bo x$ and $\bo y$, write $\bo x \perp\!\!\!\perp \bo y$ if $\bo x$ and $\bo y$ are statistically independent, and $\bo x \not\! \perp\!\!\!\perp \bo y$ otherwise.

A key structure in the BLESS model is the latent-to-observed {measurement graph}. 
This is a bipartite graph with directed edges from the latent $a_k$'s to the observed $y_j$'s indicating direct statistical dependence.
The BLESS model posits that the measurement graph is a forest of star trees; namely,
each latent variable can have multiple observed variables as \emph{children}, but each observed variable has exactly one latent \emph{parent}. 
{Although assuming that each observed variable has exactly one latent parent seems to be somewhat restrictive, we point out that the dependence among the latent variables allows the observables to still have rich joint distributions.
This is because in the BLESS model, we allow the $K$ latent variables to be \emph{arbitrarily dependent}; e.g., the latent dependence can be induced by a complicated graphical model among the latent variables themselves or even induced by some deeper latent structures. 
In Section \ref{subsec-pyramid}, we will provide a concrete example where the dependence among latent variables is induced by a deeper-layer, high-order discrete latent structure; for that model we can still apply our identifiability result for the BLESS model.}
Next we introduce the mathematical notation to equivalently represent the measurement graph.  
Define a $p\times K$ \emph{graphical matrix} $\G=(g_{j,k})$ with binary entries, where $g_{j,k}=1$ indicates $a_k$ is the latent parent of $y_j$ and $g_{j,k}=0$ otherwise.
Each row of $\G$ contains exactly one entry of ``1'' due to the star-forest graph structure.
{For $j\in[p]$, denote the $j$th row vector of matrix $\G$ as $\bo g_j$.}
Statistically, the conditional distribution of $y_j\mid \bo a$ equals that of $y_j\mid a_{k}$ if and only if $g_{j,k}=1$.
We can therefore write the conditional distribution of $y_j$ given the latent variables as:
\begin{align*}
\forall c_j \in[d],\quad \mathbb P(y_j = c_j \mid \bo a, \G) = 
\mathbb P(y_j = c_j \mid a_{k},~ g_{j,k}=1)
=
\begin{cases}
\theta^{(j)}_{c_j\mid 1}, & \text{if } a_{k} = 1;\\[3mm]
\theta^{(j)}_{c_j\mid 0}, & \text{if } a_{k} = 0.
\end{cases}
\end{align*}

{For an integer $M\geq 2$, denote the $(M-1)$-dimensional  probability simplex embedded in the $M$-dimensional Euclidean space by $\mathcal S^{M-1}:= \{(x_1,\ldots,x_M): x_m\geq 0 ~\forall m\in[M],~ \sum_{m=1}^M x_m = 1\}$.}
To complete the model specification, we need to describe the distribution of the latent variables $\bo a=(a_1,\ldots,a_K)$. 
We  adopt the flexible saturated model by endowing each binary latent pattern $\aaa\in\{0,1\}^K$ with a  proportion parameter $\nu_{\aaa}=\mathbb P(\bo a = \aaa)$ satisfying $\sum_{\aaa\in\{0,1\}^K} \nu_{\aaa}=1$, where $\bo a$ is the latent profile of a random subject in the population. 
{We use a bold vector $\nnu$ to denote a $2^K$-dimensional vector which characterizes the probability mass function (PMF) of the $K$-dimensional binary latent vector $\bo a$. 
The $\nnu$ lies in the simplex $\mathcal S^{2^K-1}$ and it has $\nu_{\aaa}$ as entries with $\aaa$ ranging in $\{0,1\}^K$.}
Note that this saturated model parameterization covers many constrained latent variable distributions as special cases. For instance, if some latent graph exists among the latent variables or there exists some higher-order latent structures, the resulting joint distribution of the latent vector $\bo a$ would still satisfy our general assumption on $\nnu$; see Section \ref{subsec-pyramid} for a concrete example. 
{In such cases, the proposed conditions on $\G$ remain sufficient for identifying the parameters $\nnu$, whereas whether those parameters underlying $\nnu$ in the more specialized model are identifiable can then be studied by assuming $\nnu$ is already identified.}

Under the widely adopted local independence assumption (i.e., observed variables are conditionally independent given the latent), the probability mass function of the observed vector $\bo y$ takes the form:
\begin{align}\label{eq-model}
    \mathbb P(\bo y = \bo c\mid \G,  \bo \theta, \nnu)
    =
    \sum_{\aaa\in\{0,1\}^K} \nu_{\aaa} \prod_{j=1}^p 
    \prod_{k=1}^{K}
    \left[
    \left(\theta^{(j)}_{c_j\mid 1}\right)^{\alpha_{k}} 
    \cdot
    \left(\theta^{(j)}_{c_j\mid 0}\right)^{1-\alpha_{k}}\right]^{g_{j,k}},
\end{align}
where $\bo c=(c_1,\ldots,c_p)^\top\in\times_{j=1}^p [d]$ is an arbitrary response pattern.
We name the model as \textit{Binary Latent cliquE Star foreSt} (BLESS) model; see the later Figure \ref{fig-graph} for graphical model representations of the model with $K=5$ latent variables.
{Throughout this work, we make the following two assumptions on the parameters in a BLESS model:
\begin{align}\label{eq-nu-positive}
\nu_{\aaa}>0 \text{ for all }\aaa\in\{0,1\}^K; 
\\
\label{eq-flip}
\theta^{(j)}_{c_j\mid 1} > \theta^{(j)}_{c_j\mid 0} ~\text{ for }~ j\in[p],~ c_j\in[d-1].
\end{align}}
Here \eqref{eq-nu-positive} is our {only} assumption on the latent variable distribution, which simply requires $\nnu$ not to be on the boundary of the probability simplex $\mathcal S^{2^K-1}$.
If, however, $\nu_{\aaa}=0$ for certain $\aaa$, then the parameter space for proportions is deficient, which will change the sufficiency and necessity of the identifiability conditions; we leave the consideration of generic identifiability in this setting for future work.
As for \eqref{eq-flip}, the goal of this assumption is to avoid the non-identifiablility issue associated with the sign flipping of each binary latent variable ($a_k$ flipping between 0 and 1).
Assuming \eqref{eq-flip} could be understood as fixing the interpretation of $a_k$ to that possessing the latent trait als suys increases the response probability to the first $d-1$ non-baseline categories.
We emphasize that fixing any other direction of the inequality different from \eqref{eq-flip} equally works for our identifiability arguments; for example, one can assume $\theta^{(j)}_{1\mid 1} < \theta^{(j)}_{1\mid 0}$ and $\theta^{(j)}_{c_j\mid 1} > \theta^{(j)}_{c_j\mid 0}$ for $c_j\geq 2$.
The key in such assumptions like \eqref{eq-flip} is simply to avoid the equality $\theta^{(j)}_{c_j\mid 1} = \theta^{(j)}_{c_j\mid 0}$, which would lead to certain singularity and non-identifiability of some parameters.

In real-world applications, the BLESS model can be useful in epidemiological diagnostic tests, educational assessments, and social science surveys, where the presence/absence of multiple latent characteristics are of interest and there are several observed proxies measuring each of them.
For instance, in disease etiology in epidemiology \citep{wu2017nested}, we can use each $a_k$ to denote the presence/absence of a pathogen, and for each pathogen a few noisy diagnostic measures $y_j$'s are observed as the children of $a_k$.
See Section \ref{sec-prac} for another real-world example.
Our BLESS model is also interestingly connected to a family of models used in causal discovery and machine learning, the \emph{pure-measurement} models in \cite{silva2006latent}. Those are linear models of continuous variables, where the latent variables are connected in an acyclic causal graph; the commonality with the BLESS model is that each observed variable has at most one latent parent. The BLESS model can be thought of as a discrete analogue of such a pure-measurement model in \cite{silva2006latent}, and more general in terms of the latent dependence structure.



\subsection{Strict, generic, and local identifiability}

{Throughout this work, we assume the number of latent variables $K$ is fixed and known.}
We first define strict identifiability. 
All the model parameters are included in the identifiability consideration, including the conditional probabilities $\btheta = \{\theta^{(j)}_{c_j\mid 1}, \theta^{(j)}_{c_j\mid 0}\}$, the proportions $\nnu$, and the discrete measurement graph structure $\G$.

\begin{definition}[Strict Identifiability]
\label{def-str}
The BLESS model is said to be strictly identifiable, if for any valid parameters $(\G, \btheta, \nnu)$, the following equality holds if and only if $(\overline{\G}, \overline\btheta,  \overline\nnu)$ and $(\G, \btheta, \nnu)$ are identical up to a permutation of $K$ latent variables:
	\begin{align}\label{eq-id}
		\mathbb P(\bo y = \bo c\mid \G, \btheta, \nnu)
		=
		\mathbb P(\bo y = \bo c\mid \overline{\G}, \overline\btheta,  \overline\nnu),
		\quad
		\forall \bo c\in \times_{j=1}^p [d].
	\end{align}
\end{definition}

The ``identifiable up to latent variable permutation'' statement in Definition \ref{def-str} is an inevitable but trivial identifiability issue common to exploratory latent variable models, such as exploratory factor analysis and mixture models.
Note that if we consider the case where $\G$ is fixed and known, identifiability of the continuous parameters $\bo\theta$ and $\nnu$ are not subject to the latent variable permutation, because $\G$ matrix already fix the order of the latent variables via its columns.
We next define generic identifiability in the context of the BLESS model. Generic identifiability is a concept proposed and popularized by \cite{allman2009}.
Given a graphical matrix $\G$ and some valid continuous parameters $(\btheta,\nnu)$,
define: 
\begin{align}\notag
\mathcal N^{\G} = 
&~\{(\btheta, \nnu)\text{ are associated with some } \G:
~
\text{there exists}~ (\overline\btheta, \overline\nnu)~\text{associated with some } \overline{\G} \text{ such}~
\\ \notag
&~~
\text{that}~ \mathbb P(\yy\mid\G, \btheta, \nnu) = \mathbb P(\yy\mid\overline\G, \overline\btheta, \overline\nnu),\text{ where $(\overline{\G}, \overline\btheta,  \overline\nnu)$ and $(\G, \btheta, \nnu)$ are \emph{not} identical}\\
\label{eq-ns}
&~~ \text{after any latent variable permutation}\}.
\end{align}

\begin{definition}[Generic Identifiability]\label{def-genid}
A BLESS model is said to be generically identifiable, if for  valid parameters $(\G, \btheta, \nnu)$, the set $\mathcal N^{\G}$ defined in \eqref{eq-ns} has measure zero with respect to the Lebesgue measure on the parameter space of $(\btheta,\nnu)$. 
\end{definition}

Generic identifiability can often suffice for data analyses purposes as pointed out by \cite{allman2009}.
Finally, we define local identifiability of continuous parameters in the model.

\begin{definition}[Local Identifiability]\label{def-local}
Under a BLESS model, a continuous parameter $\mu$ (e.g., some entry of $\btheta$ or $\nnu$) is said to be locally identifiable, if there exists an open neighborhood $\mathcal S$ of every point in the parameter space of $\mu$ such that there does not exist any alternative parameter $\overline\mu\in\mathcal S$ leading to the same distribution of the response vector $\bo y$.
\end{definition}

The lack of local identifiability has severe practical consequences, because in an arbitrarily small neighborhood of the true parameter, there exist infinitely many alternative parameters that give rise to the same observed distributions. This would render any inference conclusions invalid. 

\section{Main theoretical results}\label{sec-main}

\subsection{Theoretical results of generic identifiability and their illustrations}
\label{sec-mainsub}

This subsection presents sharp identifiability conditions and the blessing-of-dependence geometry for the BLESS model. The later Section \ref{sec-overview} will provide an overview of the general algebraic proof technique used to derive the identifiability results.

It may be expected that each latent variable needs to have at least one observed child (i.e., $\sum_{j=1}^p g_{j,k}\geq 1$) to ensure identifiability of the BLESS model.
What may not be apparent at first is that such a condition is insufficient even for generic or local identifiability to hold, let alone strict identifiability.
Our first conclusion below shows the condition that each latent variable has at least two observed children is necessary for generic identifiability or local identifiability.

\begin{proposition}[Necessary Condition for Generic Identifiability: $\geq 2$ children]
\label{prop-nece}
The following two conclusions hold.
\begin{itemize}
    \item[(a)] If some binary latent variable has only one observed variable as child (i.e., $\sum_{j=1}^p g_{j,k}=1$ for some $k$), then the model is \textbf{not even} generically identifiable or locally identifiable.
    
    \vspace{2mm}
    \item[(b)] Specifically, suppose $a_k$ has only one observed $y_j$ as child, then any of the $\theta^{(j)}_{c\mid 0}$ and $\theta^{(j)}_{c\mid 1}$ for $c\in[d]$, and $\nu_{\aaa}$ for $\aaa\in\{0,1\}^K$ can not be generically or locally identifiable. In an arbitrarily small neighborhood of any of these parameters, there exist alternative parameters that lead to the same distribution of the observables indistinguishable from the truth. 
\end{itemize}
\end{proposition}

Since local or generic identifiability are weaker notions than strict identifiability, the conclusion of ``{not even generically or locally identifiable}'' in Proposition \ref{prop-nece} also implies the failure of strict identifiability.
Such a conclusion has quite severe consequences in parameter interpretation or estimation. 
There will be one-dimensional continuum of each of $\theta^{(j)}_{c\mid 0}$ and $\theta^{(j)}_{c\mid 1}$ for $c\in[d]$, and $\nu_{\aaa}$ for $\aaa\in\{0,1\}^K$, that lead to the same probability mass function of the response vector $\bo y$.
As revealed in part (b) of Proposition \ref{prop-nece}, the parameter space will have ``flat regions'' where identifiability is no hope, hence any statistical analysis in this scenario will be meaningless.

In Figure \ref{fig-prop1}, we provide a numerical example to illustrate Proposition \ref{prop-nece}. Consider $p=5$ binary responses and $K=3$ latent variables with a graphical matrix $\G= (1 0 0;~ 0 1 0;~ 0 0 1;~ 0 1 0;~ 0 0 1)$. This $\G$ indicates that latent variable $\alpha_1$ has only one observed child $y_1$, violating the necessary identifiability condition in Proposition \ref{prop-nece}.
In the left panel of Figure \ref{fig-prop1}, the $x$-axis records nine continuous parameters, including one conditional probability $\theta^{(1)}_{1\mid 1}$ and $2^K=8$ proportions for the binary latent pattern; the black solid line represents true parameters, while the 150 colored lines represent 150 sets of alternative parameters in a neighborhood of the truth constructed based on the proof of Proposition \ref{prop-nece}. 
To see the non-identifiablility, we calculate the probability mass function of the response vector $\bo y$, which has $2^p = 32$ entries, and plot it under the true and alternative parameter sets in the right panel of Figure \ref{fig-prop1}. The $x$-axis in the plot presents the indices of the response patterns $\bo c\in\{0,1\}^5$, and the $y$-axis presents the values of $\mathbb P(\bo y = \bo c\mid \G,\btheta,\nnu)$, where the ``$+$'' symbols correspond to response probabilities given by the true parameters and the ``${\bigcirc}$'' represents those given by the 150 sets of alternative parameters. The response probabilities of the observables given by all the alternative parameters perfectly equal those under the truth.
This illustrates the severe consequence of lack of local identifiability.

\begin{figure}[h!]
    \centering
    \includegraphics[width=0.9\linewidth]{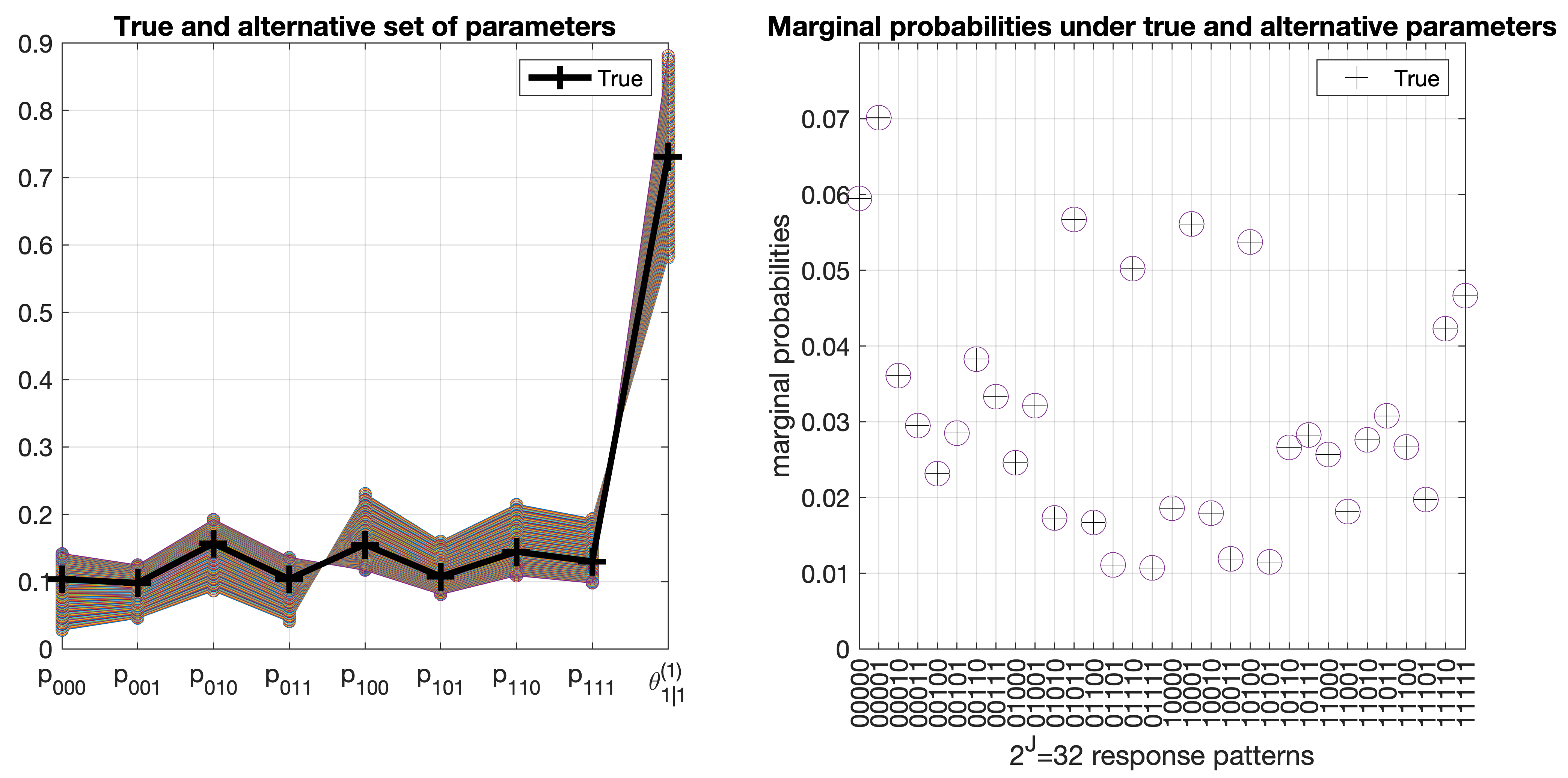}
    \caption{Illustrating Proposition \ref{prop-nece}, severe consequence of lack of local identifiability. 
    Left: the black line represents the true set of parameters and each colored line represents an alternative set of parameters. Right: marginal probability mass functions of the observed $\bo y \in \{0,1\}^5$ are plotted for all the parameter sets, ``$+$'' for the true set overlaid with circles ``${\bigcirc}$'' for 150 alternative sets.}
    \label{fig-prop1}
\end{figure}



Since each latent variable needs to have $\geq 2$ observed children for generic identifiability to possibly hold, next we focus on this setting.
The next theorem establishes a technically nontrivial result that such a condition is sufficient for identifying the matrix $\G$ in the BLESS model.

\begin{theorem}[Identifiability of the Latent-to-observed Star Forest $\G$]\label{thm-graph}
In the BLESS model, 
if each latent variable has at least two observed variables as children (i.e., $\sum_{j=1}^p g_{j,k} \geq 2$ for all $k\in[K]$), then the latent-to-observed star forest structure $\G$ is identifiable up to the permutation of the $K$ latent variables {(that is, $\G$ is identifiable up to the permutation of its $K$ columns)}.
\end{theorem}

The proof of the above Theorem \ref{thm-graph} reveals that to identify $\G$, we only need certain lower-order marginal distributions of $y_j$'s rather than the full joint distribution of all the $p$ observed variables.

We have the following main theorem on generic identifiability, which reveals the ``blessing of dependence'' phenomenon.
Denote by $\text{\normalfont{Child}}(a_k) \;\big |\; a_k$ the conditional distribution of all the child variables of $a_k$ given $a_k$; hence $\text{\normalfont{Child}}(a_k) = \{y_j:\; g_{j,k}=1\}$.
Specifically, the parameters associated with $\text{\normalfont{Child}}(a_k) \;\big |\; a_k$ are the following conditional probabilities:
    $\left\{\btheta^{(j)}:\; y_j\in\text{Child}(a_k)\right\}
    = \left\{\theta^{(j)}_{1:d\mid 0},\; \theta^{(j)}_{1:d\mid 1}:\; g_{j,k}=1\right\}.$

\begin{theorem}[Blessing of Latent Dependence for the Two-children Case]\label{thm-main}
In the BLESS model,
suppose each latent variable has two observed variables as children. 
The following conclusions hold.

\begin{itemize}
    \item[(a)] {$\G$ is identifiable and parameters $(\btheta, \nnu)$ are generically identifiable.}
    
    \item[(b)] In particular, 
    the following two statements (S1) and (S2) are equivalent: 
\begin{itemize}
    \item[(S1)] 
    $a_k \not\!\perp\!\!\!\perp (a_1,\ldots,a_{k-1}, a_{k+1}, \ldots, a_K)$ holds; 
    
    \item[(S2)] parameters associated with the conditional distributions $\text{\normalfont{Child}}(a_k) \;\big |\; a_k$ 
    are identifiable.
\end{itemize} 
\end{itemize}
\end{theorem}

Notably, the case of each latent variable having two children in Theorem \ref{thm-main} forms the exact boundary for the blessing of dependence to play a role.
As long as each latent variable has at least three observed variables as children, the Kruskal's Theorem \citep{kruskal1977three} on the uniqueness of three-way tensor decompositions kicks in to ensure identifiability. We can use an argument similar to that in \cite{allman2009} to establish this conclusion, by concatenating certain observed variables into groups and transforming the underlying $p$-way probability tensor into a three-way tensor. The following proposition formalizes this statement.

\begin{proposition}[Kruskal's Theorem Kicks in for the $\geq 3$ Children Case]\label{prop-3chi}
Under the BLESS model,
if each latent variable has at least three observed children (i.e., $\sum_{j=1}^p g_{j,k} \geq 3$ for all $k\in[K]$), then the model is always strictly identifiable, regardless of the dependence between the latent variables. 
\end{proposition}

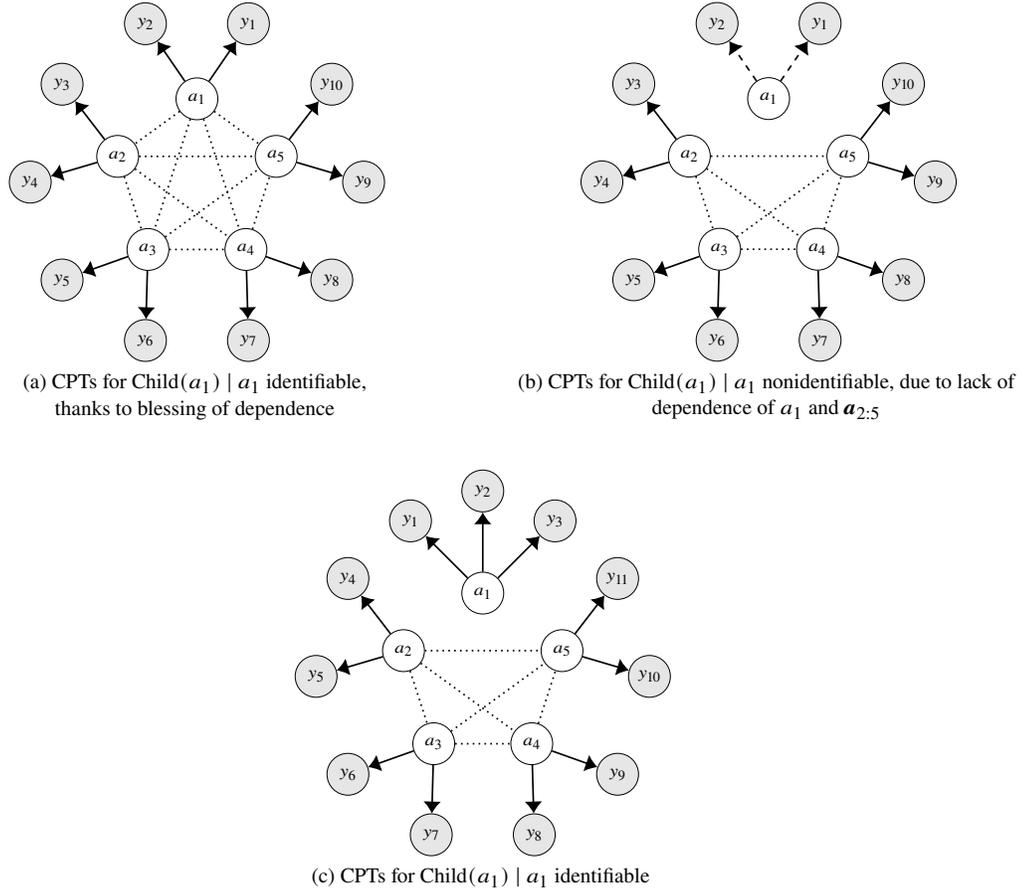
\begin{figure}[h!]
\centering
\begin{minipage}[c]{0.47\linewidth}\centering
\resizebox{0.75\linewidth}{!}{
\begin{tikzpicture}
\def\n {5}
\def\radius {1.4cm}

\foreach \s in {1,...,\n}
{
\node (\s)[draw, circle, minimum size=20pt, inner sep=0pt] at ({90+360/\n * (\s - 1)}:\radius) {$a_{\s}$};
}
\def \nn {10}
\def \radius {2.8cm}

\foreach \ss in {1,...,\nn}
{
\node (y\ss)[draw, circle, minimum size=20pt, inner sep=0pt, fill=black!10] at ({72+360/\nn * (\ss - 1)}:\radius) {$y_{\ss}$};
}

\draw[dotted, thick] (1) -- (2);
\draw[dotted, thick] (1) -- (3);
\draw[dotted, thick] (1) -- (4);
\draw[dotted, thick] (1) -- (5);
\draw[dotted, thick] (2) -- (3);
\draw[dotted, thick] (2) -- (4);
\draw[dotted, thick] (2) -- (5);
\draw[dotted, thick] (3) -- (4);
\draw[dotted, thick] (3) -- (5);
\draw[dotted, thick] (4) -- (5);

\draw[->, thick] (1) -- (y1);
\draw[->, thick] (1) -- (y2);
\draw[->, thick] (2) -- (y3);
\draw[->, thick] (2) -- (y4);
\draw[->, thick] (3) -- (y5);
\draw[->, thick] (3) -- (y6);
\draw[->, thick] (4) -- (y7);
\draw[->, thick] (4) -- (y8);
\draw[->, thick] (5) -- (y9);
\draw[->, thick] (5) -- (y10);
\end{tikzpicture}}

(a) CPTs for $\ch(a_1) \mid a_1$ identifiable,\\
thanks to blessing of dependence
\end{minipage}
\hfill
\begin{minipage}[c]{0.47\linewidth}\centering
\resizebox{0.75\linewidth}{!}{
\begin{tikzpicture}
\def \n {5}
\def \radius {1.4cm}

\foreach \s in {1,...,\n}
{
\node (\s)[draw, circle, minimum size=20pt, inner sep=0pt] at ({90+360/\n * (\s - 1)}:\radius) {$a_{\s}$};
}
\def \nn {10}
\def \radius {2.8cm}

\foreach \ss in {1,...,\nn}
{
\node (y\ss)[draw, circle, minimum size=20pt, inner sep=0pt, fill=black!10] at ({72+360/\nn * (\ss - 1)}:\radius) {$y_{\ss}$};
}
\draw[dotted, thick] (2) -- (3);
\draw[dotted, thick] (2) -- (4);
\draw[dotted, thick] (2) -- (5);
\draw[dotted, thick] (3) -- (4);
\draw[dotted, thick] (3) -- (5);
\draw[dotted, thick] (4) -- (5);

\draw[->, thick, dashed] (1) -- (y1);
\draw[->, thick, dashed] (1) -- (y2);
\draw[->, thick] (2) -- (y3);
\draw[->, thick] (2) -- (y4);
\draw[->, thick] (3) -- (y5);
\draw[->, thick] (3) -- (y6);
\draw[->, thick] (4) -- (y7);
\draw[->, thick] (4) -- (y8);
\draw[->, thick] (5) -- (y9);
\draw[->, thick] (5) -- (y10);
\end{tikzpicture}}

(b) CPTs for $\ch(a_1) \mid a_1$ nonidentifiable, due to lack of dependence of $a_1$ and $\bo a_{2:5}$
\end{minipage}

\bigskip
\begin{minipage}[c]{0.47\linewidth}\centering
\resizebox{0.75\linewidth}{!}{
\begin{tikzpicture}
\def \n {5}
\def \radius {1.4cm}

\foreach \s in {1,...,\n}
{
\node (\s)[draw, circle, minimum size=20pt, inner sep=0pt] at ({90+360/\n * (\s - 1)}:\radius) {$a_{\s}$};
}
\def \nn {11}
\def \radius {2.8cm}

\foreach \ss in {4,...,\nn}
{
\node (y\ss)[observed] at ({360/10 * (\ss)}:\radius) {$y_{\ss}$};
}

\node (y1) [observed, above left = 1 cm of 1] {$y_1$};
\node (y2) [observed, above = 1 cm of 1] {$y_2$};
\node (y3) [observed, above right = 1 cm of 1] {$y_3$};


\draw[dotted, thick] (2) -- (3);
\draw[dotted, thick] (2) -- (4);
\draw[dotted, thick] (2) -- (5);
\draw[dotted, thick] (3) -- (4);
\draw[dotted, thick] (3) -- (5);
\draw[dotted, thick] (4) -- (5);

\draw[->, thick] (1) -- (y1);
\draw[->, thick] (1) -- (y2);
\draw[->, thick] (1) -- (y3);

\draw[->, thick] (2) -- (y4);
\draw[->, thick] (2) -- (y5);
\draw[->, thick] (3) -- (y6);
\draw[->, thick] (3) -- (y7);
\draw[->, thick] (4) -- (y8);
\draw[->, thick] (4) -- (y9);
\draw[->, thick] (5) -- (y10);
\draw[->, thick] (5) -- (y11);
\end{tikzpicture}}

(c) CPTs for $\ch(a_1) \mid a_1$ identifiable
\end{minipage}

\caption{
CPTs refer to Conditional Probability Tables.
All nodes are discrete random variables, with $a_k\in\{0,1\}$ latent and $y_j\in\{1,\ldots,d\}$ observed. The parameters corresponding to the dashed directed edges in (b) are unidentifiable, because $a_1$ is indepedent of $\bo a_{2:5}$. 
}
\label{fig-graph}
\end{figure}

{The proof of Proposition \ref{prop-3chi} builds on Kruskal's Theorem, similar to many existing studies on the identifiability of discrete models. We present this side result to demonstrate the minimum condition under which Kruskal's Theorem directly kicks in to guarantee identifiability. Recall that the main result Theorem \ref{thm-main} assumes that each latent variable has only two observed children, in contrast to the condition assumed in Proposition \ref{prop-3chi}. Therefore, Theorem \ref{thm-main} along with Proposition \ref{prop-3chi} shows that the proposed proof technique can apply to cases where Kruskal's Theorem is not directly applicable.}

It is useful to give a graphical illustration of our identifiability results.
Figure \ref{fig-graph}(a)--(b) illustrate our generic identifiability conclusions and the blessing of dependence phenomenon.
With $K=5$ latent variables each having two observed variables as children (i.e., $\G = (\I_K;\; \I_K)^\top$), the parameters corresponding to Figure \ref{fig-graph}(a) are identifiable due to the dependence indicated by the dotted edges between $a_1,\ldots,a_5$; 
while the parameters corresponding to Figure \ref{fig-graph}(b) are not identifiable due to the lack of dependence between $a_1$ and $\bo a_{-1}:=(a_2,\ldots,a_5)$.
Such identifiability arguments guaranteed by Theorem \ref{thm-main}(b) are of a very fine-grained nature, stating that the dependence between a specific latent variable and the remaining ones determines the identifiability of the conditional probability tables given this very latent variable.

We provide a numerical example {with $K=2$} to corroborate the blessing-of-dependence geometry. 
Consider the BLESS model with each observed variable having $d=3$ categories and $\G=(\I_2;\; \I_2)^\top$.
We randomly generate $M=100$ sets of true parameters of the BLESS model.
Given a fixed sample size $N=10^4$, for each of the $M=100$ parameter sets we further generate $L=200$ independent datasets each with $N$ data points. We use an EM algorithm (Algorithm 1 in the Supplementary Material) to compute the maximum likelihood estimators (MLE) of the model parameters for each dataset; here we focus on estimating continuous parameters $(\btheta,\pp)$ with $\G$ fixed, because $\G$ is guaranteed to be identifiable by Theorem \ref{thm-graph}.
Ten random initializations are chosen for the EM algorithm and the one with the largest log likelihood value is taken as the MLE. After collecting the MLEs, we calculate the Mean Squares Errors (MSEs) of continuous parameters for each of the 100 true parameter sets.

\begin{figure}[h!]
     \centering
         \includegraphics[width=0.42\linewidth]{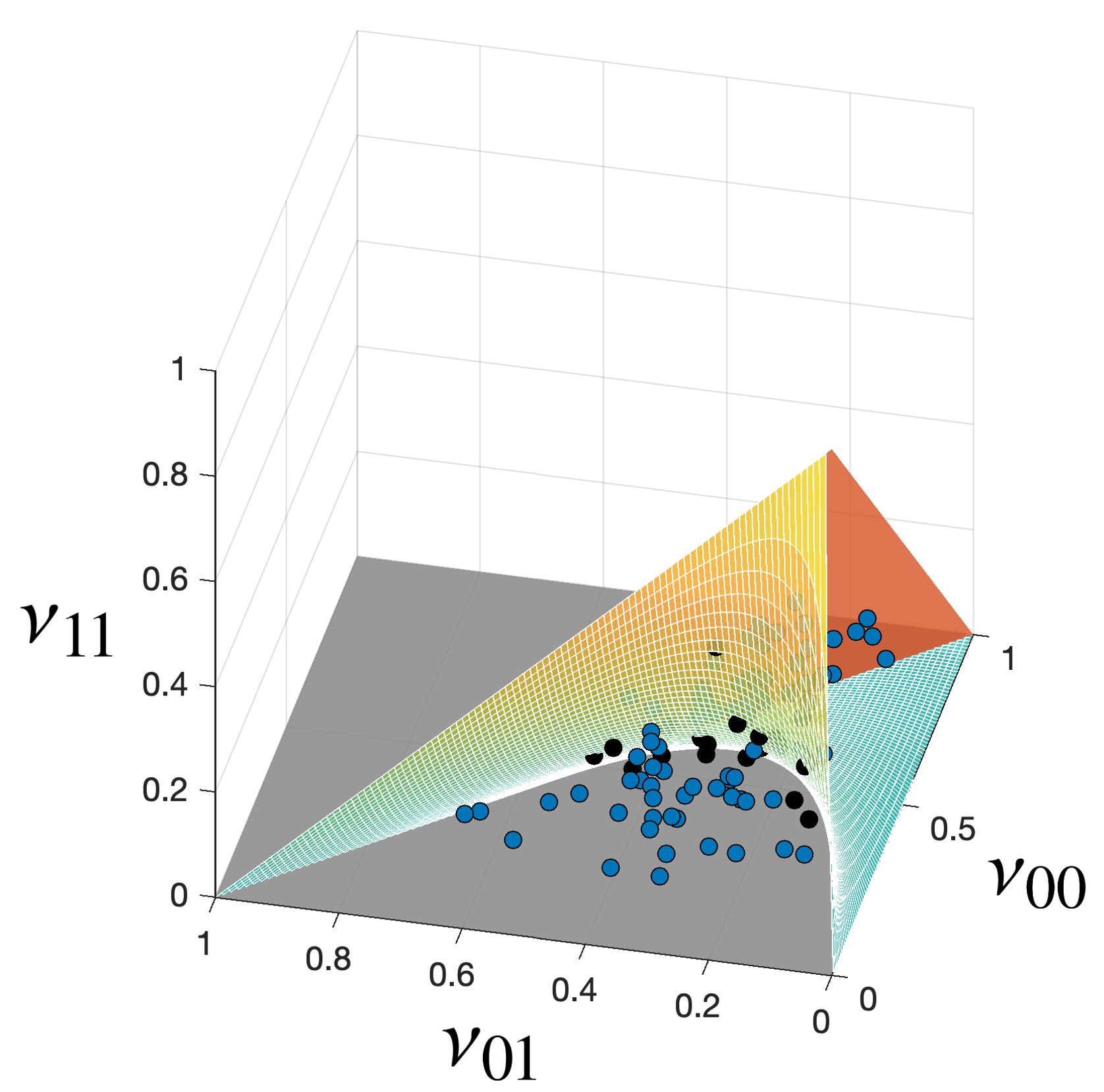}
         \hfill
         \includegraphics[width=0.4\linewidth]{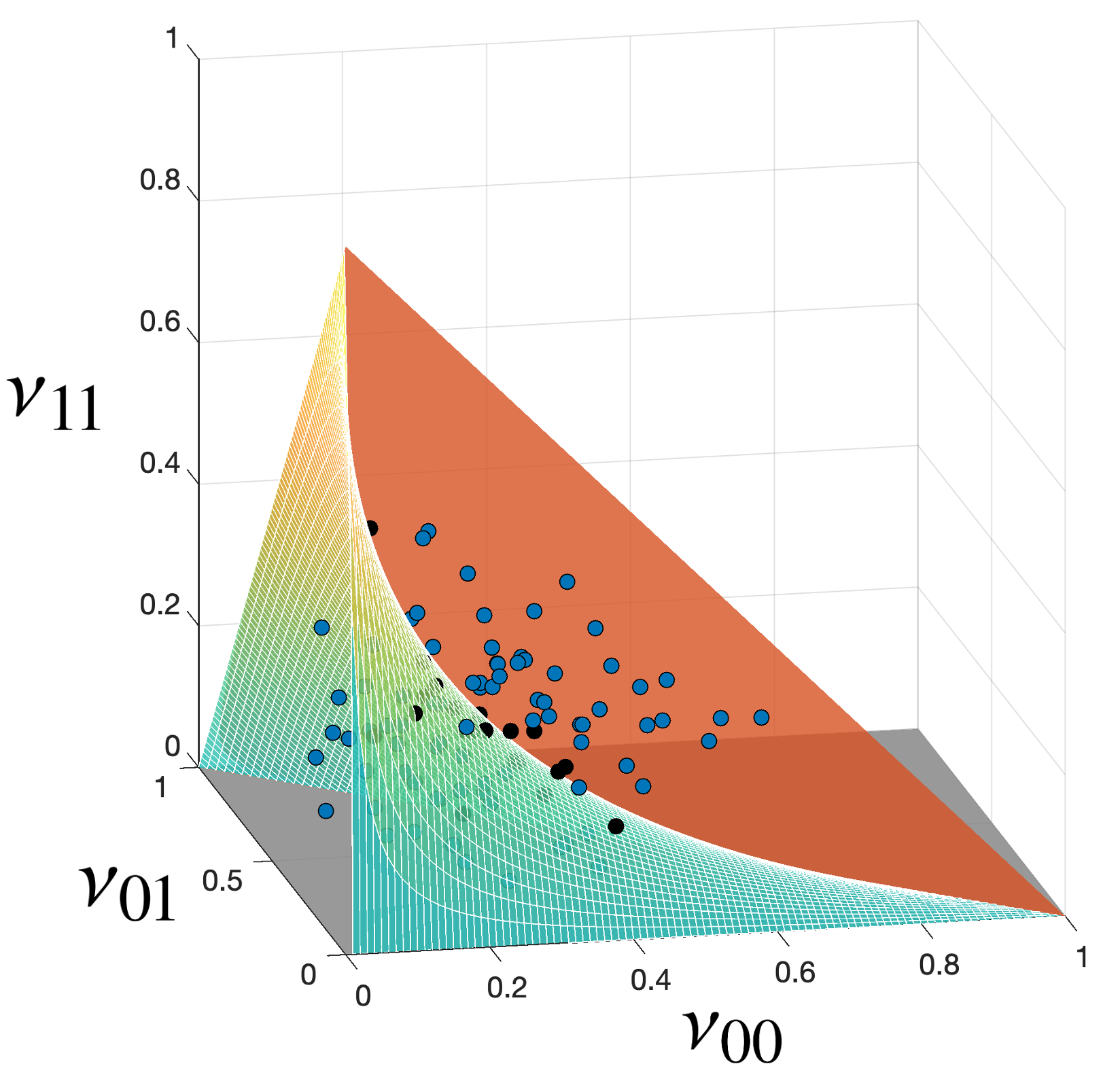}
         \caption{{Corroborating Theorem \ref{thm-main}. Two different views of the probability simplex $\mathcal S^3$ (tetrahedron) for the proportion parameters $\nnu$. The saddle surface $\mathcal N$ embedded in the simplex corresponds to the case with independent latent variables $a_1\indep a_2$.
         Black dots correspond to the 20 parameter vectors $\nnu^{(m)}$ with the largest 20 MSEs among the 100 vectors $\nnu^{(1)},\ldots,\nnu^{(100)}\in \mathcal S^3$, and blue dots correspond to the remaining 80 parameter vectors.}}
         \label{fig-mse}
\end{figure}


{Figure \ref{fig-mse} visualizes that parameter estimation becomes harder when true parameters get closer to the measure-zero non-identifiable subset of the parameter space. 
We next explain the details of this figure.
First note that the distribution of latent variables and the dependence among them are essentially characterized by the proportion parameters $\nnu =  (\nu_{00}, \nu_{01}, \nu_{10}, \nu_{11})$ where $\nnu_{\aaa} = \mathbb P(\bo a = \aaa)$ for $\aaa\in\{0,1\}^2$.
The parameter space for $\nnu$ is the three-dimensional probability simplex $\mathcal S^{3}$, and we choose to visualize $\mathcal S^{3}$ in $\mathbb R^3$ by using $\nu_{00}$, $\nu_{01}$, and $\nu_{11}$ as the $x$-, $y$-, and $z$-coordinates. Since $\nu_{00}, \nu_{01}, \nu_{11} > 0$ and $\nu_{00} + \nu_{01} + \nu_{11} < 1$, the parameter space for $(\nu_{00}, \nu_{01}, \nu_{11})$ takes the shape of a tetrahedron in $\mathbb R^3$ as depicted in the two different views of it in Figure \ref{fig-mse}.
{The triangle colored in orange represents one face of the simplex $\mathcal S^3$ that corresponds to $\nu_{00}+ \nu_{01}+\nu_{11}=1$.}
As a reference, in Figure \ref{fig-mse}(a) and (b) we also plot the measure-zero non-identifiable subset of $\mathcal S^3$, denoted by 
$$
\mathcal N = \{\nnu\in \mathcal S^3:\; \nnu\text{ satisfies } a_1 \indep a_2\}
= \{\nnu\in \mathcal S^3:\;
\nu_{00}\nu_{11} - \nu_{01}\nu_{10} = 0\}.
$$ 
Figure \ref{fig-mse} shows that the above subset $\mathcal N$ takes the shape of a smooth saddle surface embedded in the interior of the parameter space $\mathcal S^3$.
There are $M=100$ points inside the tetrahedron in Figure \ref{fig-mse}(a) and (b), each point corresponding to a particular parameter vector $\nnu^{(m)}\in\mathcal S^3$ where $m=1,2,\ldots,100$.
To inspect how the MSEs vary for different parameter vectors in $\mathcal S^3$, we plot those $\nnu^{(m)}$ with the largest 20 MSEs as black points and plot the remaining 80 vectors as blue points. 
Notably, the two views in Figure \ref{fig-mse} clearly show that the black points are closer to the saddle surface $\mathcal N$ which corresponds to $a_1\indep a_2$.
This observation means that when the true parameters $\nnu^{(m)}$ are closer to the non-identifiable measure-zero set $\mathcal N$, MSEs are larger and accurate estimation becomes statistically harder.
This simulation result empirically corroborates Theorem \ref{thm-main} and illustrates that the submodel with independent latent variables defines a singular subset within the {interior} of the parameter space.}

Summarizing all results in this section, we have the following conclusions.

\begin{corollary}\label{cor-ns} Consider the BLESS model {with a known number of latent variables $K$}. The following statements hold.

\begin{itemize}
\item[(a)]
The condition that each binary latent variable has $\geq 2$ observed variables as children is {necessary and sufficient} for the generic identifiability of the model parameters.
\item[(b)]
The condition that each binary latent variable has $\geq 3$ observed variables as children is {necessary and sufficient} for the strict identifiability of the model parameters.
\end{itemize}
\end{corollary}

It is worth noting that both the minimal conditions for strict identifiability and those for generic identifiability only concern the discrete structure in the model -- the measurement graph $\G$, but not on the specific values of the continuous parameters $\btheta$ or $\nnu$.
{When the graph $\G$ is unknown, the identifiability condition on the true graph structure is not directly checkable from observational data. In practice, after one uses some statistical method to estimate all parameters (including the graph) from the observational data, then they may check whether the estimated graph satisfies the identifiability condition.}

\subsection{Overview of the proof technique and its usefulness}\label{sec-overview}
This subsection provides an overview of our identifiability proof technique.
For ease of understanding, we next describe the technique in the context of multidimensional binary latent variables; we will later explain that these techniques are applicable to more general discrete models with latent and graphical components.
{With $K$ binary variables, we next introduce the binary vector representations of the $2^K$ integers $1,2,3,\ldots,2^K$ by vectors $\aaa_1,\aaa_2,\ldots,\aaa_{2^K}\in\{0,1\}^K$. Specifically, define a $K$-dimensional vector $\bo w=(2^{K-1}, 2^{K-2}, \cdots, 2^0)^\top$ and let
$\aaa_{\ell}^\top \bo w = \ell-1$ for each $\ell=1,2,3,\ldots,2^K.$
The goal of introducing this vector $\bo w$ is to define an unambiguous way of ordering the $2^K$ vectors in $\{0,1\}^K$ as $\aaa_1,\ldots,\aaa_{2^K}$. For example, if $K=3$, then $\bo w = (2^2, 2^1, 2^0)^\top = (4,2,1)^\top$, and the equations $\aaa_{\ell}^\top \bo w = \ell-1$ for all $\ell=1,\ldots,8$ uniquely define the meaning of each binary vector $\aaa_{\ell}$: $\aaa_1=(0,0,0), \aaa_2 = (0,0,1)$, so on and so forth.}

With $p$ discrete observed variables $y_1,\ldots,y_p$, generally denote the conditional distribution of each $y_j$ given latent pattern $\aaa_\ell$ by 
$\theta_{c\mid \aaa_\ell}^{(j)} = \mathbb P(y_j=c\mid \bo a = \aaa_\ell)$, for
$j\in[p]$, $c\in[d]$, $\ell\in[2^K]$.
Note that under the BLESS model, the $\theta_{c\mid \aaa_\ell}^{(j)}$ is a reparametrization of the probabilities $\theta_{c\mid 1}^{(j)}$ and $\theta_{c\mid 0}^{(j)}$.
According to the star-forest measurement graph structure, whether $\theta_{c\mid \aaa_\ell}^{(j)}$ equals $\theta_{c\mid 1}^{(j)}$ or $\theta_{c\mid 0}^{(j)}$ depends only on whether or not the pattern $\aaa_\ell$ possesses the  latent parent of $y_j$.
Mathematically, since vector $\cg_j$ summarizes the parent variable information of $y_j$, we have
\begin{align}\label{eq-thetaeq}
    \theta_{c\mid \aaa_\ell}^{(j)} =
    \begin{cases}
    \theta_{c\mid 1}^{(j)}, & \text{if } \alpha_{\ell,k}=1 \text{ for the $k$ where } g_{j,k}=1;\\[2mm]
    \theta_{c\mid 0}^{(j)}, & \text{if } \alpha_{\ell,k}=0 \text{ for the $k$ where } g_{j,k}=1.
    \end{cases}
\end{align}
In the above expression, the $\alpha_{\ell,k}$ denotes the $k$th entry of the binary pattern $\aaa_\ell$.
For each observed variable index $j\in[p]$, define a $d\times 2^K$ matrix $\bo\Phi^{(j)}$ as
\begin{align*}
    \bo\Phi^{(j)} 
    &= 
    \begin{pmatrix}
     \mathbb P(y_j=1\mid \bo a = \aaa_1) &~ \cdots &~  \mathbb P(y_j=1\mid \bo a = \aaa_{2^K}) \\[2mm]
     \vdots &~ \vdots &~ \vdots \\[2mm]
     \mathbb P(y_j=d\mid \bo a = \aaa_1) &~ \cdots &~  \mathbb P(y_j=d\mid \bo a = \aaa_{2^K})
    \end{pmatrix}
    =
    \begin{pmatrix}
        \theta^{(j)}_{1\mid \aaa_1} &~  \cdots &~~ \theta^{(j)}_{1\mid \aaa_{2^K}} \\[2mm]
        \vdots &~ \vdots &~~ \vdots \\[2mm]
        \theta^{(j)}_{d\mid \aaa_1} &~  \cdots &~~ \theta^{(j)}_{d\mid \aaa_{2^K}}
    \end{pmatrix},
\end{align*}
then $\bo\Phi^{(j)}$ is the conditional probability table of variable $y_j$ given $2^K$ latent patterns. Each column of $\bo\Phi^{(j)}$ is indexed by a pattern $\aaa_\ell$ and gives the conditional distribution of variable $y_j$ given $\aaa_\ell$.
Note that many entries in $\bo\Phi^{(j)}$ are equal due to \eqref{eq-thetaeq}; we deliberately choose this overparameterized matrix notation to facilitate further tensor algebra. The equality of the many parameters in each $\bo\Phi^{(j)}$ will later be carefully exploited when examining identifiability conditions.

Denote by $\bigotimes$ the Kronecker product of matrices.
We also introduce the Khatri-Rao product of matrices following the definition in the tensor decomposition literature \cite{koldabader2009} in order to facilitate the presentation of our new technique.
Specifically, the Khatri-Rao product is a column-wise Kronecker product, and
for two matrices with the same number of columns $\mathbf A=(a_{i,j})=(\bo a_{\bcolon,1}\mid\cdots\mid\bo a_{\bcolon,k})\in\mathbb R^{n\times k}$,
$\mathbf B=(b_{i,j})=(\bo b_{\bcolon,1}\mid\cdots\mid\bo b_{\bcolon,k})\in\mathbb R^{\ell\times k}$, their Khatri-Rao product
$\mathbf A\bigodot \mathbf B \in\mathbb R^{n \ell\times k}$ still has the same number of columns and can be written as
	$\mathbf A\bigodot \mathbf B
	=
	\begin{pmatrix}
		\bo a_{\bcolon,1}\bigotimes\bo b_{\bcolon,1}
		~\mid~ \cdots ~\mid~
		\bo a_{\bcolon,k}\bigotimes\bo b_{\bcolon,k}
	\end{pmatrix}$.
Under the considered model, all the $d^p$ marginal response probabilities form a $p$-way tensor $\bo\Pi=(\pi_{c_1,\cdots,c_p})$, $c_j\in[d]$,
where each entry $\pi_{c_1,\cdots,c_p} = \mathbb P (y_1=c_1,\ldots,y_p=c_p\mid \text{measurement graph structure and parameters})$ denotes the marginal probability of observing the response pattern $\yy=\bo c$ under the latent variable model.
With the above notation, the probability mass function (PMF) of vector $\bo y$ under the BLESS model in \eqref{eq-model} can be equivalently written as
\begin{align}\label{eq-kreq}
    \vect(\bo\Pi)
    = \Big(\bigodot_{j=1}^p \bo\Phi^{(j)}\Big) \cdot \nnu,
\end{align}
where $\vect(\bo\Pi)$ denotes the vectorization of the tensor $\bo\Pi$ into a vector of length $d^p$. The Khatri-Rao product of $\bo\Phi^{(j)}$ in the above display results from the basic local independence assumption in \eqref{eq-model}.
We next state a useful technical lemma. 
The following lemma characterizes a fundamental property of the transformations of Khatri-Rao product of matrices.

\begin{lemma}\label{lem-poly}
Consider an arbitrary set of conditional probability tables $\{\bo\Phi^{(j)}: j\in[p]\}$, where $\bo\Phi^{(j)}$ has size $d_j\times 2^K$ with each column summing to one.
Given any set of vectors $\{{\bo\Delta}_j:\,{j\in[p]}\}$ with $\bo\Delta_j = (\Delta_{j,1},\ldots,\Delta_{j,d_j-1}, 0)^\top \in \mathbb R^{d_j\times 1}$, 
there exists a $\prod_{j=1}^p d_j \times \prod_{j=1}^p d_j$ \textbf{invertible} matrix $\mathbf B:=\mathbf B(\{\bo\Delta_j:\,{j\in[p]}\})$ determined entirely by $\{\bo\Delta_j:\,{j\in[p]}\}$ such that 
\begin{align}\label{eq-algebra}
	\bigodot_{j\in[p]} \Big(\bo\Phi^{(j)}-\bo\Delta_j\bcdot\one^\top_{2^K} \Big)
	&= \mathbf B\left(\{\bo\Delta_j:\,{j\in[p]}\}\right) \bcdot \Big(\bigodot_{j\in[p]} \bo\Phi^{(j)}\Big),
\end{align}
where $\bo\Delta_j\bcdot\one^\top_{2^K}$ is a $d_j\times 2^K$ matrix, of the same dimension as $\bo\Phi^{(j)}$.

In addition, replacing the index $j\in[p]$ in \eqref{eq-algebra} by $j\in S$ where $S$ is an arbitrary subset of $[p]$ on both hand sides still makes the equality holds.
\end{lemma}

Note that Lemma \ref{lem-poly} covers more general settings than are currently considered, as $d_1,d_2,\ldots,d_p$ are allowed to be different.
Lemma \ref{lem-poly} covers as special case a result in \cite{xu2017rlcm} for restricted latent class models with binary responses.
Instead of exclusively considering moments of binary responses as \cite{xu2017rlcm}, our Lemma \ref{lem-poly} characterizes a general algebraic property of Khatri-Rao products of conditional probability tables of multivariate categorical data. 
This property will enable us to exert various transformations on the model parameters to investigate their identifiability. 
We provide a proof of Lemma \ref{lem-poly} below, because it is concise and delivers an insight into our technique's usefulness.
%

\begin{proof}[Proof of Lemma \ref{lem-poly}]
Consider an arbitrary subset $S\in[p]$.
The sum of all the entries in each column of $\bo\Phi^{(j)}$ is one because each column vector is a conditional probability distribution of $y_j$ given a specific latent pattern. Therefore with $\bo\Delta_j = (\Delta_{j,1},\ldots,\Delta_{j,d_j-1}, 0)^\top\in\mathbb R^{d_j}$, we have
\begin{align*}
    \bo\Phi^{(j)}-\bo\Delta_j\bcdot\one^\top_{2^K}
    &= 
    \begin{pmatrix}
        \theta^{(j)}_{1\mid \aaa_1}-\Delta_{j,1} &~  \cdots &~~ \theta^{(j)}_{1\mid \aaa_{2^K}}-\Delta_{j,1} \\[2mm]
        \vdots &~ \vdots &~~ \vdots \\[2mm]
        \theta^{(j)}_{d_j-1\mid \aaa_1}-\Delta_{j,d_j-1} &~  \cdots &~~ \theta^{(j)}_{d_j-1\mid \aaa_{2^K}}-\Delta_{j,d_j-1} \\[4mm]
        \theta^{(j)}_{d_j\mid \aaa_1} &~  \cdots &~~ \theta^{(j)}_{d_j\mid \aaa_{2^K}}
    \end{pmatrix}
    \\
    &=
    \begin{pmatrix}
    1 &~ 0  &~ \cdots &~ 0 & -\Delta_{j,1}\\
    0 &~ 1  &~ \cdots &~ 0 & -\Delta_{j,2}\\
    \vdots &~ \vdots &~ \ddots &~ 0 & \vdots\\
    0 &~ 0  &~ \cdots &~ 1 &\quad -\Delta_{j,d_j-1}\\
    -1 &~ -1  &~ \cdots & -1 & 1
    \end{pmatrix}
    \bcdot
    \begin{pmatrix}
    \theta^{(j)}_{1\mid \aaa_1} &~  \cdots &~~ \theta^{(j)}_{1\mid \aaa_{2^K}} \\[2mm]
        \vdots &~ \vdots &~~ \vdots \\[2mm]
        \theta^{(j)}_{d_j-1\mid \aaa_1} &~  \cdots &~~ \theta^{(j)}_{d_j-1\mid \aaa_{2^K}} \\[4mm]
        1 &~  \cdots &~~ 1
    \end{pmatrix}
    \\
    &=
    \underbrace{\begin{pmatrix}
    1 &~ 0  &~ \cdots &~ 0 & -\Delta_{j,1}\\
    0 &~ 1  &~ \cdots &~ 0 & -\Delta_{j,2}\\
    \vdots &~ \vdots &~ \ddots &~ 0 & \vdots\\
    0 &~ 0  &~ \cdots &~ 1 &\quad -\Delta_{j,d_j-1}\\
    -1 &~ -1  &~ \cdots & -1 & 1
    \end{pmatrix}}_{d_j\times d_j\text{ matrix, denoted by }\tilde{\bo\Delta}_j}
    \bcdot
    \underbrace{\begin{pmatrix}
    1 &~ 0  &~ \cdots &~ 0 &~ 0\\
    0 &~ 1  &~ \cdots &~ 0 &~ 0\\
    \vdots &~ \vdots &~ \ddots &~ \vdots &~ \vdots \\
    0 &~ 0  &~ \cdots &~ 1 &~ 0\\
    1 &~ 1  &~ \cdots &~ 1 &~ 1
    \end{pmatrix}}_{d_j\times d_j\text{ matrix, denoted by }\mathbf C}
    \bcdot~ \bo\Phi^{(j)}
    =:
    \tilde{\bo\Delta}_j \mathbf C \bo\Phi^{(j)}.
\end{align*}
We can see both $\tilde{\bo\Delta}_j$ and $\mathbf C$ {have full rank $d_j$, so their product $\tilde{\bo\Delta}_j \mathbf C$ also has full rank $d_j$.} Then
\begin{align*}
    \bigodot_{j\in S} \Big(\bo\Phi^{(j)}-\bo\Delta_j\bcdot\one^\top_{2^K} \Big) 
    &= 
    \bigodot_{j\in S} \Big(\tilde{\bo\Delta}_j \mathbf C  \bo\Phi^{(j)} \Big)
    =
    \bigotimes_{j\in S} (\tilde{\bo\Delta}_j \mathbf C ) \bcdot \bigodot_{j\in S}  \bo\Phi^{(j)},
\end{align*}
where the last equality follows from basic properties of the Kronecker and Khatri-Rao products and can be verified by checking corresponding entries in the products. Define
$
\mathbf B\left(\{\bo\Delta_j:\,{j\in S}\}\right) : =
\bigotimes_{j\in S} (\tilde{\bo\Delta}_j \mathbf C ),
$
then $\mathbf B\left(\{\bo\Delta_j:\,{j\in S}\}\right)$ is a {$\prod_{j\in S}d_j \times \prod_{j\in S}d_j$} \emph{invertible} matrix because it is the Kronecker product of $|S|$ invertible matrices $\tilde{\bo\Delta}_j \mathbf C$. This proves  Lemma \ref{lem-poly}.
\end{proof}

Recall that many entries in $\bo\Phi^{(j)}$ are constrained equal under the graphical matrix $\G$. 
Now suppose an alternative graphical matrix $\bar\G \in\{0,1\}^{p\times K}$ and some associated alternative parameters $(\bar\btheta, \bar\nnu)$ lead to the same distribution of $\bo y$ as $(\G, \btheta, \nnu)$. 
Then by \eqref{eq-kreq}, equations $(\bigodot_{j\in S} \bo\Phi^{(j)}) \cdot \nnu =
    (\bigodot_{j\in S} \overline{\bo\Phi}^{(j)}) \cdot \overline\nnu$ must hold for an arbitrary subset $S\subseteq[p]$.
Our goal is to study under what conditions on the true parameters, the alternative $(\bar\G, \bar\btheta, \bar\nnu)$ must be identical to the true $(\G, \btheta, \nnu)$.
By Lemma \ref{lem-poly}, for arbitrary $\{\bo\Delta_j\}$, we have
\begin{align}\notag
    &\Big(\bigodot_{j\in S} \bo\Phi^{(j)} -\bo\Delta_j\bcdot\one^\top_{2^K} \Big) \cdot \nnu 
    = \mathbf B\left(\{\bo\Delta_j:\,{j\in S}\}\right) \bcdot \Big(\bigodot_{j\in S} \bo\Phi^{(j)}\Big) \cdot \nnu 
    \\ \label{eq-trans}
    &= \mathbf B\left(\{\bo\Delta_j:\,{j\in S}\}\right) \bcdot \Big(\bigodot_{j\in S} \overline{\bo\Phi}^{(j)}\Big) \cdot \overline\nnu
    = \Big(\bigodot_{j\in S} \overline{\bo\Phi}^{(j)} -\bo\Delta_j\bcdot\one^\top_{2^K} \Big) \cdot \overline{\nnu}.
\end{align}
We next give a high-level idea of our proof procedure.
Eq.~\eqref{eq-trans} will be frequently invoked for various subsets $S\subseteq[p]$ when deriving the identifiability results.
For example, suppose we want to investigate whether a specific parameter $\theta^{(j)}_{c\mid\aaa_\ell}$ is identifiable under certain conditions.
Exploiting the fact that $\overline\G$ induces many equality constraints on the entries of $\overline{\bo\Phi}^{(j)}$, 
we will construct a set of vectors $\{\bo\Delta_j; j\in S\}$, which usually has the particular $\bar\theta^{(j)}_{c\mid\aaa_\ell}$ as an entry. These vectors $\{\bo\Delta_j; j\in S\}$ are purposefully constructed so that we can use \eqref{eq-trans} and obtain its right hand side equals zero for some polynomial equation.
This implies a polynomial involving parameters $(\G, \btheta, \nnu)$ and the constructed vectors $\{\bo\Delta_j; j\in S\}$ is equal to zero.
We will then carefully inspect under what conditions this equation implies that $\theta^{(j)}_{c\mid\aaa_\ell}$ is identifiable; 
namely, inspect whether $\theta^{(j)}_{c\mid\aaa_\ell} = \overline{\theta}^{(j)}_{c\mid\aaa_\ell}$ holds under the considered conditions.



Essentially, our proof technique exploits the following two key model properties.
\emph{First}, observed variables are conditionally independent given the (potentially multiple) latent variables. This property makes it possible to write the joint distribution of the observed variables as the product of two parts: one being the Khatri-Rao product (i.e., column-wise Kronecker product) of multiple conditional probability tables, and the other being the vector of the probability mass function of latent variables.
\emph{Second}, graphical structures exist between the latent and observed variables. Such graphs can induce many equality constraints on the conditional probability table $\bo\Phi^{(j)}$ of an observed variable given the latent.
The first property above about conditional independence is a prevailing assumption adopted in many other latent variable models. 
The second property above about graph-induced constraints also frequently appear in directed and undirected graphical models \citep{lauritzen1996graphical}.
Therefore, our technique may be useful to find identifiability conditions for other discrete models with multidimensional latent and graphical structures, e.g., discrete Bayesian networks with latent variables with application to causal inference \citep{allman2015dag, mealli2016causal} and 
mixed membership models \citep{erosheva2007aoas}.

{In our proofs of the identifiability results, the number of latent variables $K$ is assumed to be known.
To the author's best knowledge, in all previous studies that leveraged Kruskal's Theorem to establish identifiability, the number of latent variables has always been assumed as known. Compared to  Kruskal's Theorem, our proof technique provides a closer look into the identifiability of individual parameters under graphical constraints. But we still need to assume that the number of parameters is fixed when investigating the solutions to the polynomial equations \eqref{eq-trans}. We expect that to identify $K$, new approaches that look beyond the polynomial equation systems will be needed.
We leave the interesting and nontrivial question of identifying $K$ as a future research direction.}

\subsection{Discussing connections to and differences from related works}



It is worth connecting the BLESS model to discrete Latent Tree Models \citep[LTMs;][]{choi2011learning, mourad2013survey}, which are popular tools in machine learning and have applications in phylogenetics in evolutionary biology. 
Deep results about the geometry and statistical properties of LTMs are uncovered in \cite{zwiernik2012tree}, \cite{zwiernik2016semialg}, and \cite{shiers2016gltm}.
%
Conceptually, the BLESS model is more general than LTMs because in the former, the latent variables can have arbitrary dependencies according to the definition in Eq.~\eqref{eq-model}, \emph{including but not limited to} the case of a latent tree. 
In this sense, directly studying the identifiability and geometry of the BLESS model are more involved than LTMs.
Geometry and identifiability of Bayesian networks with latent variables have also been investigated in
\cite{settimi2000geometry} and \cite{allman2015dag}. 
But these above works often either consider a small number of variables, or employ certain specific assumptions on the dependence of latent variables.
In contrast, {our results imply that 
various possible models for the latent variables can be considered, and our current conditions on $\G$ remain sufficient for identifying the latent variables' probability mass function $\nnu$. In such cases, whether those parameters underlying $\nnu$ in the more specialized model are identifiable can then be studied by assuming $\nnu$ is already identified and known.}

Another interesting work is \cite{stanghellini2013id} that studied the identifiability of discrete undirected graphical models with one latent binary variable. 
\cite{stanghellini2013id}'s conditions are also related to the graphical structure, and they also provide explicit expressions for the non-identifiable subsets of measure zero. 
One key difference between \cite{stanghellini2013id} and this work is that the authors of the former considered local identifiability, {whereas this work studies strict identifiability and generic identifiability, both concerning the entire parameter space instead of a local neighborhood of the parameters and hence are more ``global'' than the notion of local identifiability}. In addition, we establish identifiability for an arbitrary number of binary latent variables instead of only one binary latent variable. \cite{stanghellini2013id}'s approach has the very nice ability to handle the conditional dependence case between the observed variables given the latent ones. Extending our technique to this scenario would be an interesting yet nontrivial future direction.


%
A generic identifiability statement related to our work appeared in \cite{gu2021idq} in the form of a small toy example for the cognitive diagnostic models mentioned earlier. More specifically, these are models where test items are designed to measure the presence/absence of multiple latent skills and binary item responses of correct/wrong answers are observed.
In the special case with two binary latent skills each measured by two binary observed variables, \cite{gu2021idq} proved the parameters are identifiable if and only if the two latent variables are not independent.
In this work, we investigate the fully general case of the BLESS model where there are (a) an arbitrary number of binary latent variables, (b) arbitrary dependence between these variables, and (c) the observed variables have an arbitrary number of categories. In this general setting, we characterize a complete picture of the generic identifiability phenomenon with respect to the latent dependence in Section \ref{sec-mainsub}.


\section{{Extensions to more complicated models}}
\label{sec-extend}

\subsection{Extension to the BLESS model with higher-order latent structures}\label{subsec-pyramid}
Studying the BLESS model provides useful theoretical insight, but admittedly, having to estimate an unrestricted distribution with $2^K-1$ parameters in $\nnu$ for $K$ binary latent variables would require too much data.
Fortunately, our technique and theory can be readily extended to more flexible models for the latent part -- for instance, when the latent variables follow a more parsimonious distribution induced by deeper latent structures.
In this subsection, we provide an illustrative example of such an extension.
Consider a two-latent-layer Bayesian Pyramid model proposed by \cite{gu2023bp}, which is a Bayesian network with two discrete latent layers; see Figure \ref{fig-2layer}. The shallower latent layer consists of binary latent variables $\bo a$ just as in our BLESS model, while the deeper latent layer only contains one discrete latent class variable $z \in [B]$. 
In this model, the vector $\bo a$ follows a classical latent class model \citep{goodman1974} with $B$ latent classes with the following parametrization:
\begin{align*}
    \mathbb P(\bo a = \aaa) &= 
    \sum_{b=1}^B \mathbb P(z = b) \prod_{k=1}^K \mathbb P(a_{k}=\alpha_k\mid z=b)
    =
    \sum_{b=1}^B \tau_b \prod_{k=1}^{K} \eta_{k,b}^{\alpha_k} (1-\eta_{k,b})^{1-\alpha_k},\quad 
    \forall\aaa\in\{0,1\}^{K};
\\
    \mathbb P(\bo y = \bo c) &=  \sum_{\aaa\in\{0,1\}^{K}} \mathbb P(\bo a = \aaa) \prod_{j=1}^p \mathbb P(y_{j}=c_j \mid \bo a = \aaa,\; \mathbf G),\quad \forall \bo c\in \times_{j=1}^p [d].
\end{align*}
\cite{gu2023bp} used an argument similar to \cite{allman2009} to establish identifiability of the above Bayesian Pyramid. Their sufficient condition for generic identifiability requires each binary latent $a_k$ to have \emph{at least three} pure children and that $K\geq 2\ceil{\log_2(B)} + 1$. 
In contrast, using our new technique, we are able to obtain a (much) weaker identifiability condition -- each binary latent $a_k$ only needs to have \emph{two pure children} because of the blessing of dependence between $a_1,\ldots,a_K$ implied by the deeper latent $z$. The following Proposition \ref{prop-bp} formalizes this statement.

\begin{figure}[h!]\centering
\resizebox{0.63\textwidth}{!}{
    \begin{tikzpicture}[scale=1.8]

    \node (v1)[neuron] at (0, 0) {$y_{1}$};
    \node (v2)[neuron] at (0.8, 0) {$y_2$};
    \node (v3)[neuron] at (1.6, 0) {$\cdots$};
    \node (v4)[neuron] at (2.4, 0) {$\cdots$};
    \node (v5)[neuron] at (3.2, 0) {$y_{2K-1}$};
    \node (v6)[neuron] at (4, 0)   {$y_{2K}$};

    \node (h1)[hidden] at (1.0, 1.2) {$a_{1}$};
    \node (h2)[hidden] at (2.0, 1.2) {$\cdots$};
    \node (h3)[hidden] at (3.0, 1.2) {$a_{K}$};
    
    \node (h0)[hidden] at (2, 2.4) {$z$};

    \node[anchor=west] (z) at (4.8, 2.4) {$z\in [B]$};
    
    \node[anchor=west] (eta) at (4.8, 1.8) {$\bo\eta = (\eta_{k,b}) $};
    
    \node[anchor=west] (h) at (4.8, 1.2) {$\bo a \in \{0, 1\}^{K}$};
   
    \node[anchor=west] (g1) at (4.8, 0.6) {$\G = (g_{j,k})$};
    
    \node[anchor=west] (v) at (4.8, 0) {$\bo y \in [d]^{2K}$};

    \draw[qedge] (h0) -- (h1) node [midway,above=-0.12cm,sloped] {\textcolor{black}{$\eta_{1,b}$}}; 
    \draw[qedge] (h0) -- (h2);
    \draw[qedge] (h0) -- (h3) node [midway,above=-0.12cm,sloped] {\textcolor{black}{$\eta_{K_1,b}$}};

    \draw[qedge] (h1) -- (v1) node [midway,above=-0.12cm,sloped] {}; 
    
    \draw[qedge] (h1) -- (v2) node [midway,above=-0.12cm,sloped] {};  
    
    \draw[qedge] (h2) -- (v3) node [midway,above=-0.12cm,sloped] {}; 
    
    \draw[qedge] (h2) -- (v4) node [midway,above=-0.12cm,sloped] {}; 
    
    
    
    
    \draw[qedge] (h3) -- (v5) node [midway,above=-0.12cm,sloped] {}; 

    \draw[qedge] (h3) -- (v6) node [midway,above=-0.12cm,sloped] {}; 
    
\end{tikzpicture}
}
\caption{Two-latent-layer Bayesian Pyramid model in \cite{gu2023bp}. Here the $\bo a$-layer-to-$\bo y$-layer measurement graph is a star tree, where each $a_k$ has exactly two children $y_{2k-1}$ and $y_{2k}$. 
}
\label{fig-2layer}
\end{figure}
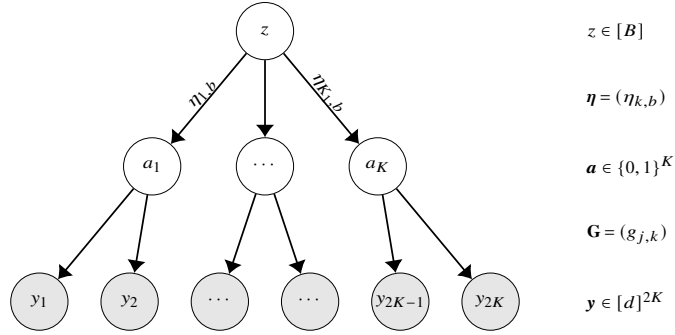

\begin{proposition}\label{prop-bp}
Consider the two-latent-layer model in Figure \ref{fig-2layer} where $\bo y\mid \aaa$ follows a star-forest graphical model and $\aaa\mid z$ follows a classical latent class model.
If each binary latent variable has two pure children and that $K \geq 2 \ceil{\log_2(B)} + 1$, then the model parameters are generically identifiable.
\end{proposition}

Proposition \ref{prop-bp} can be proved as a corollary of our main result.
Thanks to the existence of the deeper latent class variable underlying the binary latent variables $\bo a$, the following inequality holds generically for any vector $(\alpha_k,\alpha_m)\in\{0,1\}^2$:
\begin{align*}
&~\mathbb P(a_{k} = \alpha_k,\; a_{m} = \alpha_m) - \mathbb P(a_{k} = \alpha_k) \mathbb P(a_{m} = \alpha_m)\\
=&~ \sum_{b=1}^B \tau_b \eta_{k,b}^{\alpha_k} (1-\eta_{k,b})^{1-\alpha_k} \eta_{m,b}^{\alpha_m} (1-\eta_{m,b})^{1-\alpha_m} - \prod_{\ell\in\{k,m\}} \left[\sum_{b=1}^B \tau_b \eta_{\ell,b}^{\alpha_\ell} (1-\eta_{\ell,b})^{1-\alpha_\ell}\right] \neq 0.
\end{align*}
This inequality means for generic model parameters in the two-latent-layer Bayesian Pyramid, {$a_k$ is not independent of $a_m$}, hence allowing for the blessing-of-dependence to kick in to deliver identifiability.
Combining this observation with the proof of Proposition 3 in \cite{gu2023bp} that shows generic identifiability of $(\eta_{k,b})$ and $(\tau_b)$ under $K\geq 2\ceil{\log_2(B)}+1$, we obtain the much weaker identifiability condition in Proposition \ref{prop-bp}.

\subsection{Extension to a model with an arbitrary measurement graph $\G$}
\label{subsec-catdina}

In this subsection, we pursue a more challenging extension by studying a more complicated model in which $\G$ can be an arbitrary binary matrix. In other words, in this model each observed variable is not restricted to having only one latent parent as in the BLESS model.
Next, we first formally define this model, and then prove generic identifiability and reveal the blessing-of-dependence for it.
We need to introduce some additional notation.
For two vectors $\bo a=(a_1,\ldots,a_L)$ and $\bo b=(b_1,\ldots,b_L)$ of the same length $L$, we write $\bo a\succeq \bo b$ if $a_\ell\geq b_\ell$ for all $\ell\in[L]$; that is, when vector $\bo a$ is elementwisely greater than or equal to vector $\bo b$.
If $a_\ell<b_\ell$ holds for some $\ell\in[L]$, then we write $\bo a\nsucceq \bo b$.

We consider an extension of a popular psychometric model -- the so-called Deterministic Input Noisy output ``And'' gate model \citep[DINA model;][]{junker2001dina} motivated by educational cognitive diagnosis. 
The DINA model is usually used for modeling multivariate binary responses in an educational test setting. In this setting, each subject is a student test taker with the binary observed variables $\bo y$ denoting the student's correct or wrong responses to $p$ test questions, and the binary latent variables $\bo a$ encoding the student's profile of the presence or absence of $K$ skills. 
The DINA model is associated with a so-called $\Q$-matrix \citep{tatsuoka1983} that describes which skills are required/measured by each test question. Essentially, this $\Q$-matrix is equivalent to the measurement graph matrix $\G$ in our notation.
The DINA model does not restrict each test question to depend on only one latent skill, which means $\G$ can be an arbitrary binary matrix.
For $j\in[p]$, recall that $\bo g_j = (g_{j,1}, \ldots, g_{j,K}) \in\{0,1\}^K$ denotes the $j$th row vector of matrix $\G$ and it describes which skills are required by question $j$, with $g_{j,k}=1$ if skill $k$ is required and $g_{j,k}=0$ if not. 
If a student's latent skill profile $\bo a$ satisfies $\bo a \succeq \bo g_j$, then the student masters all required skills of question $j$; if $\bo a \nsucceq \bo g_j$, then the student lacks some required skills of it.
In the binary-response DINA model, the probability of providing a correct response to question $j$ for a student with latent skill profile $\bo a$ is:
\begin{align}\label{eq-bindina}
    \mathbb P^{\text{BinaryDINA}}(y_j = 1\mid \bo a) = 
    \begin{cases}
        1 - s_j, & \text{ if }\bo a\succeq \bo g_j;\\
        u_j, & \text{ if }\bo a\nsucceq \bo g_j,
    \end{cases}
\end{align}
where $s_j$ and $u_j$ have the following interpretation. Parameter $s_j = 1-\mathbb P(y_j = 1\mid \bo a \succeq \bo g_j)$ represents the probability of slipping the correct answer of question $j$ despite that the student possesses all the required skills of it (sometimes called ``capable'' of question $j$). Parameter $u_j = \mathbb P(y_j = 1\mid \bo a \nsucceq \bo g_j)$ represents the probability of correctly guessing the answer despite that the student lacks some of the required skills (``incapable'' of question $j$). Many previous studies assumed that $1-s_j> u_j$ \citep[e.g.,][]{culpepper2015dina, gu2019dina}, meaning that capable students of a question has a higher probability of answering it correctly than incapable students.

We can extend the binary-response DINA model to the case of general categorical responses, to be consistent with the response type in the BLESS model in Section \ref{sec-setup}.
Next, we formally define the categorical-response DINA model, abbreviated as CatDINA, where each observed variable $y_j$ ranges in $d$ categories for some integer $d\geq 2$.
Such an extended model could be used to model partial credits in educational tests.
For $j\in[p]$ and $c_j\in[d]$, define the conditional response probability as:
\begin{align}\label{eq-catdina}
    \mathbb P^{\text{CatDINA}}(y_j = c_j\mid \bo a) = 
    \begin{cases}
        \theta^{(j)}_{c_j\mid 1}, & \text{ if }\bo a\succeq \bo g_j;\\[3mm]
        \theta^{(j)}_{c_j\mid 0}, & \text{ if }\bo a\nsucceq \bo g_j.
    \end{cases}
\end{align}
The CatDINA model has the same number of $\theta$-parameters as the BLESS model defined in Section \ref{sec-setup}, but it allows the matrix $\G$ to take an arbitrary form rather than having only standard basis row vectors. 
The CatDINA model (and the original binary-response DINA model) assumes a \emph{conjunctive} relationship of latent variables, by grouping the latent patterns $\bo a \in \{0,1\}^K$ into two classes for each $j$: the capable class ($\bo a\succeq\bo g_j$) and the incapable class ($\bo a\nsucceq\bo g_j$). Therefore, fixing some $j\in[p]$ and $c_j\in[d]$, as defined in \eqref{eq-catdina}, the conditional response probabilities can only take two different values depending on whether $\bo a \succeq \bo g_j$.

For the binary-response DINA model in \eqref{eq-bindina},
\cite{gu2019dina} proved that the following three conditions (C), (R), and (D) are necessary and sufficient for strict identifiability when $\G$ is known:
\begin{itemize} 
\item[(C)] {\textbf{C}ompleteness}.
A $\G$-matrix with $K$ columns contains an identity submatrix $\I_K$ after some row permutation. Namely, the $\G$ can be row-permuted to take the form of
$\G=(\I_K; \G^{*\top})^\top$.

\item[(R)] {\textbf{R}epeated-Measurement.}
Each column of $\G$ contains at least three entries of ``1''s.

\item[(D)] {\textbf{D}istinctness.}
Assuming Condition (C) holds, after removing the identity submatrix $\I_K$ from $\G$, the remaining $(p-K)\times K$ submatrix $\G^{*}$ has $K$ mutually different column vectors.
\end{itemize}

We call the above three conditions the C-R-D conditions for short. For example, one can directly verify that the following $6\times 3$ matrix satisfies the C-R-D conditions:
\begin{equation}\label{eq-g63}
\G=\begin{pmatrix}
    1 & 0 & 0 \\
    0 & 1 & 0 \\
    0 & 0 & 1 \\
    \hline
    1 & 1 & 0 \\
    1 & 0 & 1 \\
    0 & 1 & 1
\end{pmatrix}.
\end{equation}
The C-R-D conditions can also be rephrased in graphical language as follows:
\begin{itemize} 
\item[(C)] 
Each latent variable has at least one observed variable as a ``pure child'', which has exactly one latent variable as its parent.

\noindent (Another equivalent way of stating condition (C) in graph theory terminology is: the bipartite graph has a ``perfect matching'' between the latent and the observed layer.)

\item[(R)] 
Each latent variable has at least three observed variables as children (not necessarily all pure children). 

\item[(D)] 
Assuming Condition (C) holds, after removing the $K$ edges in the perfect matching from the bipartite graph, the remaining graph satisfies that the $K$ latent variables' sets of children variables are mutually distinct.
\end{itemize}

It is not hard to see the equivalence and one-to-one correspondence between the above graphical-language C-R-D conditions and the previous algebraic-language C-R-D conditions.  As a concrete example, we can still consider the $6\times 3$ matrix $\G$ in \eqref{eq-g63}.
The corresponding children sets of the three latent variables are: $\ch(a_1) = \{y_1, y_4, y_5\}$, $\ch(a_2) = \{y_2, y_4, y_6\}$, and $\ch(a_3) = \{y_3, y_5, y_6\}$. In this case, condition (C) is satisfied because there exists a perfect matching with these three edges: $a_1 \to y_1$, $a_2 \to y_2$, and $a_3 \to y_3$; condition (R) is satisfied because $|\ch(a_1)| \geq 3$, $|\ch(a_2)| \geq 3$, and $|\ch(a_3)| \geq 3$; condition (D) is also satisfied because after removing those three edges in the perfect matching in condition (C), in the remaining graph, the children sets of $a_1, a_2,$ and $a_3$ are $\{y_4, y_5\}$, $\{y_4, y_6\}$, and $\{y_5, y_6\}$ respectively, which are three mutually distinct sets.

Next, we first prove that the C-R-D conditions are still sufficient for strict identifiability of the CatDINA model, and then further relax these conditions to establish generic identifiability and reveal a  blessing-of-dependence phenomenon under the CatDINA model.

\begin{proposition}[Strict identifiability of the CatDINA model]\label{thm-cdina-str}
Consider the CatDINA model with parameters $(\bo\theta,\nnu)$ satisfying assumptions \eqref{eq-nu-positive} and \eqref{eq-flip} (the same assumptions as the BLESS model defined in Section \ref{sec-setup}). Assume the $\G$ matrix is known.
Then the C-R-D conditions are sufficient for strict identifiability of parameters in the CatDINA model.
\end{proposition}

\begin{theorem}[Generic identifiability and blessing of dependence in the CatDINA model]\label{thm-cdina-gen}
\label{thm-rep} 
Consider the CatDINA model with parameters $(\bo\theta,\nnu)$ satisfying the same assumptions \eqref{eq-nu-positive} and \eqref{eq-flip} as the BLESS model.
Suppose the $\G$ matrix satisfies Condition (C) but does not satisfy Condition (R) in that $\sum_{j=1}^p g_{j,k}=2$ for some $k\in[K]$. 
In this case, the $\G$ matrix can be written in the following form after some column/row permutation, where $\G^*$ is a $(p-2)\times(K-1)$ submatrix and $\bo u$ is a $(K-1)\times 1$ vector.
\begin{align}\label{eq-q1u}
\G = 
\begin{pmatrix}
1 & \zero \\
1 & \bo u \\
\hline
\zero & \G^*
\end{pmatrix}
\end{align}
\begin{itemize}
    \item[(a)] If the submatrix $\G^*$ satisfies the C-R-D conditions and $\bo u\neq \one_{K-1}^\top$, then the parameters $(\bo\theta,\nnu)$ in the CatDINA model are generically identifiable. 
    \item[(b)] Under the condition in part (a), the measure-zero  non-identifiable set $\mathcal N$ in the parameter space is characterized by
\begin{align}\notag
\mathcal N= &~\{
\nnu\text{ satisfies }
    \nu_{(1,\aaa_1^*)} \nu_{(0,\aaa_2^*)} - \nu_{(0,\aaa_1^*)}  \nu_{(1,\aaa_2^*)} = 0
    ~~\forall 
    \aaa_1^*,\, \aaa_2^* \succeq\uu.\}\\ 
\label{eq-nset-r1}
=&~
\{\nnu\text{ satisfies } (a_1 \indep \bo a_{2:K} \mid \bo a_{2:K}\succeq\uu).\},
\end{align}
where ``$a_1 \indep \bo a_{2:K} \mid \bo a_{2:K}\succeq\uu$'' reads as: latent variables $a_1$ and $\bo a_{2:K}$ are conditionally independent given that $\bo a_{2:K}\succeq\uu$.
\end{itemize}
\end{theorem}

Theorem \ref{thm-cdina-gen} establishes generic identifiability of the CatDINA model by considering a particular violation of the strict identifiability conditions: some latent variable has only two observed children instead of three ones.
Such a consideration is inspired by the identifiability conclusions for the BLESS model in Theorem \ref{thm-main} and Proposition \ref{prop-3chi}, because having two or three children per latent variable is exactly the difference between generic and strict  identifiability under the BLESS model.
The proof of Theorem \ref{thm-cdina-gen} is more nuanced than Theorem \ref{thm-main}, because the CatDINA model has more flexible parent-child relationships between the latent and observed variables than the BLESS model.

Theorem \ref{thm-cdina-gen}(b) shows that the non-identifiable set $\mathcal N$ is characterized by the zero-set of certain polynomials only involving the parameters $\nnu=(\nu_{\aaa}:\;\aaa\in\{0,1\}^K)$ but not the $\bo\theta$-parameters. 
In the proof of Theorem \ref{thm-cdina-gen}, we first show that if the true 
$\nnu$-parameters do not satisfy $\nu_{(1,\aaa_1^*)} \nu_{(0,\aaa_2^*)} - \nu_{(0,\aaa_1^*)}  \nu_{(1,\aaa_2^*)} = 0$ for all $\aaa_1^*\neq \aaa_2^*$ with $\aaa_1^*,\, \aaa_2^* \succeq\uu$, then both $\nnu$ and $\bo\theta$ are identifiable.
Then based on such defining polynomial equations of the non-identifiable set $\mathcal N$, we further derive its equivalent interpretation of conditional independence ``$a_1 \indep \bo a_{2:K} \mid \bo a_{2:K}\succeq\uu$'' (see the proof of Theorem \ref{thm-cdina-gen} for details).

{Proving Theorem \ref{thm-cdina-gen} for the CatDINA model with an arbitrary measurement graph is {technically nontrivial and more difficult} than proving the result for the BLESS model.
It is worth emphasizing that our high-level proof technique is not restricted to models in which the conditional response
probabilities $\mathbb P(y_j=c_j\mid \bo a)$ only take two different values as in the BLESS model and the CatDINA model in \eqref{eq-catdina}. In fact, as explained earlier in Section \ref{sec-overview}, this technique essentially exploits the following two properties of a model to show identifiability:
the conditional independence of the observed variables given the latent ones, and the
graphical structure between the observed and latent variables. None of these properties rely on the assumption that the conditional response probabilities can only take two different values.
Therefore, we believe the following would be a fair statement: 
The high-level proof technique could be applied to other models based on our technical insight into its fundamental characteristics, but such extensions will not be straightforward and will indeed take significant technical effort for any specific model --
just as extending the result from the BLESS model to the CatDINA model does. Extensions of this kind to more complex models may be worth pursuing in the future.}

\section{Statistical hypothesis test of identifiability in the boundary case}
\label{sec-test}
Consider the minimal conditions for generic identifiability of the BLESS model, where certain  latent variables have only two children. 
In this case, the blessing of dependence provides a basis for performing a statistical hypothesis test of identifiability.
We have the following proposition.

\begin{proposition}\label{prop-depdep}
Under the BLESS model defined in \eqref{eq-model}, consider two different latent variables $\alpha_{k_1}$ and $\alpha_{k_2}$. 
The two groups of observed variables 
$\{y_j=c_j:\; g_{j,k_1}=1\}$ and  $\{y_m=c_m:\; g_{m,k_2}=1\}$ 
are independent if and only if $a_{k_1}$ and $a_{k_2}$ are independent.
\end{proposition}

Proposition \ref{prop-depdep} states that under the BLESS model, the dependence/independence of latent variables is exactly reflected in the dependence/independence of their observed proxies (i.e., observed children variables). This fact is apparent from the graphical representation of the BLESS model in Figure \ref{fig-graph}.
A nice implication of Theorem \ref{thm-main} and Proposition \ref{prop-depdep} is that,
we can test the marginal dependence between certain observed variables to determine model identifiability, before even trying to fit a potentially unidentifiable model to data.

Formally,  under minimal conditions for generic identifiability where some latent variable $\alpha_k$ only has two observed children, if one wishes to test the following hypothesis
$$
H_{0k}:~ \text{Parameters associated with }~ \ch(a_k)\mid a_k ~\text{ are not identifiable},
$$
then it is equivalent to testing the hypothesis $H_{0k}':~ a_k \indep \bo a_{-k}$. 
Further, to test $H_{0k}'$ it suffices to test the marginal independence between the following observed variables,
$$
H_{0k}':~ \ch(a_k) \indep \ch(\bo a_{-k}).
$$
Since $\ch(a_k)$ and $\ch(\bo a_{-k})$ are fully observed given the measurement graph, the above hypothesis $H_{0k}'$ can be easily tested.
Note that $\ch(a_k)$  can be regarded as a categorical variable with $d^{|\ch(a_k)|}$ categories and that $\ch(\bo a_{-k})$  can be regarded as another categorical variable with $d^{|\ch(\bo a_{-k})|}$ categories. So the simple $\chi^2$ test of independence between two categorical variables can be employed for testing $H_{0k}'$.
If the null hypothesis of independence is not rejected, then caution is needed in applying the BLESS model because some parameters may not be identifiable.
If, however, the hypothesis of independence is rejected, then this is statistical evidence supporting the identifiability of the BLESS model. In this case one can go on to fit the model to data, interpret the estimated parameters, and conduct further statistical analysis.

Since our hypothesis test of identifiability can be performed without fitting the BLESS model, it can serve as a first-step sanity check in real data analysis.
In a similar spirit but for a different purpose when studying the Gaussian Latent Tree Models, \cite{shiers2016gltm} proposed to test certain covariance structures of variables to determine the goodness of fit before fitting the model to data. 
To the author's best knowledge, there has not been previous formal approaches to directly {testing the identifiability} of multidimensional latent variable models. 
Our test is enabled by the discovery of the blessing of dependence phenomenon and may inspire future relevant hypothesis testing approaches in other latent variable models.

\section{A real-world example of hypothesis testing of identifiability}\label{sec-prac}

We present a real-world example in educational assessments.
The Trends in International Mathematics and Science Study (TIMSS) is a series of international assessments of the mathematics and science knowledge of fourth and eighth grade students.
TIMSS has been held every four years since 1995 in over 50 countries.
Researchers have used the cognitive diagnostic model to analyze the Austrian TIMSS 2011 data \citep{george2015cdm}, which are available in the R package \texttt{CDM}.
The dataset involves fourth grade students' correct/wrong responses to a set of TIMSS questions in mathematics.
According to educational experts, these questions were designed to measure the presence/absence statuses of $K=3$ latent skills of students: ($a_1$) Data,  ($a_2$) Geometry, and  ($a_3$) Numbers. 
Each question targets exactly one skill, which means the latent-to-observed measurement graph satisfies the assumption of the BLESS model.
In this Austrian TIMSS dataset, we focus on the first booklet containing the first $p = 21$ questions, and consider the $N=341$ students who answered all these questions. 
Table \ref{tab-timss} summarizes how these 21 questions depend on the three latent skills, i.e., what the $\G$ matrix is.

\begin{table}[h!]
    \centering
    \caption{TIMSS example. Latent-to-observed measurement graph (i.e., $\G$ matrix) between the first $p=21$ questions and $K=3$ latent skills, obtained from the R package \texttt{CDM}.}
    \label{tab-timss}

    \resizebox{0.61\linewidth}{!}{
    \begin{tabular}{lll}
    \toprule
    & Latent skill  &  Indices of questions that measure the skill\\
    \midrule
      $a_1$ & Data  & 20, 21 \\
      $a_2$ & Geometry  & 7, 8, 16, 17, 18, 19 \\
      $a_3$ & Numbers  & 1, 2, 3, 4, 5, 6, 9, 10, 11, 12, 13, 14, 15 \\
    \bottomrule
    \end{tabular}
    }
    
\end{table}

Table \ref{tab-timss} shows that the first skill ``Data'' is measured by only two questions (questions 20 and 21), hence satisfying the minimal conditions for generic identifiability.
So according to our new results, whether the model parameters are identifiable would depend on whether there exists underlying dependence between $a_1$ and $(a_2, a_3)$.
We carry out a hypothesis test of identifiability of the BLESS model. 
In particular, consider the null hypothesis
\begin{align*}
H_{0,\text{Data}}:~ &\text{Skill ``Data'' is independent of skills ``Geometry'' and ``Numbers''};
\end{align*}
based on the $\G$ matrix structure in Table \ref{tab-timss}, we can test whether the questions measuring the ``Data'' skill are independent with  those measuring the other two skills.
In particular, here we consider all the two-question-combinations consisting of one  measuring ``Geometry'' and one measuring ``Numbers'', and then test whether this combination of questions are independent of those two ``Data'' questions; namely, we test
\begin{align*}
H_{0,\text{Data}}^{j_1, j_2}:~ (y_{20}, y_{21}) \text{~are independent of~} (y_{j_1}, y_{j_2}),~
j_1 \text{~measures Geometry,~} j_2 \text{~measures Numbers}.
\end{align*}
Using the standard $\chi^2$ test of independence between two categorical variables each with $2^{2}=4$ categories, each test statistic under the null hypothesis $H_{0,\text{Data}}^{j_1, j_2}$ asymptotically follows the $\chi^2$ distribution with $df = (2^2 - 1) \cdot (2^2 - 1) = 9$ degrees of freedom. 
Out of the $6\times 13 = 78$ such test statistics, we found 73 of them are greater than the 95\% quantile of the reference distribution $\chi^2(df, 0.95) = 16.92$, where we reject the null hypothesis of independence between $(y_{20}, y_{21})$ and $(y_{j_1}, y_{j_2})$. We point out that the rejection of any of these tests $H_{0,\text{Data}}^{j_1, j_2}$ already indicates one should reject the original null $H_{0,\text{Data}}$.  
Thanks to the blessing of dependence theory, the test results provide statistical evidence to reject the original null hypothesis of non-identifiability, and hence support the identifiability of model parameters.
This provides a statistical conclusion of identifiability for the first time in such applications. 
We also provide another example about a social science survey in the Supplementary Material.

\section{Discussion}\label{sec-disc}
This work reveals a blessing-of-latent-dependence geometry for the BLESS model and its extensions, which are discrete models with multiple binary latent variables.
For the BLESS model, we show that under the minimal conditions for generic identifiability that each latent variable has exactly two observed children, the model parameters are identifiable if and only if there exists dependence between the latent variables.
{In addition, we have successfully established similar conclusions for the more complicated CatDINA model, which has a more flexible measurement graph beyond a star tree.}
In statistical modeling, the independence assumption on latent variables is predominantly adopted; e.g., in traditional factor analysis, latent factors are often assumed to be independent with a diagonal covariance matrix \citep{anderson1956fa}.
In practice, however, especially in confirmatory latent variable analysis widely seen in education, psychology, and epidemiology, latent constructs of interest often carry substantive meanings; see the real-data example in Section \ref{sec-prac}. 
As a result, it is highly likely that such latent constructs postulated by domain experts are dependent on each other, such as the presence/absence of depression and anxiety disorders in psychiatry, or the existence/non-existence of multiple pathogens in epidemiology.
From this perspective, our theoretical result provides reassurance that the dependence of latent variables can be a blessing, rather than a curse.

{We have demonstrated in Section \ref{subsec-catdina} that our proof technique can be used to study a general measurement graph between the categorical observed variables and binary latent variables.
But we find it not straightforward to extend the proof technique to models in which the latent variables are polytomous; i.e., categorical latent variable with more than two categories.
The reason is that the algebraic characterization of independence between binary variables is 
much more manageable than that for polytomous variables.
Specifically, the statement that $a_1\in\{0,1\}$ is independent with $\bo a_{2:K}\in\{0,1\}^{K-1}$ is equivalent to that the $2\times 2^{K-1}$ joint probability table of $a_1$ and $\bo a_{2:K}$ has rank one. This rank-one constraint is further equivalent to the simultaneous vanishing of ${2^{K-1}\choose 2}$ degree-2 homogeneous polynomials of the proportion parameters $\nu_{\aaa}$ for $\aaa\in\{0,1\}^K$ (see the proofs of Theorems \ref{thm-main} and \ref{thm-cdina-gen} for details). In our proof of the blessing of dependence, we are able to algebraically characterize the measure-zero non-identifiable set $\mathcal N$, and further reveal that $\mathcal N$ exactly corresponds to the zero set of the aforementioned degree-2 homogeneous polynomials. 
However, for polytomous variables $a_1$ and $\bo a_{2:K}$ each with $C>2$ categories, the independence between $a_1$ and $\bo a_{2:K}$ corresponds to the vanishing of $2\times 2$ sub-determinants of a much larger $C \times C^{K-1}$ joint probability table, which involves many more polynomial equations. As a result, it is more difficult in this case to examine the relationship between such polynomials and the non-identifiable set, and even difficult to characterize the non-identifiable set itself. 
On a related note, \cite{zwiernik2012tree} made a similar remark when studying the identifiability of latent tree models (LTMs), which could be viewed as a special case of our considered models. \cite{zwiernik2012tree} characterized the measure-zero non-identifiability set under LTMs and pointed out that extending the conclusion beyond the binary latent variable case is difficult.
Nonetheless, we would like to remark that multidimensional binary latent variable models are ubiquitous both in real-world applications (such as various cognitive diagnosis models in psychometrics \citep{rupp2008unique, von2019handbook}) and also in machine learning (such as deep belief networks and deep Boltzmann machines \citep{hinton2006fast, goodfellow2016deep}).}

As a final remark, in a study of the geometry of the simplest discrete latent variable model -- the latent class model with a unidimensional latent variable, and in its special case with only $p=2$ observed variables, \cite{fienberg2009} remarked that ``\textit{The study of higher dimensional tables is still an open area of research. The mathematical machinery required to handle larger dimensions is considerably more complicated}''. 
Indeed, due to the complexity and nonlinearity of discrete models with latent and graphical structures, previous studies about identifiability either cleverly but also directly draw on Kruskal's Theorem 
or focus on a small number of variables.
This work contributes a new technical framework (Lemma \ref{lem-poly} and related explanations in Section \ref{sec-overview}) useful to study the identifiability and geometry of general $p$-dimensional tables, which we hope will be useful more broadly.

\begin{acks}[Acknowledgments]
The author sincerely thanks the editor Prof.  Davy Paindaveine, an anonymous associate editor, and two anonymous reviewers for many constructive and detailed comments that helped to significantly improve this manuscript.
\end{acks}

\begin{funding}
The author Yuqi Gu is supported by NSF Grant DMS-2210796.
\end{funding}

\begin{supplement}
The Supplementary Material \cite{supplement} contains the technical proofs of all theoretical results, details of the algorithms, and an additional real-world example.
\end{supplement}

\bibliographystyle{imsart-nameyear.bst}
\bibliography{ref_new}

\clearpage

\begin{center}
 {\LARGE{Supplement to “Blessing of Dependence: identifiability and geometry of discrete models with
multiple binary latent variables”}}
\end{center}

\bigskip
\bigskip
\bigskip

 \renewcommand{\thesection}{S.\arabic{section}}  
 \renewcommand{\thetable}{S.\arabic{table}}  
 \renewcommand{\thefigure}{S.\arabic{figure}}
 \renewcommand{\theequation}{S.\arabic{equation}}
 \renewcommand{\thetheorem}{S.\arabic{theorem}}
 \renewcommand{\thelemma}{S.\arabic{lemma}}

 \setcounter{section}{0}
 \setcounter{equation}{0}
 \setcounter{figure}{0}

The Supplementary Material is organized as follows. 
Section \ref{sec-proof1} contains the proof of the main result Theorem \ref{thm-main}.
Section \ref{sec-addproofs} contains the proofs of the remaining theoretical results in the paper.
Section \ref{sec-algodetail} contains details of the EM algorithms and
Section \ref{sec-addreal} contains an additional real-world example.

\section{Proof of the main result Theorem \ref{thm-main}}
\label{sec-proof1}

We first define some notation. 
For multiple vectors $\bo a_1,\ldots,\bo a_M$ of the same length $L$ with $\bo a_m = (a_{m1},\ldots,a_{mL})$ for each $m$, define their elementwise maximum to be the vector 
\begin{equation}\label{eq-bigvee}
\bigvee_{m=1}^M \bo a_m = \left(\max_{1\leq m\leq M} a_{m1},~ \ldots,~ \max_{1\leq m\leq M} a_{mL}\right).
\end{equation}
Therefore,  $\bigvee_{m=1}^M \bo a_m \succeq \bo b$ means that the elementwise maximum of the $M$ vectors $\bo a_1, \ldots, \bo a_M$ is elementwisely greater than or equal to the vector $\bo b$.

{We introduce the following useful lemma before proceeding with the proof.
\begin{lemma}\label{lem-neq}
Consider true graphical matrix $\mathbf G$ and associated true parameters $\btheta, \nnu$ that satisfy \eqref{eq-flip},  suppose alternative $\ov\G, \ov\btheta, \ov\nnu$ lead to the same distribution of the observed vector $\bo y$ as the true parameters.
Then $\ov\btheta$ and $\btheta$ must satisfy the following for any $c \in [d]$:
\begin{align*}
    \ov\theta^{(j)}_{c\mid 0} \neq \theta^{(j)}_{c\mid 1}, \quad
    \ov\theta^{(j)}_{c\mid 1} \neq \theta^{(j)}_{c\mid 0}.
\end{align*}
\end{lemma}}

In the following, we prove part (a) and part (b) of the theorem respectively.

\medskip
\noindent
\textbf{Proof of Part (a) of Theorem \ref{thm-main}.}
First note that under the assumptions of the current theorem, we can apply the previous Theorem \ref{thm-graph} to obtain that the matrix $\mathbf G$ is identifiable. So it remains to consider how to identify $(\bo\theta, \nnu)$.
Suppose alternative parameters $(\overline{\bo\theta}, \overline\nnu)$ lead to the same distribution of the observables as the true parameters $(\bo\theta, \nnu)$.

Recall that there are $p=2K$ observed variables under the condition of the theorem.
We first consider an arbitrary index $k\in[K]$ and an arbitrary binary pattern $\aaa'\in\{0,1\}^K$ with $\aaa'_k=0$.
Fixing this $\aaa'$ and fixing some $c\in[d-1]$, we define several $(2K)$-dimensional vectors:
\begin{align}\label{eq-tstar}
    \bo\Delta_{:,c}
    &= \overline\theta_{c\mid 0}^{(k)}\ee_k 
    + \overline\theta_{c\mid 1}^{(K+k)}\ee_{K+k} 
    + \sum_{1\leq m(\neq k)\leq K\atop \alpha'_m=1} \theta_{c\mid 0}^{(m)} {\ee_m}
    + \sum_{1\leq m(\neq k)\leq K\atop \alpha'_m=0} \theta_{c\mid 1}^{(m)} {\ee_m};
    \\ \notag
    \bo\Delta_{:,c'} &= \zero_{2K}\text{ for $c'\neq c$}.
\end{align}
Here we use $\ee_{j}$ to denote a standard basis vector of dimension $2K$ that takes the value of one in the $j$th entry and zero in all the other entries. 
For any $j\in[2K]$, let $\bo\Delta_{j,c}$ denote the $j$th entry of the vector $\bo\Delta_{:,c}$.
The $\bo\Delta_{:,c}$ defined above is a $(2K)$-dimensional vector, with nonzero entries in the first $K$ entries and the $(K+k)$th entry; the other vectors $\bo\Delta_{:,c'}$ for $c'\neq c$ are zero vectors with the same dimension. 
We vertically stack these column vectors $\bo\Delta_{:,1},\ldots,\bo\Delta_{:,d}$ to obtain a $(2K)\times d$ matrix, and we denote the $j$th row of this matrix as $\bo\Delta_{j,:}$, so $\bo\Delta_{j,:}\in \mathbb R^d$.
Recall that $\bo\Phi^{(j)}$ denotes the $d\times 2^K$ conditional probability table for the observed variable $y_j\in[d]$. The rows of $\bo\Phi^{(j)}$ are indexed by the $d$ categories of $y_j$ and the columns indexed by the $|\{0,1\}^K|=2^K$ different binary latent patterns. 
Recall that the Khatri-Rao product of matrices is the column-wise Kronecker product of them, so the following Khatri-Rao product of the $K+1$ matrices $\bo\Phi^{(1)}, \ldots, \bo\Phi^{(K)}, \bo\Phi^{(K+k)}$ is a  $d^{K+1} \times 2^K$ matrix:
$$
\bigodot_{j\in[K]\cup \{K+k\}} \bo\Phi^{(j)}.
$$
Therefore, the particular response pattern $\bo y_c = (c,c,\ldots,c)$ indexes a row in the above Khatri-Rao product matrix and in fact this row vector can be explicitly written as
$\bigodot_{j\in[K]\cup \{K+k\}} \bo\Phi^{(j)}_{c,:}$, which is $2^K$-dimensional vector with entries $\prod_{j\in[K]\cup \{K+k\}} \bo\Phi^{(j)}_{c,\aaa}$ for $\aaa$ ranging in $\{0,1\}^K$.
We next use the proof technique described in Lemma 3.6 in the main manuscript. With the $d$-dimensional vectors $\bo\Delta_{j,:}$ defined earlier in this paragraph, Lemma 3.6 implies that
$$
\bigodot_{j\in[K]\cup \{K+k\}} \Big(\bo\Phi^{(j)} - \bo\Delta_{j,:} \bcdot \one_{2^K}^\top \Big) \bcdot \nnu = 
\bigodot_{j\in[K]\cup \{K+k\}} \Big(\ov{\bo\Phi}^{(j)} - \bo\Delta_{j,:} \bcdot \one_{2^K}^\top \Big) \bcdot \ov\nnu.
$$
Note that the two Khatri-Rao products on both hand sides of the above display both have size $d^{K+1}\times 2^K$, and the two proportion parameter vectors $\nnu$ and $\ov\nnu$ both have size $2^K\times 1$.
Furthermore, the $d^{K+1}$ rows of these Khatri-Rao products are indexed by
all of the different response patterns when the $K+1$ variables $y_1,\ldots, y_K, y_{K+k}$ each ranges in $[d]$. 
Next we specifically focus on the row in $\bigodot_{j\in[K]\cup \{K+k\}} \Big(\bo\Phi^{(j)} - \bo\Delta_{j,:} \bcdot \one_{2^K}^\top \Big)$ and $\bigodot_{j\in[K]\cup \{K+k\}} \Big(\ov{\bo\Phi}^{(j)} - \bo\Delta_{j,:} \bcdot \one_{2^K}^\top \Big)$ indexed by the response pattern with $y_1=\ldots=y_K=y_{K+k}=c$, and for this row the above equation becomes
\begin{align}\notag
\bigodot_{j\in[K]\cup \{K+k\}} \Big(\bo\Phi^{(j)}_{c,:} - \bo\Delta_{j,c} \bcdot \one_{2^K}^\top \Big) \bcdot \nnu 
&= 
\bigodot_{j\in[K]\cup \{K+k\}} \Big(\ov{\bo\Phi}^{(j)}_{c,:} - \bo\Delta_{j,c} \bcdot \one_{2^K}^\top \Big) \bcdot \ov\nnu,
\\
\label{eq-sumprod1}
\Longleftrightarrow \qquad
\sum_{\aaa\in\{0,1\}^K} 
\underbrace{\prod_{j\in[K]\cup \{K+k\}} \Big({\bo\Phi}^{(j)}_{c,\aaa} -\bo\Delta_{j, c} \Big)}_{\text{denoted as: }t_{\bo y_c,\aaa}} \cdot \nu_{\aaa}
&= 
\sum_{\aaa\in\{0,1\}^K} 
\underbrace{\prod_{j\in[K]\cup \{K+k\}}
\Big(\ov{\bo\Phi}^{(j)}_{c,\aaa} -\bo\Delta_{j, c} \Big)}_{\text{denoted as: }\ov t_{\bo y_c,\aaa}}
\cdot \ov\nu_{\aaa},
\end{align}
where the second line above is just the equivalent restatement of the first line above by following the Khatri-Rao product definition.
With the definitions of $\bo\Delta_{j,c}$ in \eqref{eq-tstar}, we claim that the $\ov t_{\bo y_c,\aaa}$ defined on the right hand side (RHS) of the above \eqref{eq-sumprod1} equals zero for all $\aaa\in\{0,1\}^K$.
This is true because due to the first two terms $\overline\theta_{c\mid 0}^{(k)}\ee_k + \overline\theta_{c\mid 1}^{(K+k)}\ee_{K+k}$ in $\bo\Delta_{:,\,c}$ defined in \eqref{eq-tstar}, the $\ov t_{\bo y_c,\aaa}$ contains a factor of
$$
\Big(\ov{\bo\Phi}^{(k)}_{c,\aaa} - \overline\theta_{c\mid 0}^{(k)}\Big) \Big(\ov{\bo\Phi}^{(K+k)}_{c,\aaa} - \overline\theta_{c\mid 1}^{(K+k)}\Big),
$$
and this factor must be zero because if $\aaa\succeq \cg_k=\cg_{K+k}$, then the second factor $\Big(\ov{\bo\Phi}^{(K+k)}_{c,\aaa} - \overline\theta_{c\mid 1}^{(K+k)}\Big)=0$ , and if $\aaa\nsucceq \cg_k=\cg_{K+k}$, then the first factor $\Big(\ov{\bo\Phi}^{(k)}_{c,\aaa} - \overline\theta_{c\mid 0}^{(k)}\Big)=0$.
Now we have shown $\ov t_{\bo y_c, \aaa} = 0$ for all $\aaa$, so all terms on the RHS of \eqref{eq-sumprod1} are zero and \eqref{eq-sumprod1} now becomes
\begin{align*}
\sum_{\aaa\in\{0,1\}^K} t_{\bo y_c,\aaa} \nu_{\aaa} = 0.
\end{align*}
Now due to the third term $\sum_{1\leq m(\neq k)\leq K\atop \alpha'_m=1} \theta_{c\mid 0}^{(m)}$ in the definition of $\bo\Delta_{:, c}$ in \eqref{eq-tstar}, the term $t_{\bo y_c, \aaa} = \prod_{j\in[K]\cup \{K+k\}} \Big({\bo\Phi}^{(j)}_{c,\aaa} -\bo\Delta_{j, c} \Big)$ contains a factor
$$
\prod_{1\leq m(\neq k)\leq K\atop \alpha'_m=1} \left( {\bo\Phi}^{(j)}_{c,\aaa} - \theta_{c\mid 0}^{(m)} \right) = \prod_{1\leq m(\neq k)\leq K\atop \alpha'_m=1} \left( {\theta}^{(j)}_{c\mid\aaa} - \theta_{c\mid 0}^{(m)} \right),
$$
where we use ${\bo\Phi}^{(j)}_{c,\aaa}$ and ${\theta}^{(j)}_{c\mid\aaa}$ interchangeably to denote the same quantity -- the conditional probability of $y_j=c$ given $\bo a=\aaa$.
The above factor would equal zero if for some $m\in[K], m\neq k$ there is $\alpha_m'=1$ but $\alpha_m=0$. Similarly, due to the fourth term $\sum_{1\leq m(\neq k)\leq K\atop \alpha'_m=0} \theta_{c\mid 1}^{(m)}$ in the definition of $\bo\Delta_{1:2K,c}$ in \eqref{eq-tstar},  the entry $t_{\yy_c, \aaa}$ contains a factor 
$$
\prod_{1\leq m(\neq k)\leq K\atop \alpha'_m=0} \left({\theta}^{(j)}_{c\mid\aaa} - \theta_{c\mid 1}^{(m)} \right),
$$
and this factor would equal zero if for some $m\in[K], m\neq k$ there is $\alpha_m'=0$ but $\alpha_m=1$. Summarizing the above two situations, we have that $t_{\yy_c, \aaa}=0$ if binary pattern $\aaa$ does not exactly equal pattern $\aaa'$ on all but the $k$th entry.
Recall that $\alpha'_k=0$. Denote by $\aaa'+\ee_k$ the binary pattern that equals $\aaa'$ on all but the $k$th entry, with the $k$th entry being one.
Then the previously obtained equality $\sum_{\aaa\in\{0,1\}^K} t_{\yy_c, \aaa} \cdot \nu_{\aaa} = 0$ can be written as
\begin{align*}
    0 &= \sum_{\aaa\in\{0,1\}^K} t_{\yy_c, \aaa} \cdot \nu_{\aaa} \\ \notag
    &= t_{\yy_c,\aaa'}\cdot \nu_{\aaa'} + t_{\yy_c,\aaa'+\ee_k}\cdot \nu_{\aaa'+\ee_k}
    \\[4mm]
    &= \prod_{1\leq m(\neq k)\leq K\atop \alpha'_m=1} \left( \theta_{c\mid 1}^{(m)} - \theta_{c\mid 0}^{(m)} \right) \times 
    \prod_{1\leq m(\neq k)\leq K\atop \alpha'_m=0} \left( \theta_{c\mid 0}^{(m)} - \theta_{c\mid 1}^{(m)} \right) \\ \notag
    & \quad \times \left\{
    \nu_{\aaa'}
    \Big( \theta_{c\mid \aaa'}^{(k)} -  \ov\theta_{c\mid 0}^{(k)} \Big)
    \Big( \theta_{c\mid \aaa'}^{(K+k)} -  \ov\theta_{c\mid 1}^{(K+k)} \Big)
    +
    \nu_{\aaa'+\ee_k}
    \Big( \theta_{c\mid \aaa'+\ee_k}^{(k)} -  \ov\theta_{c\mid 0}^{(k)} \Big)
    \Big( \theta_{c\mid \aaa'+\ee_k}^{(K+k)} -  \ov\theta_{c\mid 1}^{(K+k)} \Big)
    \right\}
    \\[4mm]
    &= \prod_{1\leq m(\neq k)\leq K\atop \alpha'_m=1} 
    \left( \theta_{c\mid 1}^{(m)} - \theta_{c\mid 0}^{(m)} \right) \times 
    \prod_{1\leq m(\neq k)\leq K\atop \alpha'_m=0} \left( \theta_{c\mid 0}^{(m)} - \theta_{c\mid 1}^{(m)} \right) \\
    & \quad \times \left\{
    \nu_{\aaa'}
    \Big( \theta_{c\mid 0}^{(k)} -  \ov\theta_{c\mid 0}^{(k)} \Big)
    \Big( \theta_{c\mid 0}^{(K+k)} -  \ov\theta_{c\mid 1}^{(K+k)} \Big)
    +
    \nu_{\aaa'+\ee_k}
    \Big( \theta_{c\mid 1}^{(k)} -  \ov\theta_{c\mid 0}^{(k)} \Big)
    \Big( \theta_{c\mid 1}^{(K+k)} -  \ov\theta_{c\mid 1}^{(K+k)} \Big)
    \right\},
\end{align*}
where the last equality above follows from two facts (1) $\theta_{c\mid \aaa'}^{(k)} = \theta_{c\mid 0}^{(k)}$ due to $\alpha'_k=0$; and (2) $\theta_{c\mid \aaa'+\ee_k}^{(K+k)} = \theta_{c\mid 1}^{(K+k)}$ due to $(\aaa'+\ee_k)_k=1$. Now note that in the above display, the first two product factors are nonzero because of the following assumption made in the main text
\begin{equation}\notag
\theta^{(j)}_{c\mid 1} > \theta^{(j)}_{c\mid 0},\quad c=1,\ldots,d-1.
\end{equation}
Therefore, we obtain
\begin{align}\label{eq-key}
    \nu_{\aaa'}
    \Big( \theta_{c\mid 0}^{(k)} -  \ov\theta_{c\mid 0}^{(k)} \Big)
    \Big( \theta_{c\mid 0}^{(K+k)} -  \ov\theta_{c\mid 1}^{(K+k)} \Big)
    +
    \nu_{\aaa'+\ee_k}
    \Big( \theta_{c\mid 1}^{(k)} -  \ov\theta_{c\mid 0}^{(k)} \Big)
    \Big( \theta_{c\mid 1}^{(K+k)} -  \ov\theta_{c\mid 1}^{(K+k)} \Big)=0.
\end{align}
Note that the above key equation holds for an arbitrary $\aaa'$ with $\alpha_k'=0$ and also for an arbitrary $c\in\{1,\ldots,d-1\}$. For each $c$ define 
\begin{align}\label{eq-defx}
    x_{0k,c} & = \Big( \theta_{c\mid 0}^{(k)} -  \ov\theta_{c\mid 0}^{(k)} \Big)
    \Big( \theta_{c\mid 0}^{(K+k)} -  \ov\theta_{c\mid 1}^{(K+k)} \Big),\\ \notag
    x_{1k,c} & = \Big( \theta_{c\mid 1}^{(k)} -  \ov\theta_{c\mid 0}^{(k)} \Big)
    \Big( \theta_{c\mid 1}^{(K+k)} -  \ov\theta_{c\mid 1}^{(K+k)} \Big).
\end{align}
We claim that only the first factor of $x_{0k,c}$ and only the second factor of $x_{1k,c}$ can potentially be zero {and explain the reasons below. Take $x_{1k,c}$ for example. 
Lemma \ref{lem-neq} guarantees that $\theta_{c\mid 1}^{(k)} \neq   \ov\theta_{c\mid 0}^{(k)}$ so the first factor $\Big( \theta_{c\mid 1}^{(k)} -  \ov\theta_{c\mid 0}^{(k)} \Big)$ in $x_{1k, c}$ must be nonzero. Therefore, only the second factor $\Big( \theta_{c\mid 1}^{(K+k)} -  \ov\theta_{c\mid 1}^{(K+k)} \Big)$ in $x_{1k, c}$ could potentially be zero. Similarly, Lemma \ref{lem-neq} also guarantees that the second factor in $x_{0k,c}$ is nonzero because $\theta_{c\mid 0}^{(K+k)} \neq  \ov\theta_{c\mid 1}^{(K+k)}$.}
Then we have that
\begin{align}\label{eq-idx1}
&x_{0k,c}=0 ~\text{ if and only if }~
\theta_{c\mid 0}^{(k)} =  \ov\theta_{c\mid 0}^{(k)},\\ \label{eq-idx2}
&x_{1k,c}=0 ~\text{ if and only if }~
\theta_{c\mid 1}^{(K+k)} =  \ov\theta_{c\mid 1}^{(K+k)}.
\end{align}
Then with $\aaa'$ ranging over all the $2^{K-1}$ possible configurations and $c$ ranging over $\{1,\ldots,d-1\}$, Eq.~\eqref{eq-key} implies the following systems of equations hold,

\begin{align}\label{eq-keysys}
    \underbrace{\begin{pmatrix}
        \nu_{{\aaa'}^{(1)}} &~ \nu_{{\aaa'}^{(1)}+\ee_k} \\
        \nu_{{\aaa'}^{(2)}} &~ \nu_{{\aaa'}^{(2)}+\ee_k} \\
        \vdots &~\vdots \\
        \nu_{{\aaa'}^{\left(2^{K-1}\right)}} &~ \nu_{{\aaa'}^{\left(2^{K-1}\right)}+\ee_k} \\
    \end{pmatrix}}_{\text{matrix }\mathbf P^{(k)}\text{ of size }2^{K-1}\times 2}
    \bo\cdot \begin{pmatrix}
        x_{0k,1} &~ x_{0k,2} &~ \cdots &~ x_{0k,d-1} \\
        x_{1k,1} &~ x_{1k,2} &~ \cdots &~ x_{1k,d-1}
    \end{pmatrix}_{2\times (d-1)}
    =
    \zero_{2^{K-1}\times (d-1)},
\end{align}
where ${\aaa'}^{(1)}, {\aaa'}^{(2)}, \ldots, {\aaa'}^{(2^{K-1})}$ represent the $|\{0,1\}^{K-1}|=2^{K-1}$ possible configurations $\aaa'$ can take, all having the $k$th entry being zero. We denote the $2^{K-1}\times 2$ matrix on the left hand side of \eqref{eq-keysys} consisting of $\nu_{\aaa}$'s by $\mathbf P^{(k)}$, and denote its first column by $\mathbf P^{(k)}_{\bcolon,1}$ and its second column by $\mathbf P^{(k)}_{\bcolon,2}$.
The system \eqref{eq-keysys} can be written as 
$$
x_{0k,c}\mathbf P^{(k)}_{\bcolon,1} + x_{1k,c}\mathbf P^{(k)}_{\bcolon,2} = \zero_{2^{K-1}\times 1}, \quad
c=1,\ldots,d-1,
$$
therefore we know that $(x_{0k,1},~ x_{1k,1}) = \cdots = (x_{0k,d-1},~ x_{1k,d-1}) = (0,~0)$ holds if the two $2^{K-1}$-dimensional vectors $\mathbf P^{(k)}_{\bcolon,1}$ and $\mathbf P^{(k)}_{\bcolon,2}$ are linearly independent. 
{Now note that the entries in  $\mathbf P^{(k)}_{\bcolon,1}$ and $\mathbf P^{(k)}_{\bcolon,2}$ are entries $\nu_{\aaa}$ for $\aaa\in\{0,1\}^K$, so the values of $\nnu=(\nu_{\aaa})$ that would yield the two vectors $\mathbf P^{(k)}_{\bcolon,1}$ and $\mathbf P^{(k)}_{\bcolon,2}$ linearly dependent are the zero set of certain polynomials of $\nu_{\aaa}$'s. More specifically, the following set
\begin{align}\label{eq-defnk}
    \mathcal N_{k}=
    \{\nnu:\, \mathbf P^{(k)}_{\bcolon,1} \text{ and } \mathbf P^{(k)}_{\bcolon,2} \text{ are linearly dependent for }\mathbf P^{(k)}\text{ defined in }\eqref{eq-keysys}.\}\end{align}
is a zero set of all the $2\times 2$ sub-determinants of the $2^{K-1}\times 2$ matrix consisting  $\mathbf P^{(k)}_{\bcolon,1}$ and $\mathbf P^{(k)}_{\bcolon,2}$ as the two columns.}
Hence $\mathcal N_k$ forms a algebaric subvariety of the parameter space $\bo\Delta^{2^K-1}$ of $\nnu$ and $\mathcal N_k$ has  measure zero with respect to the Lebesgue measure on $\bo\Delta^{2^K-1}$. Further, recall that as long as $\nnu \not\in \mathcal N_k$ and $\nu_{\aaa}>0$, there is $x_{0k,c}=x_{1k,c}=0$ which implies the identifiability of $\theta^{k}_{c\mid 0}$ and $\theta^{K+k}_{c\mid 1}$ as shown in \eqref{eq-idx1} and \eqref{eq-idx2}. Summarizing the conclusion for all the $k\in[K]$, we have that as long as
\begin{align}\label{eq-nk}
    \nnu \not\in \cup_{k\in[K]} \mathcal N_k,
\end{align}
all the $\theta$-parameters will be identifiable. Since $\cup_{k\in[K]} \mathcal N_k \subseteq \bo\Delta^{2^K-1}$ has measure zero with respect to the Lebesgue measure on $\bo\Delta^{2^K-1}$, we have essentially shown that $\btheta$ are generically identifiable. Further, when $\nnu \not\in \cup_{k\in[K]} \mathcal N_k$ and $\btheta$ are identifiable with $\bo\Phi_j = \ov{\bo\Phi}_j$ for all $j$, we next show that $\nnu$ are also identifiable. Consider the equations given by the first $K$ observed variables,
\begin{align*}
    \Big(\bigodot_{j\in[K]} \bo\Phi^{(j)}\Big) \cdot \nnu =
    \Big(\bigodot_{j\in[K]} \overline{\bo\Phi}^{(j)}\Big) \cdot \overline\nnu
    =
    \Big(\bigodot_{j\in[K]} {\bo\Phi}^{(j)}\Big) \cdot \overline\nnu,
\end{align*}
Given an arbitrary binary pattern $\aaa$ and any $c\in\{1,\ldots,d-1\}$, define 
\begin{align*}
\bo\Delta_{:, c} &= \sum_{1\leq k\leq K\atop \alpha_k=1} \theta^{(k)}_{c\mid 0}\ee_k + \sum_{1\leq k\leq K\atop \alpha_k=0} \theta^{(k)}_{c\mid 1}\ee_k,\quad
\yy_c = c\sum_{k=1}^K \ee_k.
\end{align*}
Note the $\bo\Delta_{:, c}$ and $\yy_c$ are different from the previously defined $\bo\Delta_{:,c}$ and $\yy_c$ in \eqref{eq-tstar}.
For any $\aaa'\in\{0,1\}^K$, denote by $t_{\yy_c, \aaa'}$ and $\ov t_{\yy_c, \aaa'}$ the element in $\bigodot_{j\in[K]} {\bo\Phi}^{(j)}$ and $\bigodot_{j\in[K]} \ov{\bo\Phi}^{(j)}$, respectively, indexed by response pattern $\yy_c$ and latent pattern $\aaa'$.
{According to the definition of the Khatri-Rao product, $t_{\yy_c, \aaa'}$ and $\ov t_{\yy_c, \aaa'}$ have the following expression:
\begin{align*}
t_{\yy_c, \aaa'} = \ov t_{\yy_c, \aaa'}  &= \prod_{1\leq k\leq K\atop \alpha_k=1} (\theta^{(k)}_{c\mid \aaa'}-\theta^{(k)}_{c\mid 0}) \prod_{1\leq k\leq K\atop \alpha_k=0} (\theta^{(k)}_{c\mid \aaa'}-\theta^{(k)}_{c\mid 1}).
\end{align*}
for some $\alpha_k=1$ we have $\alpha'_k=0$, then it will hold that $\theta^{(k)}_{c\mid \aaa'} = \theta^{(k)}_{c\mid 0}$ and hence $\theta^{(k)}_{c\mid \aaa'} -\theta^{(k)}_{c\mid 0} = 0$, which implies $t_{\yy_c, \aaa'} = \ov t_{\yy_c, \aaa'} = 0$.
On the other hand, if for some $\alpha_k=0$ we have $\alpha'_k=1$, then it will hold that $\theta^{(k)}_{c\mid \aaa'} = \theta^{(k)}_{c\mid 1}$ and hence $\theta^{(k)}_{c\mid \aaa'} -\theta^{(k)}_{c\mid 1} = 0$, which also implies $t_{\yy_c, \aaa'} = \ov t_{\yy_c, \aaa'} = 0$.
In summary, as long as $\alpha_k \neq \alpha'_k$ for any $k\in[K]$, it will hold that $t_{\yy_c, \aaa'} = \ov t_{\yy_c, \aaa'} = 0$.
Therefore $t_{\yy_c, \aaa'}$ and $\ov t_{\yy_c, \aaa'}$ are nonzero only if $\aaa'=\aaa$.}
Therefore
\begin{align*}
&~\prod_{1\leq k\leq K\atop \alpha_k=1} \Big(\theta^{(k)}_{c\mid 1} - \theta^{(k)}_{c\mid 0}\Big)
\prod_{1\leq k\leq K\atop \alpha_k=0} \Big(\theta^{(k)}_{c\mid 0} - \theta^{(k)}_{c\mid 1}\Big) \nu_{\aaa}\\
=&~
\prod_{1\leq k\leq K\atop \alpha_k=1} \Big(\theta^{(k)}_{c\mid 1} - \theta^{(k)}_{c\mid 0}\Big)
\prod_{1\leq k\leq K\atop \alpha_k=0} \Big(\theta^{(k)}_{c\mid 0} - \theta^{(k)}_{c\mid 1}\Big) \ov \nu_{\aaa},
\end{align*}
which further gives $\nu_{\aaa}=\ov \nu_{\aaa}$ because $\prod_{1\leq k\leq K\atop \alpha_k=1} \Big(\theta^{(k)}_{c\mid 1} - \theta^{(k)}_{c\mid 0}\Big)
\prod_{1\leq k\leq K\atop \alpha_k=0} \Big(\theta^{(k)}_{c\mid 0} - \theta^{(k)}_{c\mid 1}\Big) \neq 0$. Since $\aaa$ above is arbitrary, we have obtained $\nnu=\ov\nnu$. Thus far we have proven that as long as $\nnu$ satisfies \eqref{eq-nk}, there must be $(\btheta, \nnu)=(\ov\btheta, \ov\nnu)$. This establishes the generic identifiability of all the model parameters and completes the proof of part (a) of the theorem.

\bigskip
\noindent
\textbf{Proof of Part (b) of Theorem \ref{thm-main}.}
We next prove part (b) of the theorem by showing that the two vectors $\mathbf P^{(k)}_{\bcolon,1}$ and $\mathbf P^{(k)}_{\bcolon,2}$ are linearly dependent if and only if 
$\alpha_k$ and other latent variables are statistically independent.
If the two vectors $\mathbf P^{(k)}_{\bcolon,1}$ and $\mathbf P^{(k)}_{\bcolon,2}$ are linearly dependent, then with out loss of generality we can assume $\mathbf P^{(k)}_{\bcolon,2} = \rho\cdot \mathbf P^{(k)}_{\bcolon,1}$ for some $\rho\neq 0$. Then by \eqref{eq-keysys}, there is $\nu_{{\aaa'}^{(\ell)}+\ee_k} = \rho \cdot \nu_{{\aaa'}^{(\ell)}}$ for $\ell=1,\ldots,2^{K-1}$, which implies
\begin{align}\label{eq-crossprod}
    \nu_{{\aaa'}^{(\ell)}+\ee_k}\cdot \nu_{{\aaa'}^{(m)}}=
    \nu_{{\aaa'}^{(m)}+\ee_k}\cdot \nu_{{\aaa'}^{(\ell)}}, ~~\text{for any}~~ 1\leq m,\ell\leq 2^{K-1}.
\end{align}
Denote $(a_1,\ldots,a_{k-1},a_{k+1},\ldots,a_K)=:\bo a_{-k}$. Since all the $a_m$'s are binary, for any $\bo s\in\{0,1\}^{K-1}$ and any $t\in\{0,1\}$ we have
\begin{align}\notag
    &~\mathbb P (\bo a_{-k} = \bo s,\; a_k = 1)\\\notag
    =&~
    \mathbb P (\bo a_{-k} = \bo s,\; a_k = 1) \Big(\sum_{\bo b\in\{0,1\}^K} \nu_{\bo b} \Big)\\\notag
    =&~
    \mathbb P (\bo a_{-k} = \bo s,\; a_k = 1) \Big\{\sum_{\aaa'\in\{0,1\}^K\atop \alpha'_k=0} \Big(\nu_{\aaa'+\ee_k} + \nu_{\aaa'}\Big) \Big\}\\\notag
    =&~
    \sum_{\bo a'\in\{0,1\}^K\atop a'_k=0} 
    \Big\{ \mathbb P (\bo a_{-k} = \bo s,\; a_k = 1)\cdot \nu_{\aaa'+\ee_k} +
    \mathbb P (\bo a_{-k} = \bo s,\; a_k = 1)\cdot \nu_{\aaa'} \Big\}
    \\ \notag
    \stackrel{\eqref{eq-crossprod}}{=}&~ \sum_{\aaa'\in\{0,1\}^K\atop \alpha'_k=0} 
    \Big\{  \mathbb P (\bo a_{-k} = \bo s,\; a_k = 1)\cdot \nu_{\aaa'+\ee_k}
    +
    \mathbb P (\bo a_{-k} = \bo s,\; a_k = 0)\cdot \nu_{\aaa'+\ee_k}\Big\}
    \\ \notag
    =&~ \sum_{\aaa'\in\{0,1\}^K\atop \alpha'_k=0} 
    \mathbb P (\bo a_{-k} = \bo s) \nu_{\aaa'+\ee_k}\\  \label{eq-aaa1}
    =&~\mathbb P (\bo a_{-k} = \bo s) \cdot \mathbb P(a_k=1),
\end{align}
where the last but third equality results from $ \mathbb P (\bo a_{-k} = \bo s,\; a_k = 1)\cdot \nu_{\aaa'} = \mathbb P (\bo a_{-k} = \bo s,\; a_k = 0)\cdot \nu_{\aaa'+\ee_k}$ by \eqref{eq-crossprod}.
Further, we can show that 
\begin{align}\notag
\mathbb P (\bo a_{-k} = \bo s,\; a_k = 0) &= \mathbb P (\bo a_{-k} = \bo s) - \mathbb P (\bo a_{-k} = \bo s,\; a_k = 1) \\ \notag
&= \mathbb P (\bo a_{-k} = \bo s) - \mathbb P (\bo a_{-k} = \bo s) \cdot \mathbb P(a_k=1)\\ \label{eq-aaa2}
&= \mathbb P (\bo a_{-k} = \bo s) \cdot \mathbb P(a_k=0).
\end{align}
The above two conclusions \eqref{eq-aaa1} and \eqref{eq-aaa2} indicate $a_k$ and $\bo a_{-k}$ are statistically independent, that is, $\alpha_k\indep\aaa_{-k}$. Recall the definition in \eqref{eq-defnk} that $\mathcal N_k=\{\nnu:\, \mathbf P^{(k)}_{\bcolon,1} \text{ and } \mathbf P^{(k)}_{\bcolon,1} \text{ defined in \eqref{eq-defnk}}$  $\text{are linearly dependent}.\}$, and now we have shown
\begin{flalign*}
    \text{(C1)}\quad &\nnu \not\in \mathcal N_k
    ~\Longrightarrow~ 
    \theta^{(k)}_{\bcolon}\text{ and } \theta^{(K+k)}_{\bcolon}\text{ are identifiable.} &&
    \\
   \text{(C2)}\quad &\nnu \in \mathcal N_k
    ~\Longrightarrow~ 
    a_k \indep\bo a_{-k}. &&
\end{flalign*}
Next we show that the reverse directions of the above two claims (C1) and (C2) also hold; namely, we next show that 

\begin{flalign*}
    (\tilde{\text{C1}})\quad &
    \text{If }\theta^{(k)}_{\bcolon}\text{ and } \theta^{(K+k)}_{\bcolon} \text{ are identifiable, then } \nnu\not\in\mathcal N_k; &&
    \\
    (\tilde{\text{C2}})\quad &
    \text{If } a_k \indep \bo a_{-k},\text{ then }\nnu\in\mathcal N_k. &&
\end{flalign*}
Suppose $\alpha_k \indep\aaa_{-k}$, then for any $\bo s\in\{0,1\}^{K-1}$ and $z\in\{0,1\}$
there is $\mathbb P (\bo a_{-k} = \bo s, \; a_k=z) = \mathbb P (\bo a_{-k} = \bo s) \cdot \mathbb P(a_k=z)$, which implies the following,
\begin{align*}
\text{for any } \bo s\in\{0,1\}^{K-1},\quad
\begin{cases}
\dfrac{\mathbb P (\bo a_{-k} = \bo s, \; a_k=1)}{\mathbb P (\bo a_{-k} = \bo s)} = \mathbb P(a_k=1) := \rho_1;
\\[4mm]
\dfrac{\mathbb P (\bo a_{-k} = \bo s, \; a_k=0)}{\mathbb P (\bo a_{-k} = \bo s)} = \mathbb P(a_k=0) := \rho_0.
\end{cases}
\end{align*}
\vspace{2mm}

\noindent
Taking the ratio of the above two equalities gives
\begin{align*}
    \frac{\mathbb P (\bo a_{-k} = \bo s, \; a_k=1)}{\mathbb P (\bo a_{-k} = \bo s, \; a_k=0)} = \frac{\rho_1}{\rho_0} ~~\text{for any}~~\bo s\in\{0,1\}^{K-1}.
\end{align*}
Recalling the definition of the $2^{K-1}\times 2$ matrix $\mathbf P^{(k)}$ in \eqref{eq-keysys}, the above equality exactly means the two vectors $\mathbf P^{(k)}_{\bcolon,1}$ and $\mathbf P^{(k)}_{\bcolon,2}$ are linearly dependent. So we have shown $(\tilde{\text{C2}})$ holds.

Finally, to show $(\tilde{\text{C1}})$ holds, it suffices to prove that if $\nnu\in \mathcal N_k$, then $\theta^{(k)}_{\bcolon}$ and  $\theta^{(K+k)}_{\bcolon}$ are not identifiable. To this end, we next explicitly construct alternative parameters  $\ov\theta^{(k)}_{\bcolon}$ and  $\ov\theta^{(K+k)}_{\bcolon}$ that lead to the same distributions of the observables as the true parameters $\theta^{(k)}_{\bcolon}$ and  $\theta^{(K+k)}_{\bcolon}$.
If $\nnu\in\mathcal N_k$, then without loss of generality we can assume there exists $\rho > 0$ such that $\nu_{\aaa'+\ee_k} = \rho \cdot \nu_{\aaa'}$ for any $\aaa'$ with $\alpha'_k=0$. Then equations \eqref{eq-key} now become
\begin{align}\label{eq-rhotheta}
    \Big( \theta_{c\mid 0}^{(k)} -  \ov\theta_{c\mid 0}^{(k)} \Big)
    \Big( \theta_{c\mid 0}^{(K+k)} -  \ov\theta_{c\mid 1}^{(K+k)} \Big)
    +
    \rho\cdot 
    \Big( \theta_{c\mid 1}^{(k)} -  \ov\theta_{c\mid 0}^{(k)} \Big)
    \Big( \theta_{c\mid 1}^{(K+k)} -  \ov\theta_{c\mid 1}^{(K+k)} \Big)=0.
\end{align}
Now consider an arbitrary $\ov \theta_{c\mid 0}^{(k)}$ in a small neighborhood of the true parameter $\theta_{c\mid 0}^{(k)}$.
{We treat the unknown alternative parameter $\ov\theta_{c\mid 1}^{(K+k)}$ as an unknown variable and solve \eqref{eq-rhotheta} for $\ov\theta_{c\mid 1}^{(K+k)}$. The explicit solution is as follows:
$$
\ov\theta_{c\mid 1}^{(K+k)} = 
\theta_{c\mid 1}^{(K+k)} + 
\frac{\Big(\theta_{c\mid 0}^{(k)} - \ov\theta_{c\mid 0}^{(k)}\Big)\Big(\theta_{c\mid 0}^{(K+k)} - \theta_{c\mid 1}^{(K+k)}\Big)}
{\Big(\theta_{c\mid 0}^{(k)} - \ov\theta_{c\mid 0}^{(k)}\Big) + \rho\cdot \Big(\theta_{c\mid 1}^{(k)} - \ov\theta_{c\mid 0}^{(k)}\Big)}.
$$
The sum in the denominator in the expression of $\ov\theta_{c\mid 1}^{(K+k)}$ above is a result of solving the linear equation in \eqref{eq-rhotheta}.}
Note that the above $\ov\theta_{c\mid 1}^{(K+k)}$ and $\ov \theta_{c\mid 0}^{(k)}$ satisfy all the equations in $\Big(\bigodot_{j\in[p]} \bo\Phi^{(j)}\Big) \cdot \nnu =
    \Big(\bigodot_{j\in[p]} \overline{\bo\Phi}^{(j)}\Big) \cdot \overline\nnu$. 
We have thus shown that $\theta_{c\mid 1}^{(K+k)}$ and $\theta_{c\mid 0}^{(k)}$ are not identifiable and prove the previous claim $(\tilde{\text{C1}})$.

In summary, now that we have proven (C1), $(\tilde{\text{C1}})$, (C2), $(\tilde{\text{C2}})$, there are
\begin{align*}
    & \theta^{(k)}_{\bcolon}\text{ and } \theta^{(K+k)}_{\bcolon}\text{ are identifiable.}\\
    \Longleftrightarrow~  &\nnu \not\in\mathcal N_k=\{\nnu:\, \mathbf P^{(k)}_{\bcolon,1} \text{ and } \mathbf P^{(k)}_{\bcolon,1} \text{ are linearly dependent}.\}
    \\
    \Longleftrightarrow~ &
    a_k \indep \bo a_{-k}.
\end{align*}
We have established that conditional probabilities $\theta^{(k)}_{\bcolon}$ and  $\theta^{(K+k)}_{\bcolon}$ are identifiable if and only if $a_k \not\! \indep \bo a_{-k}$, this means the parameters associated with $\ch(a_k)\mid a_k$ are identifiable if and only if $a_k \not\! \indep \bo a_{-k}$.
This completes the proof of Theorem \ref{thm-main}.
\qed

\section{Additional proofs of the theoretical results}\label{sec-addproofs}

\subsection{Proof of Proposition \ref{prop-nece}}
Under the condition of the proposition, we construct a non-identifiable example as follows. 
{Recall that in the BLESS model, each observed variable has at most one latent parent. Therefore, under the condition of the proposition, we can assume without loss of generality that the matrix $\G$ takes the following form:}
\begin{align*}
    \G = 
    \begin{pmatrix}
    1 & \zero\\
    \zero & ~\G^\star
    \end{pmatrix},
\end{align*}
where $\G^\star$ has size $(p-1)\times (K-1)$.
Given arbitrary valid model parameters $(\nnu,\btheta)$, we next construct an alternative set of parameters $(\ov\nnu, \ov\btheta) \neq (\nnu,\btheta)$ such that $(\ov\nnu, \ov\btheta)$ and $(\nnu,\btheta)$ lead to the same distribution of the observed response vector $\bo y$.
Suppose $\ov\theta^{(j)}_{c_j\mid x} = \theta^{(j)}_{c_j\mid x}$ for all $j\in\{2,\ldots,p\}$, $c_j\in[d]$, and $x\in\{0,1\}$.
Then $\mathbb P(\bo y\mid \nnu,\btheta) = \mathbb P(\bo y\mid \ov\nnu, \ov\btheta)$ implies the following equations
\begin{align*}
\forall\aaa^*\in\{0,1\}^{K-1},
\quad
\forall c\in[d],
\quad
    \theta^{(1)}_{c\mid 0} \nu_{(0,\aaa^*)} + \theta^{(1)}_{c\mid 1} \nu_{(1,\aaa^*)} = \ov\theta^{(1)}_{c\mid 0} \ov \nu_{(0,\aaa^*)} + \ov\theta^{(1)}_{c\mid 1}  \ov \nu_{(1,\aaa^*)}.
\end{align*}
For each possible $\aaa^*\in\{0,1\}^{K-1}$, we sum the $d$ equations above for $c=1,\ldots,d$ and further obtain $\nu_{(0,\aaa^*)} + \nu_{(1,\aaa^*)} = \ov \nu_{(0,\aaa^*)} + \ov \nu_{(1,\aaa^*)}$. Therefore the above system of equations are equivalent to the following,
\begin{align*}
\forall\aaa^*\in\{0,1\}^{K-1},
\quad
\begin{cases}
    \nu_{(0,\aaa^*)} + \nu_{(1,\aaa^*)} = \ov \nu_{(0,\aaa^*)} + \ov \nu_{(1,\aaa^*)};
    \\[5mm]
    \theta^{(1)}_{c\mid 0} \nu_{(0,\aaa^*)} + \theta^{(1)}_{c\mid 1} \nu_{(1,\aaa^*)} = \ov\theta^{(1)}_{c\mid 0} \ov \nu_{(0,\aaa^*)} + \ov\theta^{(1)}_{c\mid 1}  \ov \nu_{(1,\aaa^*)},\quad c\in[d].
\end{cases}
\end{align*}
We next set $\ov\theta^{(1)}_{c\mid 0} = \theta^{(1)}_{c\mid 0}$ for all $c\in\{1,\ldots,d\}$, and take the alternative $\ov \theta^{(1)}_{1\mid 1}$ from an arbitrarily small neighborhood of the true parameter $\theta^{(1)}_{1\mid 1}$ with $\ov\theta^{(1)}_{1\mid 1} \neq \theta^{(1)}_{1\mid 1}$. Then
\begin{align}\label{eq-alt}
\begin{cases}
    \ov \nu_{(1,\aaa^*)} =  \nu_{(1,\aaa^*)} \cdot \dfrac{\theta^{(1)}_{1\mid 1} - \theta^{(1)}_{1\mid 0}}{\ov\theta^{(1)}_{1\mid 1} - \theta^{(1)}_{1\mid 0}},
    \quad \forall\aaa^*\in\{0,1\}^{K-1} ;
    \\[8mm]
    \ov \nu_{(0,\aaa^*)} = \nu_{(0,\aaa^*)} + \nu_{(1,\aaa^*)} \cdot \dfrac{\ov\theta^{(1)}_{1\mid 1} - \theta^{(1)}_{1\mid 1}}{\ov\theta^{(1)}_{1\mid 1} - \theta^{(1)}_{1\mid 0}}, \quad\forall\aaa^*\in\{0,1\}^{K-1};
    \\[8mm]
    \ov\theta^{(1)}_{c\mid 1} = \theta^{(1)}_{c\mid 0} + (\theta^{(1)}_{c\mid 1} - \theta^{(1)}_{c\mid 0}) \cdot \dfrac{\ov\theta^{(1)}_{1\mid 1} - \theta^{(1)}_{1\mid 0}}{\theta^{(1)}_{1\mid 1} - \theta^{(1)}_{1\mid 0}},\quad \forall c=2,\ldots,d;
\end{cases}
\end{align}
{We next show that the alternative parameters $\ov \nu_{(1,\aaa^*)}$, $\ov \nu_{(0,\aaa^*)}$, and $\ov\theta^{(1)}_{c\mid 1}$ defined above are different from the true parameters.
First define the ratio terms as follows:
$$
\rho_1 = \dfrac{\theta^{(1)}_{1\mid 1} - \theta^{(1)}_{1\mid 0}}{\ov\theta^{(1)}_{1\mid 1} - \theta^{(1)}_{1\mid 0}},\qquad
\rho_2 = \dfrac{\ov\theta^{(1)}_{1\mid 1} - \theta^{(1)}_{1\mid 1}}{\ov\theta^{(1)}_{1\mid 1} - \theta^{(1)}_{1\mid 0}}.
$$
So we can re-express the alternative parameters $\ov \nu_{(1,\aaa^*)}$, $\ov \nu_{(0,\aaa^*)}$, and $\ov\theta^{(1)}_{c\mid 1}$ defined in \eqref{eq-alt} as 
\begin{align*}
\begin{cases}
    \ov \nu_{(1,\aaa^*)} =  \nu_{(1,\aaa^*)} \cdot \rho_1,
    \quad \forall\aaa^*\in\{0,1\}^{K-1} ;
    \\[8mm]
    \ov \nu_{(0,\aaa^*)} = \nu_{(0,\aaa^*)} + \nu_{(1,\aaa^*)} \cdot \rho_2, \quad\forall\aaa^*\in\{0,1\}^{K-1};
    \\[8mm]
    \ov\theta^{(1)}_{c\mid 1} = \theta^{(1)}_{c\mid 0} + (\theta^{(1)}_{c\mid 1} - \theta^{(1)}_{c\mid 0}) \cdot \rho_2,\quad \forall c=2,\ldots,d;
\end{cases}
\end{align*}
Note that the alternative parameter $\ov \nu_{(1,\aaa^*)}$ differs from the true parameter $\nu_{(1,\aaa^*)}$ by a multiplicative factor  $\rho_1$.
Since we have assumed $\ov \theta^{(1)}_{1\mid 1} \neq \theta^{(1)}_{1\mid 1}$, the ratio $\rho_1\neq 1$ which means $\ov \nu_{(1,\aaa^*)} =  \rho\cdot\nu_{(1,\aaa^*)}  \neq \nu_{(1,\aaa^*)}$.
Further, $\ov\theta^{(1)}_{1\mid 1} \neq \theta^{(1)}_{1\mid 1}$ also means that the ratio $\rho_2 \neq 0$, which implies $\ov \nu_{(0,\aaa^*)} = \nu_{(0,\aaa^*)} +  \rho_2\cdot \nu_{(1,\aaa^*)} \neq \nu_{(0,\aaa^*)}$.
Finally, our model assumption $\theta^{(1)}_{1\mid 1} \neq \theta^{(1)}_{1\mid 0}$ also means $\rho_2\neq 1$, which implies $$\ov\theta^{(1)}_{c\mid 1} = \theta^{(1)}_{c\mid 0} + \rho_2\cdot (\theta^{(1)}_{c\mid 1} - \theta^{(1)}_{c\mid 0}) \neq \theta^{(1)}_{c\mid 0} + (\theta^{(1)}_{c\mid 1} - \theta^{(1)}_{c\mid 0})  = \theta^{(1)}_{c\mid 1}.$$
Now we have shown that $\ov \nu_{(1,\aaa^*)} \neq  \nu_{(1,\aaa^*)}$, $\ov \nu_{(0,\aaa^*)} \neq  \nu_{(0,\aaa^*)}$, and $ \ov\theta^{(1)}_{c\mid 1} \neq \theta^{(1)}_{c\mid 1}$ for $c=2,\ldots,d$.}
Note that the alternative parameter $\ov \theta^{(1)}_{1\mid 1}$ can be chosen from an arbitrarily small neighborhood of the true parameter $\theta^{(1)}_{1\mid 1}$, so we have proven that even local identifiability fails to hold in the considered setting.
This completes the proof of Proposition \ref{prop-nece}.
\qed

\subsection{Proof of Theorem \ref{thm-graph}}
We prove the theorem in two steps.
\noindent\textbf{Step 1.}
In this step we prove the following lemma.
\begin{lemma}\label{lem-gg}
Suppose $\G = (\I_K,\; \I_K)^\top$, which vertically stacks two identity submatrices $\I_K$. Consider that $(\G, \btheta, \nnu)$ and $(\bar\G, \bar\btheta, \bar\nnu)$ lead to the same distribution of the observed vector $\bo y$.
For an arbitrary $h\in[K]$, if there exists two sets $\mathcal A\subseteq [K]\setminus\{h\}$ and $\mathcal B\subseteq\{K+1,\ldots,J\}$ such that $\G$ satisfies
\begin{align*}
    \max_{m\in \mathcal B}~ g_{m,h} &= 0,\\
    \max_{m\in \mathcal B}~ g_{m,k} &=1 \;\text{for all}\; k\in \mathcal A,
\end{align*}
then $\bar\G$ must satisfy
$
\vee_{k\in\mathcal A}~ \bar \cg_k \nsucceq \bar\cg_h.
$
\end{lemma}


Please see the proof of Lemma \ref{lem-gg} in the Supplementary Material.

\bigskip
\noindent\textbf{Step 2.}
{First consider the case where $\G = (\I_K,\; \I_K)^\top$; extension to cases where $\G$ contains more than $2K$ rows will be discussed in the end of this Step 2.}
We next show $\bar{\G} = (\I_K,\; \I_K)^\top$ holds up to a column permutation.
{Let $\mathcal B=\{K+1,\ldots,2K\}\setminus\{K+h\}$} and $\mca_h=[K]\setminus\{h\}$ for an arbitrary index $h\in[K]$. Then the condition in Lemma \ref{lem-gg} is satisfied and 
$$\bigvee_{k\in\mathcal A_h}~ \bar \cg_k \nsucceq \bar\cg_h,$$
which implies that the row vector $\bar\cg_h$ contains an entry of ``1'' in some column $q_h$ with all the $\bar\cg_k$ in $\mathcal A_h$ having ``0'' in this column $q_h$.
Since the above holds for all the $h\in[K]$, we obtain that the $K$ row vectors $\bar g_1,\ldots,\bar g_K$ contains ``1''s in $K$ different columns. This exactly implies that $\bar\G_{1:K,\cdot}$ equals the identity matrix $\I_K$ up to a column permutation. Since the first $K$ rows and the second $K$ rows of $\G$ are both $\I_K$, by symmetry to the above deduction we can also obtain that $\bar\G_{(K+1):(2K),\cdot}$ equals $\I_K$ up to a column permutation. 

Now it only remains to show that the column permutations of $\bar\G_{1:K,\cdot}$ and $\bar\G_{(K+1):(2K),\cdot}$ are the same.
Suppose $\bar\cg_k = \bar\cg_{K+k'}$ for some $k,k'\in[K]$.
Define 
\begin{align*}
    \bo\Delta_{1:p, c} &= \theta^{(k)}_{c\mid 0} \ee_k + \theta^{(K+k')}_{c\mid 1} \ee_{K+k'},\quad \bo\Delta_{1:p, c'} = \zero_{p}\text{ for }c'\neq c;\\
    \yy_c &= c(\ee_k + \ee_{K+k'}).
\end{align*}
Now let $\bo\Delta_{j,:}$ denote the $d$-dimensional vector with entries $(\bo\Delta_{j,1},\ldots,\bo\Delta_{j,d})$.
With this definition, we claim that the row vector corresponding to response pattern $\yy_c$ of $\bigodot_{j\in[p]} \Big(\ov{\bo\Phi}^{(j)}-\bo\Delta_{j,:}\bcdot\one^\top_{2^K} \Big)$ must be a zero-vector.
This is because any entry in this row must contain a factor of 
$$
\Big(\ov\theta^{(k)}_{c\mid\aaa} - \overline\theta_{c\mid 0}^{(k)}\Big) \Big(\ov\theta^{(K+k')}_{c\mid\aaa} - \overline\theta_{c\mid 1}^{(K+k')}\Big),
$$
and this factor must be zero because if $\ov\theta^{(k)}_{c\mid\aaa} - \overline\theta_{c\mid 0}^{(k)} \neq 0$ then $\aaa\succeq \cg_k=\cg_{K+k'}$, and then $\ov\theta^{(K+k')}_{c\mid\aaa} - \overline\theta_{c\mid 1}^{(K+k')} = 0$ must hold.
Now that $\bigodot_{j\in[p]} \Big(\ov{\bo\Phi}^{(j)}-\bo\Delta_{j,:}\bcdot\one^\top_{2^K} \Big)_{\yy_c,\bcolon}$ is a zero-vector, \eqref{eq-trans} gives that 
\begin{align*}
    0 &=
    \bigodot_{j\in[p]} \Big(\ov{\bo\Phi}^{(j)}-\bo\Delta_{j,:}\bcdot\one^\top_{2^K} \Big)_{\yy_c,\bcolon}\cdot \ov\nnu 
    \\
    &
    =
    \bigodot_{j\in[p]} \Big({\bo\Phi}^{(j)}-\bo\Delta_{j,:}\bcdot\one^\top_{2^K} \Big)_{\yy_c,\bcolon}\cdot \nnu \\
    &=
    \left(\theta^{(k)}_{c\mid 1} - \theta^{(k)}_{c\mid 0}\right)
    \left(\theta^{(K+k')}_{c\mid 0} - \theta^{(K+k')}_{c\mid 1}\right) 
    \Big(\sum_{\aaa\succeq \cg_k\atop \aaa\nsucceq\cg_{K+k'}}\nu_{\aaa}\Big).
\end{align*}
If the set $\mathcal M := \{\aaa\in\{0,1\}^K:\,\aaa\succeq \cg_k, \aaa\nsucceq\cg_{K+k'}\}$ is nonempty, then the above equation gives a contradiction. This means $\mathcal M$ must be an empty set, which implies that $\cg_{K+k'} = \cg_{k}$ must hold. Considering the true $\mathbf G=(\I_K, \I_K)^\top$, we have that $k'=k$ must hold. Now we have shown that as long as $\bar\cg_k = \bar\cg_{K+k'}$, there is $k=k'$. This shows $\bar\G_{1:K,\cdot}=\bar\G_{(K+1):(2K),\cdot}$ holds. 

{Next, we consider the case where $\G$ contains more than $2K$ rows with $\G = (\I_K; \I_K; \G^{\star\top})^\top$. (i.e., some latent variable has more than two observed children). Then for any $j=2K+1,\ldots,p$, suppose $y_j$'s latent parent is $a_k$ so $\bo g_j = \ee_k$, where $\ee_k$ here is a $K$-dimensional standard basis vector. Then we only need to change the order of this variable $y_j$ and variable $y_k$ and the graphical matrix corresponding to the following $2K$ variables is still $(\I_K; \I_K)^\top$:
$$
\underbrace{y_1,\ldots,y_{k-1}, y_j, y_{k+1}, \ldots, y_K}_{\text{first $K$ variables forming $\I_K$}}, \underbrace{y_{K+1},\ldots, y_{2K};}_{\text{second $K$ variables forming $\I_K$}}
$$
denote the graphical matrix corresponding to the above $2K$ variables by $\G_{[1:k-1,j,k+1:2K],:}$. Then following exactly the same argument as in the previous paragraph when $\G = (\I_K; \I_K)^\top$, we can get $\ov\G_{[1:k-1,j,k+1:2K],:} = \G_{[1:k-1,j,k+1:2K],:}$, which proves $\ov{\bo g}_j = \bo g_j$. This shows that when $\G = (\I_K; \I_K; \G^{\star\top})^\top$, we still have $\ov\G = \G$ and the measurement graph structure is identifiable.}
Now we have completed the proof of Theorem \ref{thm-graph}.


\qed

\subsection{Proof of Proposition \ref{prop-3chi}}
Under the assumption that each latent variable has three children, we show identifiability in a similar fashion as the proof of Theorem 4 in \cite{allman2009} by using Kruskal's theorem.
{Note that \cite{allman2009} considered a general parameter space {without any inequality constraints} of model parameters and established generic identifiability.
But in our model, we have assumed the following inequality constraints in Equation (3) in the main text:
$$
\theta^{(j)}_{c_j\mid 1} > \theta^{(j)}_{c_j\mid 0} ~\text{ for }~ j\in[p],~ c_j\in[d-1].
$$
By examining the proof of Theorem 4 in \cite{allman2009} and carefully adapting Kruskal's Theorem to our setting, we find that the above inequality constraints on the $\theta$-parameters exactly rule out the non-identifiable case in the parameter set. 
Therefore, we are able to obtain the strict identifiability result in the following proof.}

Now we proceed with the proof of the proposition.
Under the assumption that each latent variable has at least three children variables, suppose without loss of generality that {$\G=(\I_K,\;\I_K,\;\I_K,\;\G^{\star\top})^\top$}, where the submatrix $\G^\star$ can take an arbitrary form. Suppose the alternative parameters $\ov\btheta$, $\ov\nnu$ associated with a potentially different $\ov\G$ lead to the same distribution of the $p$ observed variables.
Group the first $K$ observed variables $y_1,\ldots,y_K$ into one discrete variable with $d^K$ categories and denote it by $z_1$, then each of the $d^K$ possible configurations of the vector $\tilde{\bo c}=(y_1,\ldots,y_K)$ corresponds to one category that $z_1$ can take.
Similarly group $y_{K+1},\ldots,y_{2K}$ into another variable $z_2$, and group $y_{2K+1},\ldots,y_{3K}$ into another variable $z_3$. 
Then given latent pattern $\aaa$, the conditional probability table of $z_1$, $z_2$, $z_3$ each has size $d^K \times 2^K$; denote such a table by ${\bo\Psi}_m$. Based on the star-forest dependence graph structure it is not hard to deduct that each such $d^K \times 2^K$ table can be written as 
$$
{\bo\Psi}_1
=\bigotimes_{j=1}^K \begin{pmatrix}
\theta^{(j)}_{1\mid 0} ~&~ \theta^{(j)}_{1\mid 1} \\[2mm]
\vdots ~&~ \vdots\\[2mm]
\theta^{(j)}_{d\mid 0} ~&~ \theta^{(j)}_{d\mid 1}
\end{pmatrix},
\quad
{\bo\Psi}_2
=\bigotimes_{j=K+1}^{2K} \begin{pmatrix}
\theta^{(j)}_{1\mid 0} ~&~ \theta^{(j)}_{1\mid 1} \\[2mm]
\vdots ~&~ \vdots\\[2mm]
\theta^{(j)}_{d\mid 0} ~&~ \theta^{(j)}_{d\mid 1}
\end{pmatrix},
\quad
{\bo\Psi}_3
=\bigotimes_{j=2K+1}^{3K} \begin{pmatrix}
\theta^{(j)}_{1\mid 0} ~&~ \theta^{(j)}_{1\mid 1} \\[2mm]
\vdots ~&~ \vdots\\[2mm]
\theta^{(j)}_{d\mid 0} ~&~ \theta^{(j)}_{d\mid 1}
\end{pmatrix}.
$$
Recall the assumption \ref{eq-flip} that $\theta^{(j)}_{c\mid 1} > \theta^{(j)}_{c\mid 0}$ for all $j\in[p]$ and $c\in[d-1]$, which implies $\theta^{(j)}_{d\mid 1} < \theta^{(j)}_{d\mid 0}$. Therefore the following inequality always holds for any $c\in[d-1]$,
$$
\theta^{(j)}_{c\mid 0}\cdot \theta^{(j)}_{d\mid 1} - 
\theta^{(j)}_{c\mid 1}\cdot \theta^{(j)}_{d\mid 0} < 0,
$$
which implies each $d\times 2$ factor matrix in the definition of ${\bo\Psi}_1$, ${\bo\Psi}_2$, and ${\bo\Psi}_3$ has full column rank 2. Since the Kronecker product of full-rank matrices is still full-rank, we obtain that each of ${\bo\Psi}_1$, ${\bo\Psi}_2$, ${\bo\Psi}_3$ has full column rank $2^K$. 

Next further group the variable $z_3$ and all the remaining variables $y_{3K+1}, \ldots, y_{p}$ (if they exist) into another discrete variable $z_4$ with $d^{p-2K}$ categories. Denote the conditional probability table of $z_4$ by ${\bo\Psi}_4$, which has size $d^{p-2K} \times 2^K$. Then by definition there is
$$
{\bo\Psi}_4 = {\bo\Psi}_3 \bigodot 
\underbrace{\bo\Phi_{3K+1} \bigodot \bo\Phi_{3K+2} \cdots \bigodot \bo\Phi_{p}}_{p-3K\text{ matrices}}.
$$
Since every matrix in the above Khatri-Rao product is a conditional probability table with each column summing to one, the  ${\bo\Psi}_3$ can be obtained by summing appropriate rows of ${\bo\Psi}_4$. This indeed indicates that the column rank of ${\bo\Psi}_4$ will not be smaller than that of ${\bo\Psi}_3$, so ${\bo\Psi}_4$ also has full rank $2^K$. 
Note that for alternative parameters $\bo\Phi_j$ there is
\begin{align*}
    \Big({\bo\Psi}_1 \bigodot {\bo\Psi}_2 \bigodot {\bo\Psi}_4 \Big) \bcdot \nnu = 
    \Big(\ov{\bo\Psi}_1 \bigodot \ov{\bo\Psi}_2 \bigodot \ov{\bo\Psi}_4 \Big) \bcdot \ov\nnu 
\end{align*}
Now we invoke Kruskal's theorem \citep{kruskal1977three} as follows on the uniqueness of three-way tensor decompositions.
Let $\mathbf M_1, \mathbf M_2, \mathbf M_3$ be three matrices of size $a_m\times r$ for $m=1,2,3$,  and $\mathbf N_1, \mathbf N_2, \mathbf N_3$ be three matrices each with $r$ columns. Suppose $\bigodot_{m=1}^3 \mathbf M_m  \cdot \one = \bigodot_{m=1}^3 \mathbf N_m \cdot \one$. Denote by $\rank_{\text{Kr}}(\mathbf M)$ the Kruskal rank of a matrix $\mathbf M$, which is the maximum number $R$ such that every $R$ columns of $\mathbf M$ are linearly independent.
If $\rank_{\text{Kr}}(\mathbf M_1) + \rank_{\text{Kr}}(\mathbf  M_2) + \rank_{\text{Kr}}(\mathbf  M_3) \geq 2r + 2$, then Kruskal's theorem guarantees that there exists a permutation matrix $\mathbf P$ and three invertible diagonal matrices $\mathbf D_m$ with $\mathbf D_1 \mathbf D_2 \mathbf D_3 =\mathbf I_r$ and $\mathbf N_m =\mathbf  M_m \mathbf D_m \mathbf P$ for each $m=1,2,3$.

Based on Kruskal's theorem stated above, we can show that $\bo\Psi_m=\ov{\bo\Psi}_m$ for $m=1,2,4$ and $\nnu=\ov\nnu$ up to a column latent class permutation. 
Finally, note that both individual entries $\ov\theta_{c\mid 1}^{(j)}$, $\ov\theta_{c\mid 0}^{(j)}$, and the graphical matrix $\ov\G$ can read off from the $\ov{\bo\Psi}_m$. This implies the $\ov\btheta$ and $\ov\G$ must also equal the $\btheta$ and $\G$ up to a latent variable permutation. The proof is complete. 
\qed


\subsection{Proof of Proposition \ref{thm-cdina-str}}
We extend the proof of Theorem 1 in \cite{gu2019dina} from the binary-response DINA model to the CatDINA model.
Fix an arbitrary response category $c\in[d-1]$. We can group all of the other categories in $[d]\setminus\{c\}$ into one big category, so that the model reduces to the binary-response DINA model. Given a $\G$ matrix, consider true parameters $(\bo\theta, \nnu)$ satisfying \eqref{eq-nu-positive} and \eqref{eq-flip} and alternative parameters $(\ov{\bo\theta}, \ov\nnu)$. Define
\begin{align*}
    s_j = 1- \theta^{(j)}_{c\mid 1},\quad u_j = \theta^{(j)}_{c\mid 0},
\end{align*}
then $(s_j, u_j)$ can be viewed as the new slipping and guessing parameters under the reduced binary-response DINA model.
For the alternative parameters, similarly define $\ov s_j = 1- \ov \theta^{(j)}_{c\mid 1}$ and $\ov u_j = \ov \theta^{(j)}_{c\mid 0}$.
{We next show that if $(\bo\theta, \nnu)$ and $(\ov{\bo\theta}, \ov\nnu)$ lead to the same distribution of the observed vector $\bo y$, then it implies certain equations for the new binary-response DINA model. 
Given any response vector $\bo y\in\times_{j=1}^p [d]$, we introduce surrogate response variables $\mathbf R=(R_1,\ldots,R_p)\in\{0,1\}^p$ as follows:
\begin{align*}
    R_j = \begin{cases}
        1, & \text{if } y_j = c;\\
        0, & \text{if } y_j \neq c.
    \end{cases}
\end{align*}
Then, the fact that $(\bo\theta, \nnu)$ and $(\ov{\bo\theta}, \ov\nnu)$ lead to the same distribution of $\bo y$ implies that the following holds for any pattern $\bo r\in\{0,1\}^p$:
\begin{align*}
    \sum_{\aaa\in\{0,1\}^K}\MP(\bo R= \bo r\mid\aaa) \nu_{\aaa} = 
    \sum_{\aaa\in\{0,1\}^K}\ov\MP(\bo R=\bo r\mid\aaa) \ov\nu_{\aaa},
\end{align*}
which can be further written as follows:}
\begin{align*}
&~ \sum_{\aaa\in\{0,1\}^K} \nu_{\aaa} \prod_{j=1}^p \Big[(1-s_j)^{\mathbbm{1}(\aaa\succeq\bo g_j)} u_j^{\mathbbm{1}(\aaa\nsucceq\bo g_j)} \Big]^{\mathbbm{1}(r_j = c)} \Big[1-(1-s_j)^{\mathbbm{1}(\aaa\succeq\bo g_j)} u_j^{\mathbbm{1}(\aaa\nsucceq\bo g_j)} \Big]^{\mathbbm{1}(r_j \neq c)}
\\
=&~
\sum_{\aaa\in\{0,1\}^K} \ov\nu_{\aaa} \prod_{j=1}^p \Big[(1-\ov s_j)^{\mathbbm{1}(\aaa\succeq\bo g_j)} \ov u_j^{\mathbbm{1}(\aaa\nsucceq\bo g_j)} \Big]^{\mathbbm{1}(r_j = c)} \Big[1-(1-\ov s_j)^{\mathbbm{1}(\aaa\succeq\bo g_j)} \ov u_j^{\mathbbm{1}(\aaa\nsucceq\bo g_j)} \Big]^{\mathbbm{1}(r_j \neq c)}.
\end{align*}
Now note that the above system of $2^p$ equations are exactly the same as the $2^p$ equations under the binary-response DINA model. 
Also, our parameter assumptions \eqref{eq-nu-positive} and \eqref{eq-flip} are consistent with the assumptions in \cite{gu2019dina} that $\nu_{\aaa}>0$ for all $\aaa$ and $1-s_j>u_j$ for all $j\in[p]$.
Therefore, when the $\G$ matrix satisfy the C-R-D conditions, we have $\nu_{\aaa} = \ov\nu_{\aaa}$ for all $\aaa\in\{0,1\}^K$, $s_j = \ov s_j$ and $u_j = \ov u_j$ for all $j\in[p]$ following the conclusion in Theorem 1 in \cite{gu2019dina}. This proves the identifiability of $\nnu$ and $\Big\{\theta^{(j)}_{c\mid 1}, \theta^{(j)}_{c\mid 0}: j\in[p]\Big\}$ in the CatDINA model. Since the response category $c$ chosen above is an arbitrary category, we have shown that all the $\theta$-parameters  $\Big\{\theta^{(j)}_{c\mid 1}, \theta^{(j)}_{c\mid 0}: j\in[p], c\in[d]\Big\}$ are identifiable. This shows the strict identifiability of all parameters in the CatDINA model and completes the proof of the proposition.
\qed

\subsection{Proof of Theorem \ref{thm-cdina-gen}}

\textbf{Proof of part (a) about the generic identifiability conclusion.}
Below we rewrite the form of the $\G$ matrix stated in the theorem,
\begin{align*}
\G = \begin{pmatrix}
1 & \zero \\
1 & \uu \\
\hline
\zero & \G^*
\end{pmatrix}.
\end{align*}
Under the above $\G$ matrix, suppose true parameters $(\btheta, \nnu)$ and alternative parameters $(\ov\btheta, \ov\nnu)$ give rise to the same distribution of the observed vector $\bo y$.
Fix a response category $c\in[d-1]$.
For notational convenience, denote by ${\mathbb{P}}(\cdot)$ the probability distribution under the true parameters, and denote by $\overline{\mathbb{P}}(\cdot)$ the probability distribution under the alternative ones.
For any binary pattern $\aaa\in\{0,1\}^K$, denote
$$
\theta^{(j)}_{c\mid \aaa}=\MP(y_j = c\mid\bo a = \aaa),\quad \ov\theta^{(j)}_{c\mid \aaa} = \ov\MP(y_j = c\mid\bo a = \aaa).
$$
For a $(K-1)$-dimensional binary vector $\aaa^*\in\{0,1\}^{K-1}$,  let $(0,\aaa^*)$, $(1,\aaa^*) \in\{0,1\}^K$ denote two $K$-dimensional binary vectors.

Recall that the $(p-2)\times (K-1)$ submatrix $\G^*$ satisfies the C-R-D condition by our assumption in theorem.
Note that when fixing $j\in[p]$ and $c\in[d]$ and varying $\aaa$, $\Big\{\theta^{(j)}_{c\mid \aaa}:\aaa\in\{0,1\}^J\Big\}$ can only take two different values: either $\theta^{(j)}_{c\mid 0}$ or $\theta^{(j)}_{c\mid 1}$.
Since the $\G$ matrix satisfies $g_{3,1} = g_{4,1} = \cdots = g_{p,1} = 0$, we have that for all $j\geq 3$, the observed variable $y_j$ does not depend on the first latent variable $a_1$. As a result, the conditional probability $\theta^{(j)}_{c\mid \aaa}=\MP(y_j = c\mid \bo a = \aaa)$ also does not depend on whether $\alpha_1 = 1$ or $\alpha_1 = 0$.
This fact implies the following equality: 
$$\theta^{(j)}_{c\mid (1,\aaa^*)} = \theta^{(j)}_{c\mid (0,\aaa^*)},\quad
\ov\theta^{(j)}_{c\mid (1,\aaa^*)} = \ov\theta^{(j)}_{c\mid (0,\aaa^*)},\quad \forall j\in\{3,\ldots,p\},~\forall \aaa^*\in\{0,1\}^{K-1}.
$$
{We next use a similar spirit as the proof of Proposition \ref{thm-cdina-str}.}
Fix a response category $c\in[d-1]$. Given any response vector $\bo y\in\times_{j=1}^p [d]$, introduce surrogate response variables $\mathbf R=(R_1,\ldots,R_p)\in\{0,1\}^p$ as:
\begin{align*}
    R_j = \begin{cases}
        1, & \text{if } y_j = c;\\
        0, & \text{if } y_j \neq c.
    \end{cases}
\end{align*}
For any pattern $\bo r\in\{0,1\}^p$, we have
\begin{align*}
    \sum_{\aaa\in\{0,1\}^K}\MP(\bo R= \bo r\mid\aaa) \nu_{\aaa} = 
    \sum_{\aaa\in\{0,1\}^K}\ov\MP(\bo R=\bo r\mid\aaa) \ov\nu_{\aaa}
\end{align*}
Now for any $\bo r = (r_1, r_2, \bo r^*)\in\{0,1\}^p$, 
\begin{align*}
&~\MP(\mathbf R = \bo r) \\
=&~ \sum_{\aaa\in\{0,1\}^K} \MP(\mathbf R = \bo r, \bo a = \aaa)
\\
=&~ \sum_{\aaa\in\{0,1\}^K} \MP(R_1=r_1, R_2=r_2,\bo a = \aaa) \MP(\mathbf R_{3:p} = \bo r^*\mid R_1=r_1, R_2=r_2,\bo a = \aaa) \\
=&~ \sum_{\aaa\in\{0,1\}^K} \MP(R_1=r_1, R_2=r_2,\bo a = \aaa) \MP(\mathbf R_{3:p} = \bo r^*\mid \bo a = \aaa) \\
=&~ \sum_{\aaa\in\{0,1\}^K} \MP(R_1=r_1, R_2=r_2,\bo a = \aaa) \prod_{j=3}^p \MP(R_j = r_j\mid \bo a_{2:K} = \aaa_{2:K}) 
\quad (\text{since }g_{j,1}=0\text{ for }j\geq 3)
\\
=&~ \sum_{\aaa^*\in\{0,1\}^{K-1}} \Big[\MP(R_1=r_1, R_2=r_2,\bo a_{2:K} = \aaa^*, a_1 = 1)  
+ \MP(R_1=r_1, R_2=r_2,\bo a_{2:K} = \aaa^*, a_1 = 0)\Big]
\\
& \qquad\qquad\qquad \times \prod_{j=3}^p \MP(R_j = r_j\mid \bo a_{2:K} = \aaa^*)
\\
=&~ \sum_{\aaa^*\in\{0,1\}^{K-1}} \MP(R_1=r_1, R_2=r_2,\bo a_{2:K} = \aaa^*) \prod_{j=3}^p \MP(R_j = r_j\mid \bo a_{2:K} = \aaa^*).
\end{align*}
Since the true parameters and alternative parameters give the same marginal distribution of $\mathbf R$, they satisfy the following $2^{p-2}$ equations when fixing $(r_1, r_2)$ and varying $\bo r_{3:J}$ in $\{0,1\}^{p-2}$:
\begin{align}\label{eq-r12}
    &~\sum_{\aaa^*\in\{0,1\}^{K-1}} \MP(R_1=r_1, R_2=r_2,\bo a_{2:K} = \aaa^*) \prod_{j=3}^p \MP(R_j = r_j\mid \bo a_{2:K} = \aaa^*)\\ \notag
    =&~ \sum_{\aaa^*\in\{0,1\}^{K-1}} \ov\MP(R_1=r_1, R_2=r_2,\bo a_{2:K} = \aaa^*) \prod_{j=3}^p \ov\MP(R_j = r_j\mid \bo a_{2:K} = \aaa^*).
\end{align}
Now we obtain an interesting and important observation: fixing $(r_1,r_2)$ in one of $(0,0), (1,0), (0,1)$, and $(1,1)$, the above $2^{p-2}$ equations can be equivalently viewed as characterizing another CatDINA model with $p-2$ questions (which are the original questions $y_3,\ldots,y_p$) and $K-1$ latent variables (which are the original latent variables $a_2,\ldots,a_K$), and the new graphical matrix for this model is just the submatrix $\G_{3:p,2:K} = \G^*$ of the original $\G$ matrix.
Since $\G^*$ satisfies the C-R-D conditions, so the parameters for this CatDINA model are strictly identifiable, so 
\begin{align}\label{eq-from3}
\MP(R_j = r_j\mid \bo a_{2:K} = \aaa^*) &= \ov\MP(R_j = r_j\mid \bo a_{2:K} = \aaa^*),
\\
\label{eq-r1r2}
\MP(R_1=r_1, R_2=r_2,\bo a_{2:K} = \aaa^*) &= \ov  \MP(R_1=r_1, R_2=r_2,\bo a_{2:K} = \aaa^*).
\end{align}
Equation \eqref{eq-from3} above directly implies all the $\theta$-parameters associated with $y_3,\ldots,y_p$ are identifiable:
\begin{align}\label{eq-theta3}
    \theta^{(j)}_{c\mid 0} = \ov\theta^{(j)}_{c\mid 0},\quad \theta^{(j)}_{c\mid 1} = \ov\theta^{(j)}_{c\mid 1},\quad \forall j\in\{3,\ldots,p\}.
\end{align}
Now we spell out the four equations implied by \eqref{eq-r1r2} when $(r_1,r_2)$ varies in $\{0,1\}^2$:
\begin{align*}
\begin{cases}
\MP(R_1=0, R_2=0,\bo a_{2:K} = \aaa^*) = \ov  \MP(R_1=0, R_2=0,\bo a_{2:K} = \aaa^*),\\
\MP(R_1=0, R_2=1,\bo a_{2:K} = \aaa^*) = \ov  \MP(R_1=0, R_2=1,\bo a_{2:K} = \aaa^*),\\
\MP(R_1=1, R_2=0,\bo a_{2:K} = \aaa^*) = \ov  \MP(R_1=1, R_2=0,\bo a_{2:K} = \aaa^*),\\
\MP(R_1=1, R_2=1,\bo a_{2:K} = \aaa^*) = \ov  \MP(R_1=1, R_2=1,\bo a_{2:K} = \aaa^*);
\end{cases}
\end{align*}
which are equivalent to the following system of equations (by adding up appropriate equations):
\begin{align}\label{eq-r1r2-prob}
\begin{cases}
\MP(R_1\geq 0, R_2\geq 0,\bo a_{2:K} = \aaa^*) = \ov  \MP(R_1\geq 0, R_2\geq 0,\bo a_{2:K} = \aaa^*),\\
\MP(R_1\geq 1, R_2\geq 0,\bo a_{2:K} = \aaa^*) = \ov  \MP(R_1\geq 1, R_2\geq 0,\bo a_{2:K} = \aaa^*),\\
\MP(R_1\geq 0, R_2\geq 1,\bo a_{2:K} = \aaa^*) = \ov  \MP(R_1\geq 0, R_2\geq 1,\bo a_{2:K} = \aaa^*),\\
\MP(R_1=1, R_2= 1,\bo a_{2:K} = \aaa^*) = \ov  \MP(R_1=1, R_2= 1,\bo a_{2:K} = \aaa^*).
\end{cases}
\end{align}
For notational simplicity, we next denote 
$$
\theta^{(j)}_{c\mid 0} = u_j, \quad \theta^{(j)}_{c\mid 1} = w_j, ~~ j=1,2,
$$
and define similar notations for the alternative parameters with $\ov \theta^{(j)}_{c\mid 0} = \ov u_j$ and $\ov \theta^{(j)}_{c\mid 1} = \ov w_j$ for $j=1,2$.
Recall that the $\G$ matrix has the second row being $(1, \bo u)$, so for any $\aaa^*\in\{0,1\}^{K-1}$ satisfying $\aaa^* \succeq \uu$ (i.e. vector $\aaa$ is elementwisely greater than or equal to vector $\uu$), the left hand side of each equation in \eqref{eq-r1r2-prob} becomes
\begin{align*}
\begin{cases}
\MP(R_1\geq 0, R_2\geq 0, \bo a_{2:K} = \aaa^*) = \MP(\bo a_{2:K} = \aaa^*)= \nu_{(0,\aaa^*)}+  \nu_{(1,\aaa^*)},
\\[2mm]
\MP(R_1\geq 1, R_2\geq 0, \bo a_{2:K} = \aaa^*) 
= \MP(R_1\geq 1, \bo a_{2:K} = \aaa^*) 
= u_1\cdot \nu_{(0,\aaa^*)} +  w_1\cdot \nu_{(1,\aaa^*)},\\[2mm]
\MP(R_1\geq 0, R_2\geq 1, \bo a_{2:K} = \aaa^*) 
= \MP(R_2\geq 1, \bo a_{2:K} = \aaa^*) 
= u_2\cdot \nu_{(0,\aaa^*)} +  w_2\cdot \nu_{(1,\aaa^*)},\\[2mm]
\MP(R_1=1, R_2= 1,\bo a_{2:K} = \aaa^*) 
= u_1 u_2 \cdot \nu_{(0,\aaa^*)} +  v_1 v_2\cdot \nu_{(1,\aaa^*)}
\end{cases}
\end{align*}
Therefore, any $\aaa^*\in\{0,1\}^{K-1}$ satisfying $\aaa^* \succeq \uu$, Eq.~\eqref{eq-r1r2-prob} can be simply written as
\begin{align}\label{eq-key12}
\begin{cases}
\nu_{(0,\aaa^*)}+  \nu_{(1,\aaa^*)} = \ov \nu_{(0,\aaa^*)}+\ov  \nu_{(1,\aaa^*)};\\[2mm]
u_1\cdot \nu_{(0,\aaa^*)} +  w_1\cdot \nu_{(1,\aaa^*)} = \ov u_1 \cdot\ov \nu_{(0,\aaa^*)} + \ov w_2 \cdot\ov \nu_{(1,\aaa^*)};\\[2mm]
u_2 \cdot \nu_{(0,\aaa^*)} +  w_2 \cdot \nu_{(1,\aaa^*)} = \ov u_2 \cdot \ov \nu_{(0,\aaa^*)} + \ov w_2\cdot \ov \nu_{(1,\aaa^*)};\\[2mm]
u_1 u_2 \cdot \nu_{(0,\aaa^*)} +  v_1 v_2\cdot \nu_{(1,\aaa^*)} 
= \ov u_1 \ov u_2 \cdot\bar \nu_{(0,\aaa^*)} + \ov w_1 \ov w_2 \cdot\bar \nu_{(1,\aaa^*)}.
\end{cases}
\end{align}
First, we transform the above system of equations to obtain
\begin{align*}
\begin{cases}
(u_1 - w_1) (u_2 - \ov w_2) \cdot \nu_{(0,\aaa^*)} = (\ov u_1 - w_1) (\ov u_2 - \ov w_2) \cdot \ov \nu_{(0,\aaa^*)},\\[2mm]
(u_2-\ov w_2)\cdot \nu_{(0,\aaa^*)} + (w_2-\ov w_2)\cdot \nu_{(1,\aaa^*)} = (\ov u_2-\ov w_2)\cdot \ov\nu_{(0,\aaa^*)}.
\end{cases}
\end{align*}
According to Lemma \ref{lem-neq}, the right hand sides of the two equations above are both nonzero. Therefore we can take the ratio of these two equations, which gives
\begin{equation}\notag
f_1(\aaa^*) = \frac{( u_1 -  w_1)\cdot( u_2 - \bar w_2)}{( u_2 - \bar w_2)  + ( w_2 - \bar w_2)\cdot  \nu_{(1,\aaa^*)} /  \nu_{(0,\aaa^*)}}
=\bar  u_1 -  w_1,\quad \forall \aaa^*\in \{0,1\}^{K-1}.
\end{equation}
So for two arbitrary vectors $\aaa_1^*$, $\aaa_2^*\in\{0,1\}^{K-1}$ with $\aaa^*_1, \aaa^*_2 \succeq \uu$, our above deduction gives $f_1(\aaa_1^*)=f_1(\aaa^*_2)=\bar u_1 -  w_1$. This implies
\begin{align}\notag
\frac{( u_1 -  w_1)\cdot( u_2 - \bar w_2)}{( u_2 - \bar w_2)  + ( w_2 - \bar w_2)\cdot  \nu_{(1,\aaa^*_1)} /  \nu_{(0,\aaa^*_1)}}
&=
\frac{( u_1 -  w_1)\cdot( u_2 - \bar w_2)}{( u_2 - \bar w_2)  + ( w_2 - \bar w_2)\cdot  \nu_{(1,\aaa^*_2)} /  \nu_{(0,\aaa^*_2)}},
\\\notag
\Longleftrightarrow\qquad
(w_2 - \bar w_2)\cdot  \frac{\nu_{(1,\aaa_1^*)}}{\nu_{(0,\aaa_1^*)}}
&=(w_2 - \bar w_2)\cdot  \frac{\nu_{(1,\aaa_2^*)}}{\nu_{(0,\aaa_2^*)}},
\\\label{eq-w2}
(w_2 - \bar w_2)\cdot \Big(\frac{\nu_{(1,\aaa_1^*)}}{\nu_{(0,\aaa_1^*)}}& - \frac{\nu_{(1,\aaa_2^*)}}{\nu_{(0,\aaa_2^*)}} \Big) 
= 0.
\end{align}
The last equality above has an important implication: as long as there exist one pair of different vectors  $\aaa_1^*$, $\aaa_2^*\in\{0,1\}^{K-1}$ with $\aaa^*_1, \aaa^*_2 \succeq \uu$ such that 
\begin{align}\label{eq-nset1}
\nu_{(1,\aaa_1^*)} \nu_{(0,\aaa_2^*)} - \nu_{(0,\aaa_1^*)} \nu_{(1,\aaa_2^*)} \neq 0,
\end{align}
then we will have 
$$
\frac{\nu_{(1,\aaa_1^*)}}{\nu_{(0,\aaa_1^*)}} - \frac{\nu_{(1,\aaa_2^*)}}{\nu_{(0,\aaa_2^*)}} \neq 0.
$$
Using the above inequality to examine \eqref{eq-w2}, we get $w_2 = \ov w_2$.
Note that under the assumption stated in the theorem that $\bo u\neq \one_{K-1}$, there indeed exists such two distinct vectors $\aaa^*_1$, $\aaa^*_2$ satisfying $\aaa^*_1, \aaa^*_2\succeq \bo u$. Therefore, $w_2 = \ov w_2$ (i.e., $w_2 = \theta^{(2)}_{c\mid 1}$ is identifiable) as long as $\nnu\not\in \mathcal N$ where the set $\mathcal N$ is defined below:
\begin{align}\label{eq-mathn}
\mathcal N
= \{\nnu\text{ satisfies } 
    \nu_{(1,\aaa_1^*)} \nu_{(0,\aaa_2^*)} - \nu_{(0,\aaa_1^*)}  \nu_{(1,\aaa_2^*)} = 0 \text{ for any } \aaa_1^*\neq \aaa_2^*\text{ with } \aaa_1^*,\, \aaa_2^* \succeq\uu\}.
\end{align}

Next we transform the system of equations \eqref{eq-key12} in another way to get
$$
\begin{cases}
( w_1 -  u_1)\cdot( w_2 - \bar u_2)\cdot \nu_{(1,\aaa^*)}
=
(\bar w_1 -  u_1)\cdot(\bar w_2 - \bar u_2)\cdot \bar \nu_{(1,\aaa^*)};\\[2mm]
(u_2 - \bar u_2)\cdot  \nu_{(0,\aaa^*)}  + (w_2 - \bar u_2)\cdot  \nu_{(1,\aaa^*)}
= (\bar w_2 - \bar u_2)\cdot \bar \nu_{(1,\aaa^*)}.
\end{cases}
$$
The ratio of the above two equations is
$$
f_2(\aaa^*) := \frac{( w_1 -  u_1)\cdot( w_2 - \bar u_2)}{(u_2 - \bar u_2) \cdot  \nu_{(0,\aaa^*)} / \nu_{(1,\aaa^*)}  + (w_2 - \bar u_2)} = \bar w_1 -  u_1.
$$
Again we have $f_2(\aaa_1^*) = f_2(\aaa_2^*)$ for any $\aaa_1^*,\aaa_2^*\succeq\uu$ with $\aaa_1^*\neq\aaa_2^*$. Therefore,
$$
(u_2 - \bar u_2) \cdot  \frac{\nu_{(0,\aaa_1^*)}}{ \nu_{(1,\aaa_1^*)}} = (u_2 - \bar u_2) \cdot  \frac{\nu_{(0,\aaa_2^*)}}{\nu_{(1,\aaa_2^*)}}, \quad\Longrightarrow\quad
(u_2 - \bar u_2) \cdot \left( \frac{\nu_{(0,\aaa_1^*)}}{ \nu_{(1,\aaa_1^*)}} -  \frac{ \nu_{(0,\aaa_2^*)}}{\nu_{(1,\aaa_2^*)}} \right) = 0.
$$
Therefore, as long as $\nnu\not\in\mathcal N$ for $\mathcal N$ defined earlier in \eqref{eq-mathn}, we also have $u_2 = \bar u_2$ and $u_2$ is identifiable.
Now we have shown that $(w_2, u_2)$ are identifiable if $\nnu\not\in\mathcal N$.

Now note that the system of equations \eqref{eq-key12} are symmetric about $(w_2, u_2)$ and $(w_1, u_1)$. Therefore, $(w_1, u_1)$ are also identifiable if $\nnu\not\in\mathcal N$. In summary, when $\nnu\not\in\mathcal N$, all the $\theta$-parameters associated with $y_1$ and $y_2$ are identifiable 
$$
\theta^{(j)}_{c\mid 0} = \ov\theta^{(j)}_{c\mid 0},\quad \theta^{(j)}_{c\mid 1} = \ov\theta^{(j)}_{c\mid 1},\quad \forall j\in\{1,2\}.
$$
Combining the above conclusion with \eqref{eq-theta3} and noting that the category $c\in[d-1]$ is arbitrary, we obtain that all the $\theta$-parameters associated with $y_1,\ldots,y_p$ are identifiable if $\nnu\not\in\mathcal N$.

Next, we show that the proportion parameters $\nnu$ are also identifiable when $\nnu\not\in\mathcal N$.
First recall that $\bo\Phi^{(j)}$ is a $d\times 2^K$ with the $(c,\aaa)$th entry being $\mathbb P(y_j=c\mid \bo a=\aaa)$. When $\nnu\not\in\mathcal N$, we have shown $\bo\Phi^{(j)} = \ov{\bo\Phi}^{(j)}$ for all $j\in[p]$ when $\nnu\not\in\mathcal N$.
since the first $K$ rows of the $\G$ matrix form an identity matrix $I_K$, consider the following equation under the true and alternative parameters:
\begin{align*}
    \bigodot_{j=1}^K \bo\Phi^{(j)} \cdot \nnu = \bigodot_{j=1}^K \ov{\bo\Phi}^{(j)} \cdot \ov\nnu  &= \bigodot_{j=1}^K {\bo\Phi}^{(j)} \cdot \ov\nnu;
    \\
    \Longrightarrow\quad
    \bigodot_{j=1}^K \bo\Phi^{(j)} \cdot \Big( \nnu-\ov\nnu\Big) &= \zero_{2^K}.
\end{align*}
Since $\G_{1:K,:} = \mathbf I_K$, we can use a similar argument as in the proof of Proposition \ref{prop-3chi} and show that the $d^K \times 2^K$ matrix $\bigodot_{j=1}^K \bo\Phi^{(j)}$ has full column rank $2^K$; specifically, this is because this matrix has the following equivalent representation as a Kronecker product of $K$ rank-two matrices:
\begin{align*}
\bigodot_{j=1}^K \bo\Phi^{(j)}
= \bigotimes_{j=1}^K  \begin{pmatrix}
\theta^{(j)}_{1\mid 0} ~&~ \theta^{(j)}_{1\mid 1} \\[2mm]
\vdots ~&~ \vdots\\[2mm]
\theta^{(j)}_{d\mid 0} ~&~ \theta^{(j)}_{d\mid 1}
\end{pmatrix}.
\end{align*}
The fact that $\bigodot_{j=1}^K \bo\Phi^{(j)}$ has full column rank implies that the earlier equation $\bigodot_{j=1}^K \bo\Phi^{(j)} \cdot \Big( \nnu-\ov\nnu\Big) = \zero_{2^K}$ has a unique solution
$$
\nnu-\ov\nnu = \zero_{2^K},
$$
so $\nnu=\ov\nnu$ holds (that is, $\nu_{\aaa}=\ov \nu_{\aaa}$ for all $\aaa\in\{0,1\}^K$).
This means we have shown all the model parameters are identifiable when $\nnu\not\in \mathcal N$ with $\mathcal N$ defined \eqref{eq-mathn}.
Since $\mathcal N$ is a measure-zero subset of the probability simplex $\mathcal S^{2^K-1}$, we have proved the generic identifiability of the CatDINA model parameters.


\bigskip\noindent
\textbf{Proof of part (b) about the blessing of dependence.}
We next examine the non-identifiable set $\mathcal N$ defined in \eqref{eq-mathn}  and reveal the blessing of dependence.
Consider $\nnu\in\mathcal N$.
For an arbitrary binary vector $\aaa = (\alpha_1, \aaa^*)$ where the $(K-1)$-dimensional subvector satisfies $\aaa^*\succeq\uu$, we have
\begin{align*}
&~\mathbb P(a_{1} = \alpha_1) \mathbb P(\boldsymbol{a}_{2:K} = \aaa^*)\\
=&~ \Big(\sum_{\bo\beta\in\{0,1\}^{K-1}} \nu_{(\alpha_1,\,\bo\beta)}\Big) (\nu_{(\alpha_1,\,\aaa^*)} + \nu_{(1-\alpha_1,\,\aaa^*)}) 
\\
= &~ \sum_{\bo\beta\in\{0,1\}^{K-1}} \nu_{(\alpha_1,\,\bo\beta)} \nu_{(\alpha_1,\,\aaa^*)} + \sum_{\bo\beta\in\{0,1\}^{K-1}} \nu_{(\alpha_1,\,\bo\beta)} \nu_{(1-\alpha_1,\,\aaa^*)}
\\
=&~ \sum_{\bo\beta\in\{0,1\}^{K-1}} \nu_{(\alpha_1,\,\bo\beta)} \nu_{(\alpha_1,\,\aaa^*)} + \sum_{\bo\beta\in\{0,1\}^{K-1}} \nu_{(1-\alpha_1,\,\bo\beta)} \nu_{(\alpha_1,\,\aaa^*)}
\quad (\text{because we consider } \nnu \in \mathcal N)
\\
=&~ \Big(\sum_{\bo\beta\in\{0,1\}^{K-1}} \nu_{(\alpha_1,\,\bo\beta)} + \sum_{\bo\beta\in\{0,1\}^{K-1}} \nu_{(1-\alpha_1,\,\bo\beta)}\Big) \nu_{(\alpha_1,\,\aaa^*)} \\
=&~ \nu_{(\alpha_1,\,\aaa^*)} 
= \mathbb P(\boldsymbol{a} = \aaa).
\end{align*}
The third equality above holds because by the definition of $\mathcal N$, the following holds for any $\nnu \in \mathcal N$: 
$$\nu_{(\alpha_1,\,\bo\beta)}  \nu_{(1-\alpha_1,\,\aaa^*)} = \nu_{(1-\alpha_1,\,\bo\beta)}  \nu_{(\alpha_1,\,\aaa^*)},\quad
\forall \alpha_1\in\{0,1\}, ~
\forall \aaa^*,\bo\beta \succeq\uu.
$$ 
Now we obtain that if $\nnu\in\mathcal N$, then $\mathbb P(\boldsymbol{a} = (\alpha_1,\aaa^*)) = \mathbb P(a_{1} = \alpha_1) \mathbb P(\boldsymbol{a}_{2:K} = \aaa^*)$ for any $\alpha_1\in\{0,1\}$ and $\aaa^*\succeq\uu$.
This implies if $\nnu\in\mathcal N$, then the first latent variable $a_1$ is conditionally independent of the other latent variables $\boldsymbol{a}_{2:K}$ provided that $\boldsymbol{a}_{2:K}\succeq\uu$.

On the other hand, if latent variables $a_1$ and $\boldsymbol{a}_{2:K}$ are conditionally independent given $\boldsymbol{a}_{2:K} \succeq\uu$, then for any $\aaa^*\succeq\uu$ we have
\begin{align*}
\frac{\nu_{(1,\aaa^*)}}{\nu_{(0,\aaa^*)}} 
= \frac{\mathbb P(\boldsymbol{a} = (1,\aaa^*))}{\mathbb P(\boldsymbol{a} = (0,\aaa^*))}
= \frac{\mathbb P(a_{1} = 1) \mathbb P(\boldsymbol{a}_{2:K} = \aaa^*)}{\mathbb P(a_{1} = 0) \mathbb P(\boldsymbol{a}_{2:K} = \aaa^*)}
= \frac{\mathbb P(a_{1} = 1)}{\mathbb P(a_{1} = 0)} =:\rho.
\end{align*}
This means for any $\aaa_1^*\neq\aaa_2^*$ with $\aaa_1^*,\aaa_2^*\succeq\uu$, the equality $\nu_{(1,\aaa^*_1)}/\nu_{(0,\aaa^*_1)} - \nu_{(1,\aaa^*_2)}/\nu_{(0,\aaa^*_2)} = \rho-\rho = 0$ must hold, which is equivalent to $\nu_{(1,\aaa^*_1)} \nu_{(0,\aaa^*_2)} - \nu_{(0,\aaa^*_1)} \nu_{(1,\aaa^*_2)} = 0$ for any $\aaa_1^*\neq\aaa_2^*$ with $\aaa_1^*,\aaa_2^*\succeq\uu$.
This means if $a_1\indep\boldsymbol{a}_{2:K}\mid \boldsymbol{a}_{2:K}\succeq\uu$ holds, then  $\nnu\in\mathcal N$ must be true.
%

Now we have proved the statement that $$a_1\indep\bo a_{2:K}\mid \boldsymbol{a}_{2:K}\succeq\uu,$$ is exactly equivalent to the statement that
$$\nnu\in\mathcal N = \{ \nu_{(1,\aaa^*_1)} \nu_{(0,\aaa^*_2)} - \nu_{(0,\aaa^*_1)} \nu_{(1,\aaa^*_2)} = 0\text{ holds for any }\aaa_1^*\neq\aaa_2^*\text{ with }\aaa_1^*,\aaa_2^*\succeq\uu\}.$$
This completes the proof of the theorem.
\qed

\subsection{Proof of Proposition \ref{prop-depdep}}
Denote the marginal probability mass function of the vector $(\alpha_{k_1}, \alpha_{k_2})$ by $\{\tilde \nu_{(\alpha_{k_1}, \alpha_{k_2})};\; (\alpha_{k_1}, \alpha_{k_2})\in\{0,1\}^2\}$. Each $\tilde \nu_{(\alpha_{k_1}, \alpha_{k_2})} = \mathbb P(a_{k_1} = \alpha_{k_1}, a_{k_2} = \alpha_{k_2})$ can be obtained by summing up appropriate entries of the vector $(\nu_{\aaa};\; \aaa\in\{0,1\}^K)$.
Similarly, denote the marginal distribution of each $\alpha_k\in\{0,1\}$ by $\tilde\nu_{\alpha_k} = \mathbb P(a_{k}=\alpha_k)$.
Then we have
\begin{align*}
&~\mathbb P(\{y_j=c_j:\; j\in \ch(a_{k_1})\},~ \{y_m=c_m:\; m\in \ch(a_{k_2})\}) \\
= &~
 \sum_{\aaa\in\{0,1\}^K} \nu_{\aaa} \prod_{j\in  \ch(a_{k_1}, a_{k_2})} 
    \prod_{k=1}^{K}
    \left[
    \left(\theta^{(j)}_{c_j\mid 1}\right)^{\alpha_{k}} 
    \cdot
    \left(\theta^{(j)}_{c_j\mid 0}\right)^{1-\alpha_{k}}\right]^{\mathbbm{1}(g_{j,k}=1)} \\
= &~
 \sum_{\aaa\in\{0,1\}^K} \nu_{\aaa} \prod_{j\in  \ch(a_{k_1}, a_{k_2})} 
\mathbb P(y_j \mid \bo g_j, \aaa) \\
= &~ \sum_{(\alpha_{k_1}, \alpha_{k_2})\in\{0,1\}^2} \tilde\nu_{(\alpha_{k_1}, \alpha_{k_2})} 
\prod_{j\in  \ch(a_{k_1}, a_{k_2})} \mathbb P(y_j \mid \bo g_j, \aaa) \\
= &~ \sum_{(\alpha_{k_1}, \alpha_{k_2})\in\{0,1\}^2} \tilde\nu_{(\alpha_{k_1}, \alpha_{k_2})} 
\prod_{j\in \ch(a_{k_1})} \mathbb P(y_j \mid \alpha_{k_1})
\prod_{j\in \ch(a_{k_2})} \mathbb P(y_j \mid \alpha_{k_2})
\\
\stackrel{(\star)}{=}
&~ \sum_{(\alpha_{k_1}, \alpha_{k_2})\in\{0,1\}^2} \tilde\nu_{\alpha_{k_1}} \tilde\nu_{\alpha_{k_2}}   \prod_{j\in \ch(a_{k_1})} \mathbb P(y_j \mid \alpha_{k_1})
\prod_{j\in \ch(a_{k_2})} \mathbb P(y_j \mid \alpha_{k_2}) 
\\
= &~
\left(\sum_{\alpha_{k_1}\in\{0,1\}} \tilde\nu_{\alpha_{k_1}} \prod_{j\in \ch(a_{k_1})} \mathbb P(y_j \mid \alpha_{k_1}) \right)
\cdot 
\left(\sum_{\alpha_{k_2}\in\{0,1\}} \tilde\nu_{\alpha_{k_2}} \prod_{j\in \ch(a_{k_2})} \mathbb P(y_j \mid \alpha_{k_2}) \right)
\\
= &~
\left(\sum_{\aaa\in\{0,1\}^K} \tilde\nu_{\alpha_{k_1}} \prod_{j\in \ch(a_{k_1})} \mathbb P(y_j \mid \aaa, \bo g_j) \right)
\cdot 
\left(\sum_{\aaa\in\{0,1\}^K} \tilde\nu_{\alpha_{k_2}} \prod_{j\in \ch(a_{k_2})} \mathbb P(y_j \mid \aaa, \bo g_j) \right)
\\
= &~ \mathbb P(\{y_j=c_j:\; j\in \ch(a_{k_1})\})\cdot \mathbb P(\{y_m=c_m:\; m\in \ch(a_{k_2})\}),
\end{align*}
where $(\star)$ follows from the independence between $\alpha_{k_1}$ and $\alpha_{k_2}$.

On the other hand, the above deduction also implies that if $\{y_j;\; j\in\ch(a_{k_1})\}$ and $\{y_j;\; j\in\ch(a_{k_2})\}$ are not independent, then
\begin{align*}
&~\mathbb P(\{y_j=c_j:\; j\in \ch(a_{k_1})\},~ \{y_m=c_m:\; m\in \ch(a_{k_2})\}) \\
&\quad - 
\mathbb P(\{y_j=c_j:\; j\in \ch(a_{k_1})\})\cdot \mathbb P(\{y_m=c_m:\; m\in \ch(a_{k_2})\}) \\
=&~ 
\sum_{(\alpha_{k_1}, \alpha_{k_2})\in\{0,1\}^2} (\tilde\nu_{(\alpha_{k_1}, \alpha_{k_2})} - \tilde\nu_{\alpha_{k_1}} \tilde\nu_{\alpha_{k_2}}  )
\prod_{j\in \ch(a_{k_1})} \mathbb P(y_j \mid \alpha_{k_1})
\prod_{j\in \ch(a_{k_2})} \mathbb P(y_j \mid \alpha_{k_2})\\
\neq &~ 0
\end{align*}
for some $\{c_j;\; j\in\ch(a_{k_1})\}$. This implies that there must exist some $(\alpha_{k_1}, \alpha_{k_2})\in\{0,1\}^2$ such that $\tilde\nu_{(\alpha_{k_1}, \alpha_{k_2})} - \tilde\nu_{\alpha_{k_1}} \tilde\nu_{\alpha_{k_2}} \neq 0$.
This means $a_{k_1} \notindep a_{k_2}$.
The proof of the Proposition is complete.
\qed



\subsection{Proof of Lemma \ref{lem-neq}}
We use proof by contradiction. Suppose $\ov\theta^{(j)}_{c\mid 0} = \theta^{(j)}_{c\mid 1}$ for some $j$ and $c$. First consider $c<d$ then by our assumption there is $\theta^{(j)}_{c\mid 0} < \theta^{(j)}_{c\mid 1}$.
Then we have
\begin{align*}
    &~\sum_{\aaa:\, \aaa\succeq \cg_j} \nu_{\aaa} \theta^{(j)}_{c\mid 1} + \sum_{\aaa:\, \aaa\nsucceq \cg_j} \nu_{\aaa} \theta^{(j)}_{c\mid 0} 
    \\
    &~<~ \theta^{(j)}_{c\mid 1} =\ov\theta^{(j)}_{c\mid 0}
    ~<~
    \sum_{\aaa:\, \aaa\succeq \ov\cg_j} \ov\nu_{\aaa} \ov\theta^{(j)}_{c\mid 1} + \sum_{\aaa:\, \aaa\nsucceq \ov\cg_j} \ov\nu_{\aaa} \ov\theta^{(j)}_{c\mid 0}
\end{align*}
{The above inequality can be equivalently written as 
$$
\sum_{\aaa} \nu_{\aaa} \theta^{(j)}_{c\mid \aaa} < \sum_{\aaa} \ov\nu_{\aaa} \ov\theta^{(j)}_{c\mid \aaa},
$$
which directly contradicts the following fact implied by that $\ov\G, \ov\btheta, \ov\nnu$ lead to the same distribution of the observed vector $\bo y$,}
$$\mathbb P(y_j=c\mid \G, \btheta,\nnu) = \sum_{\aaa\in\{0,1\}^K} \nu_{\aaa} \theta^{(j)}_{c\mid\aaa} = \sum_{\aaa\in\{0,1\}^K} \ov\nu_{\aaa} \ov\theta^{(j)}_{c\mid\aaa}
= \mathbb P(y_j=c\mid \ov\G, \ov\btheta, \ov\nnu).$$
This contradiction shows $\ov\theta^{(j)}_{c\mid 0} \neq \theta^{(j)}_{c\mid 1}$ must hold for any $ 1\leq c\leq d-1$. Similarly we can prove $\ov\theta^{(j)}_{d\mid 0} \neq \theta^{(j)}_{d\mid 1}$. By symmetry we also have $ \ov\theta^{(j)}_{c\mid 1} \neq \theta^{(j)}_{c\mid 0}$ for all $j\in[p]$ and all $c\in[d]$. This proves Lemma \ref{lem-neq}.
\qed

\subsection{Proof of Lemma \ref{lem-gg}}
We next prove by contradiction. Assume there exists some $h\in[K]$ and a set $\mca \subseteq [K]\setminus\{h\}$, such that 
\begin{align}\label{eq-pbc}
    \bigvee_{k\in\mathcal A}~ \bar \cg_k \succeq \bar\cg_h
\end{align}
and also assume that there exists a set $\mcb\subseteq\{K+1,\ldots,J\}$ such that $\max_{m\in \mathcal B}\; g_{m,h} = 0$ and $\max_{m\in \mathcal B}\; g_{m,k}=1 ~\text{for all}~ k\in \mathcal A$.
{We next explain why assuming $\bigvee_{k\in\mathcal A}~ \bar \cg_k \succeq \bar\cg_h$ in \eqref{eq-pbc} is the correct starting point in the proof by contradiction. Under the definition of $\bigvee$ in \eqref{eq-bigvee}, the conclusion stated in Lemma \ref{lem-gg} is $\bigvee_{k\in\mathcal A}~ \bar \cg_k \nsucceq \bar\cg_h$, which is equivalent to stating that 
$
\max_{k\in\mathcal A} \ov g_{km}
< \ov g_{hm}$ for some $m\in[K]$.
Then in order to prove by contradiction, we assume the negation of the above statement, which is:
$$\max_{k\in\mathcal A} \ov g_{km}
 \geq \ov g_{hm} \text{ for all } m\in[K]
\quad\Longleftrightarrow \quad
\bigvee_{k\in\mathcal A}~ \bar \cg_k \succeq \bar\cg_h.
$$
Therefore, assuming $\bigvee_{k\in\mathcal A}~ \bar \cg_k \succeq \bar\cg_h$ is the correct procedure of proof by contradiction, and any contradiction as a consequence of this assumption would prove the original conclusion $\bigvee_{k\in\mathcal A}~ \bar \cg_k \nsucceq \bar\cg_h$ of Lemma \ref{lem-gg}.}

First, for each $c\in\{1,\ldots,d-1\}$ define
\begin{align}\label{eq-delta1}
    \bo\Delta_{1:p,\;c}^{*} &= 
    \bar\theta_{c\mid 1}^{(h)} \ee_h
    +
    \sum_{k\in\mca}\bar\theta_{c\mid 0}^{(k)} \ee_k
    +
    \sum_{m=K+1}^p \theta_{c\mid 0}^{(m)} \ee_m,
    \\
    \yy_c^{*} &= c\left(\ee_h
    +
    \sum_{k\in\mca} \ee_k
    +
    \sum_{m=K+1}^p \ee_m\right).
\end{align}
Under the above definitions, we claim that the row vector of $\bigodot_{j\in[p]} \Big(\overline{\bo\Phi}^{(j)}-\bo\Delta_{j,:}\bcdot\one^\top_{2^K} \Big)$ indexed by response pattern $\yy_c^{*}$ is an all-zero vector. 
To see this, note that for any $\aaa\in\{0,1\}^K$, the corresponding element in the row denoted by $\bar t_{\yy_c^{*}, \aaa}$ contains a factor 
$$
f_{\aaa} =\left(\bar \theta_{c\mid\aaa}^{(h)} - \bar\theta_{c\mid 1}^{(h)} \right) \prod_{k\in\mca} \left(\bar \theta_{c\mid\aaa}^{(k)} - \bar\theta_{c\mid 0}^{(k)} \right).
$$
This factor $f_{\aaa}$ is potentially nonzero only if $\bar \theta_{c\mid\aaa}^{(j)} \neq \bar\theta_{c\mid 1}^{(h)}$ and $\bar \theta_{c\mid\aaa}^{(k)} \neq \bar\theta_{c\mid 0}^{(k)}$ for all $k\in\mca$ (equivalently, $\bar \theta_{c\mid\aaa}^{(k)} = \bar\theta_{c\mid 1}^{(k)}$ for all $k\in\mca$). 
However, this is impossible for any $\aaa$ under the assumption \eqref{eq-pbc} that $\vee_{k\in\mathcal A}~ \bar \cg_k \succeq \bar\cg_h$. This is because for any $\aaa$ such that $\bar \theta_{c\mid\aaa}^{(k)} = \bar\theta_{c\mid 1}^{(k)}$ for all $k\in\mca$, there must be $\aaa\succeq \vee_{k\in\mathcal A}~ \bar \cg_k$, and our assumption \eqref{eq-pbc} further gives $\aaa\succeq  \bar\cg_h$, which means $\bar\theta^{(h)}_{c\mid\aaa} = \bar\theta^{(h)}_{c\mid 1}$.
This proves $f_{\aaa}=0$ must hold for all $\aaa\in\{0,1\}^K$.
Since $\bar t_{\yy_c^{*}, \aaa}$ contains $f_{\aaa}$ as a factor, there is $\bar t_{\yy_c^{*}, \aaa}=0$ for all $\aaa\in\{0,1\}^K$.
Therefore $\sum_{\aaa\in\{0,1\}^K}  t_{\yy_c^{*}, \aaa}  \nu_{\aaa} = \sum_{\aaa\in\{0,1\}^K} \bar t_{\yy_c^{*}, \aaa} \bar \nu_{\aaa}=0$. 
Now we focus on $t_{\yy_c^{*},\aaa}$. Note that $\G_{(K+1):2K} = \I_K$. Due to the term $\sum_{m=K+1}^p \theta_{c\mid 0}^{(m)} \ee_m$ in the definition of $\bo\Delta$ in \eqref{eq-delta1}, we have $t_{\yy_c^{*}, \aaa}$ is potentially nonzero only if $\aaa=\one_K$. Therefore
\begin{align*}
    0 = \nu_{\one_K} \left(\theta_{c\mid 1}^{(h)} - \bar\theta_{c\mid 1}^{(h)}\right)
    \left(\theta_{c\mid 1}^{(k)} -\bar\theta_{c\mid 0}^{(k)}\right)
    \prod_{m=K+1}^p \left(\theta_{c\mid 1}^{(m)} - \theta_{c\mid 0}^{(m)}\right).
\end{align*}
This gives $\theta_{c\mid 1}^{(h)} =\bar\theta_{c\mid 1}^{(h)}$.

Second, recall the set $\mcb\subseteq\{K+1,\ldots,p\}$ defined earlier satisfies that $\max_{m\in \mathcal B}\; g_{m,h} = 0$ and $\max_{m\in \mathcal B}\; g_{m,k}=1 ~\text{for all}~ k\in \mathcal A$. For each $c\in\{1,\ldots,d-1\}$, now define
\begin{align}\label{eq-delta2}
    \bo\Delta_{1:p,\;c}^{**} &= 
    \bar\theta_{c\mid 1}^{(h)} \ee_h
    +
    \sum_{k\in\mca}\bar\theta_{c\mid 0}^{(k)} \ee_k
    +
    \sum_{m\in\mcb} \theta_{c\mid 0}^{(m)} \ee_m,
    \\
    \yy_c^{**} &= c\left(\ee_h
    +
    \sum_{k\in\mca} \ee_k
    +
    \sum_{m\in\mcb} \ee_m\right).
\end{align}
Under the above new definitions, we still claim that the row vector of $\bigodot_{j\in[p]} \Big(\overline{\bo\Phi}^{(j)}-\bo\Delta_{j,:}\bcdot\one^\top_{2^K} \Big)$ indexed by response pattern $\yy_c^{**}$ is an all-zero vector. The reasoning is similar to that in the previous paragraph after \eqref{eq-delta1}, because that earlier argument only depends on the fact that $\bo\Delta_{1:p,\;c}^{*}$ contains the first two groups of terms $\bar\theta_{c\mid 1}^{(h)} \ee_h + \sum_{k\in\mca}\bar\theta_{c\mid 0}^{(k)} \ee_k$, and $\bo\Delta_{1:p,\;c}^{**}$ also contains such two groups of terms.
Therefore $\sum_{\aaa\in\{0,1\}^K}  t_{\yy_c^{**}, \aaa}  \nu_{\aaa} = \sum_{\aaa\in\{0,1\}^K} \bar t_{\yy_c^{**}, \aaa} \bar \nu_{\aaa} = 0$. 
Considering the $\theta_{c\mid 1}^{(h)} =\bar\theta_{c\mid 1}^{(h)}$ obtained in the end of last paragraph, the element $t_{\yy_c^{**}, \aaa}$ would equal zero if $\alpha_h = 1$; this is because $t_{\yy_c^{**}, \aaa}$ contains a factor $\theta_{c\mid \aaa}^{(h)} - \bar \theta_{c\mid 1}^{(h)}$ which equals zero if $\alpha_h = 1$.
This means the element $t_{\yy_c^{**}, \aaa}$ has the following property,
\begin{align*}
    &t_{\yy_c^{**}, \aaa}= \\
    &\begin{cases}
    \paren{\theta_{c\mid 0}^{(h)} - \bar \theta_{c\mid 1}^{(h)}} \prod_{k\in\mca}\paren{\theta_{c\mid 1}^{(k)} - \bar\theta_{c\mid 0}^{(k)}} \prod_{m\in\mcb} \paren{ \theta_{c\mid 1}^{(m)} - \theta_{c\mid 0}^{(m)}}, 
    & \alpha_{h} = 0\text{ and }\aaa \succeq \bigvee_{m\in \mca\cup\mcb}\;\cg_m;\\
    0, & \text{otherwise.}
    \end{cases}
\end{align*}
Now an important observation is that the following set $\mathcal M$ of $K$-dimensional binary vectors is nonempty,
$$
\mathcal M :=
\left\{\aaa\in\{0,1\}^K: \alpha_{h} = 0, \text{ and }\aaa \succeq \bigvee_{m\in \mca\cup\mcb}\;\cg_m  \right\}.
$$
This is true because  $\max_{m\in \mathcal B}\; g_{m,h} = 0$ and $\max_{m\in \mathcal B}\; g_{m,k}=1 ~\text{for all}~ k\in \mathcal A$, and hence $\aaa \succeq \vee_{m\in \mca\cup\mcb}\;\cg_m$ still allows for $\alpha_h$ (that is, the $h$th element of $\aaa$) to be potentially zero.
Now the equation
$$\sum_{\aaa\in\{0,1\}^K}  t_{\yy_c^{**}, \aaa}  \nu_{\aaa} = 0$$
can be equivalently written as
\begin{align}\label{eq-contr}
     \paren{\theta_{c\mid 0}^{(h)} - \bar \theta_{c\mid 1}^{(h)}} \prod_{k\in\mca}\paren{\theta_{c\mid 1}^{(k)} - \bar\theta_{c\mid 0}^{(k)}} \prod_{m\in\mcb} \paren{ \theta_{c\mid 1}^{(m)} - \theta_{c\mid 0}^{(m)}} \paren{\sum_{\aaa\in\mathcal M} \nu_{\aaa}} = 0.
\end{align}
Recall that $\theta_{c\mid 0}^{(j)} \neq \theta_{c\mid 1}^{(j)}$ for all $j\in[p]$ and $\bar \theta_{c\mid 0}^{(j)} \neq \bar \theta_{c\mid 1}^{(j)}$ for all $j\in[p]$, and also $\sum_{\aaa\in\mathcal M} \nu_{\aaa}>0$.
Therefore each factor of the left hand side of \eqref{eq-contr} is nonzero, which gives a contradiction. This means the assumption \eqref{eq-pbc} in the beginning of the proof is incorrect and the Lemma \ref{lem-gg} is proved.
\qed

\section{EM algorithms for the BLESS model}\label{sec-algodetail}

\noindent\textbf{When $\G$ is known and fixed.}
We first consider the scenario where the measurement graph $\G$ is known or already estimated, and describe the EM algorithm for the continuous parameters $\btheta$ and $\nnu$.
Denote the subject-specific latent pattern indicators by $z_{i,\aaa} = \mathbbm{1}(\bo a = \aaa)$ and $\mathbf Z=(z_{i,\aaa};\; i\in[N], \aaa\in\{0,1\}^K)$.
An important observation is that the following equivalent formulation holds under the BLESS model, 
\begin{align*}
\theta^{(j)}_{c\mid\aaa}
=& \mathbbm{1}(\aaa\succeq \bo g_j) \theta^{(j)}_{c\mid 1} + [1-\mathbbm{1}(\aaa\succeq \bo g_j)] \theta^{(j)}_{c\mid 0} \\
=& \Big(\sum_{k=1}^K \alpha_k g_{j,k}\Big) \theta^{(j)}_{c\mid 1} + \Big(1-\sum_{k=1}^K \alpha_k g_{j,k}\Big) \theta^{(j)}_{c\mid 0}.
\end{align*}
Therefore, the complete data log-likelihood function under the BLESS model can be written as follows,
\begin{align*}
&~\ell(\btheta, \nnu\mid \mathbf Y, \mathbf Z, \G) \\
=&~ \sum_{i=1}^N \Big\{  \sum_{\aaa\in\{0,1\}^K} \Big[z_{i,\aaa}\log(\nu_{\aaa}) + \sum_{j=1}^p \sum_{c=1}^{d} (y_{ijc}z_{i,\aaa})\log(\theta^{(j)}_{c\mid\aaa}) \Big] \Big\} \\
=&~ \sum_{i=1}^N \sum_{\aaa\in\{0,1\}^K} \sum_{j=1}^p \sum_{c=1}^{d}  y_{ijc}z_{i,\aaa}   \Big[\sum_{k=1}^K\alpha_k g_{j,k} \log(\theta^{(j)}_{c\mid 1}) + \Big(1-\sum_{k=1}^K\alpha_k g_{j,k}\Big) \log(\theta^{(j)}_{c\mid 0})\Big] \\
&\quad + \sum_{\aaa\in\{0,1\}^K} \sum_{i=1}^N z_{i,\aaa}\log(\nu_{\aaa}).
\end{align*}
The above formulation allows for a convenient EM algorithm to compute the MLE, which iterates through E-step and a M-step towards convergence of the marginal log-likelihood. We present this EM algorithm in Algorithm \ref{algo-em1}.

\begin{algorithm}[h!]
\caption{EM algorithm for the BLESS Model when $\G$ is Known}\label{algo-em1}

\SetKwInOut{Input}{input}
\SetKwInOut{Output}{Output}

\KwData{Observed data array $\mathbf Y=(y_{ijc})_{N\times p \times d}\in\{0,1\}^{N\times p\times d}$ and number of latent variables $K$.}
 \While{not converged}{
 
\tcp{E Step}
Calculate the conditional expectation of each $z_{i,\aaa}$:
$$
\mathbb E[z_{i,\aaa}] \leftarrow
\frac{ \nu_{\aaa}\prod_{j=1}^p \prod_{c=1}^{d} (\theta^{(j)}_{c\mid \aaa})^{y_{ijc}} }{ \sum_{\aaa' \in \{0,1\}^K} \nu_{\aaa'}\prod_{j=1}^p \prod_{c=1}^{d} (\theta^{(j)}_{c\mid \aaa'})^{y_{ijc}} },\quad i\in[N],~ \aaa\in\{0,1\}^K.
$$

  \tcp{M Step}
  Update continuous parameters $\btheta$ and $\nnu$:
    \begin{align*}
    \theta_{c\mid 1}^{(j)} &\leftarrow \frac{\sum_{\aaa}\sum_{i=1}^N \mathbb E[z_{i,\aaa}]\sum_{k=1}^K \alpha_k g_{j,k} y_{ijc}}{\sum_{\aaa}\sum_{i=1}^N \mathbb E[z_{i,\aaa}] \sum_{k=1}^K \alpha_k g_{j,k}},\quad j\in[p],~ c\in[d];\\
    \theta_{c\mid 0}^{(j)} &\leftarrow \frac{\sum_{\aaa}\sum_{i=1}^N \mathbb E[z_{i,\aaa}](1-\sum_{k=1}^K \alpha_k g_{j,k}) y_{ijc}}{\sum_{\aaa}\sum_{i=1}^N \mathbb E[z_{i,\aaa}] (1-\sum_{k=1}^K \alpha_k g_{j,k})},\quad  j\in[p],~ c\in[d];\\
    \nu_{\aaa} &\leftarrow \frac{\sum_{i=1}^N \mathbb E[z_{i,\aaa}]}{\sum_{\aaa'\in\{0,1\}^K} \sum_{i=1}^N \mathbb E[z_{i,\aaa'}]}, \quad \aaa\in\{0,1\}^K.
    \end{align*}
    Update $\theta_{c\mid \aaa}^{(j)} = \mathbbm{1}(\aaa\succeq \bo g_j)\theta_{c\mid 1}^{(j)} + (1-\mathbbm{1}(\aaa\succeq \bo g_j)) \theta_{c\mid 0}^{(j)}$ after completing the M Step.
    
}

\Output{Parameters $\btheta$, $\nnu$.}
\end{algorithm}

\medskip
\noindent
\textbf{When $\G$ is unknown.}
We next describe a more general approximate EM algorithm that jointly estimate the $\G$ matrix and the continuous parameters.
Introduce notation $\bo s = (s_1,\ldots,s_p)$ with each $s_j\in[K]$, where $s_j = k$ if $g_{j,k}=1$. Then there is a one-to-one correspondence between the vector $\bo s$ and matrix $\G$.
We can just augment the EM algorithm described above by adding the following step of drawing samples of $\{g_{j,k}\}$ in the E step.
The conditional distribution of each $s_j$ is the Categorical distribution with parameters as follows,
\begin{align*}
\gamma_{j,k} = \mathbb P(s_j=k\mid -) 
&= \frac{\prod_{\aaa} \prod_{i=1}^N \prod_{c=1}^d 
[(\theta_{c\mid 1}^{(j)})^{ \alpha_k }  (\theta_{c\mid 0}^{(j)})^{ 1- \alpha_k }]^{y_{ijc} z_{i,\aaa}}}
{\sum_{k'=1}^K \prod_{\aaa} \prod_{i=1}^N \prod_{c=1}^d [(\theta_{c\mid 1}^{(j)})^{ \alpha_{k'} }  (\theta_{c\mid 0}^{(j)})^{ 1- \alpha_{k'} }]^{y_{ijc} z_{i,\aaa}}}\\
&= \frac{\prod_{\aaa} \prod_{c=1}^d 
[(\theta_{c\mid 1}^{(j)})^{ \alpha_k }  (\theta_{c\mid 0}^{(j)})^{ 1- \alpha_k }]^{\sum_{i=1}^N y_{ijc} z_{i,\aaa}}}
{\sum_{k'=1}^K \prod_{\aaa}\prod_{c=1}^d [(\theta_{c\mid 1}^{(j)})^{ \alpha_{k'} }  (\theta_{c\mid 0}^{(j)})^{ 1- \alpha_{k'} }]^{\sum_{i=1}^N y_{ijc} z_{i,\aaa}}}.
\end{align*}
Since the entries of the $\G$ are needed in the E step of the algorithm, after obtaining the $\gamma_{j,k}$, we let $s_j=k$ if the current posterior probability $\mathbb P(s_j=k\mid -)$ is the largest among all the $K$ posterior probabilities.
Such a procedure has a similar spirit to a classification EM algorithm \citep{celeux1992cem}, but the difference is that we use this procedure to update the graphical structure (the entries of the measurement graph), instead of updating the subject-specific latent variables as in classification EM.
We present this general EM algorithm dealing with unknown $\G$ in Algorithm \ref{algo-em2}.

\begin{algorithm}[h!]
\caption{Approximate EM algorithm for the BLESS Model when $\G$ is Unknown}\label{algo-em2}

\SetKwInOut{Input}{input}
\SetKwInOut{Output}{Output}

\KwData{Observed data array $\mathbf Y=(y_{ijc})_{N\times p \times d}\in\{0,1\}^{N\times p\times d}$ and number of latent variables $K$.}

 \While{not converged}{
 
\tcp{E Step}
Calculate the conditional expectation of each $z_{i,\aaa}$:
$$
\mathbb E[z_{i,\aaa}] = \mathbb P(\bo a_i=\aaa\mid -) \leftarrow
\frac{ \nu_{\aaa}\prod_{j=1}^p \prod_{c=1}^{d} (\theta^{(j)}_{c\mid \aaa})^{y_{ijc}} }{ \sum_{\aaa' \in \{0,1\}^K} \nu_{\aaa'}\prod_{j=1}^p \prod_{c=1}^{d} (\theta^{(j)}_{c\mid \aaa'})^{y_{ijc}} },\quad i\in[N],~ \aaa\in\{0,1\}^K.
$$

Draw each $\bo a_i$ from the above Categorical distribution with $2^K$ components.

For each $j\in[p]$ and $k\in[K]$, let
\begin{align*}
\gamma_{j,k} &\leftarrow 
\frac{\prod_{\aaa} \prod_{c=1}^d 
[(\theta_{c\mid 1}^{(j)})^{ \alpha_k }  (\theta_{c\mid 0}^{(j)})^{ 1- \alpha_k }]^{\sum_{i=1}^N y_{ijc} z_{i,\aaa}}}
{\sum_{k'=1}^K \prod_{\aaa}\prod_{c=1}^d [(\theta_{c\mid 1}^{(j)})^{ \alpha_{k'} }  (\theta_{c\mid 0}^{(j)})^{ 1- \alpha_{k'} }]^{\sum_{i=1}^N y_{ijc} z_{i,\aaa}}},
\\
g_{j,k} &  \leftarrow 1 ~\text{ if } ~\gamma_{j,k}=\max \{\gamma_{j,1},\ldots,\gamma_{j,K}\}; ~~ g_{j,k}\leftarrow 0~~ \text{otherwise}.
\end{align*}

\tcp{M Step}
Update continuous parameters $\btheta$ and $\nnu$:
    \begin{align*}
    \theta_{c\mid 1}^{(j)} &\leftarrow \frac{\sum_{\aaa}\sum_{i=1}^N \mathbb E[z_{i,\aaa}]\sum_{k=1}^K \alpha_k g_{j,k} y_{ijc}}{\sum_{\aaa}\sum_{i=1}^N \mathbb E[z_{i,\aaa}] \sum_{k=1}^K \alpha_k g_{j,k}},\quad j\in[p],~ c\in[d];\\
    \theta_{c\mid 0}^{(j)} &\leftarrow \frac{\sum_{\aaa}\sum_{i=1}^N \mathbb E[z_{i,\aaa}](1-\sum_{k=1}^K \alpha_k g_{j,k}) y_{ijc}}{\sum_{\aaa}\sum_{i=1}^N \mathbb E[z_{i,\aaa}] (1-\sum_{k=1}^K \alpha_k g_{j,k})},\quad  j\in[p],~ c\in[d];\\
    \nu_{\aaa} &\leftarrow \frac{\sum_{i=1}^N \mathbb E[z_{i,\aaa}]}{\sum_{\aaa'\in\{0,1\}^K} \sum_{i=1}^N \mathbb E[z_{i,\aaa'}]}, \quad \aaa\in\{0,1\}^K.
    \end{align*}
    Update $\theta_{c\mid \aaa}^{(j)} = \mathbbm{1}(\aaa\succeq \bo g_j)\theta_{c\mid 1}^{(j)} + (1-\mathbbm{1}(\aaa\succeq \bo g_j)) \theta_{c\mid 0}^{(j)}$ after completing the M Step.
    
}

\Output{Measurement graph $\G$ and parameters $\btheta$, $\nnu$.}
\end{algorithm}


\section{Real-world example about a prevention science survey}\label{sec-addreal} 
An influential paper in prevention science \cite{lanza2013} used the latent class model (LCM; with a unidimensional latent variable) to analyse the treatment effects on different latent subgroups, and illustrated the method using a dataset extracted
from the National Longitudinal Survey of Adolescent Health (NLSAH). 
Observed data for each subject are $p=6$ dichotomized characteristics: household poverty; single-parent status;
peer cigarette use; peer alcohol use; neighborhood unemployment; and neighborhood poverty. 
These observables actually measure three risks, with the first two measuring ($\alpha_1$) \emph{household risk}, the middle two measuring 
($\alpha_2$) \emph{peer risk}, and the last two measuring 
($\alpha_3$) \emph{neighborhood risk}.
According to the estimated conditional probability tables of the observed variables given the five latent classes, \cite{lanza2013} interpreted the latent classes as (a) Overall low risk, (b) Peer risk, (c) Household \& neighborhood (economic) risk, (d) Household \& peer risk, and (e) Overall high (multicontext) risk. 
Interestingly, we note that the analysis in \cite{lanza2013} lends itself to a reformulation using the BLESS model, and we argue that such a reformulation provides an interpretable graphical modeling alternative to plain latent class analysis.
Specifically, if viewing the three underlying risks as three latent variables, then the latent-to-observed measurement graph indeed takes a star-forest shape; see Table \ref{tab-prev1} for details of the $\G$ matrix.
More importantly, the aforementioned five latent classes can be nicely formulated as five different binary configurations of the three latent risks, as $(0,0,0)$, $(1,0,0)$, $(1,0,1)$, $(1,1,0)$, and $(1,1,1)$, respectively.
Here $\alpha_k=1$ indicates the higher risk group while $\alpha_k=0$ indicates the lower risk group.
See Table \ref{tab-prev2} for the multidimensional binary configurations of latent classes.

\begin{table}[h!]
    \caption{Prevention science survey example reformulated using the BLESS model. Latent-to-observed measurement graph structure $\G_{6\times 3}$.}
    \label{tab-prev1}
    
    \centering
    \resizebox{\textwidth}{!}{
    \begin{tabular}{llccc}
    \toprule
     & \multirow{3}{*}{Item Content} & & Fine-grained Latent Risks & \\
     \cmidrule(lr){3-5}
     & & $\alpha_1$ & $\alpha_2$ & $\alpha_3$ \\
     & & Household risk & Peer risk & Neighborhood risk \\
     \midrule
       1 & Household poverty & 1 & 0 & 0 \\
       2 & Single-parent status & 1 & 0 & 0\\
       3 & Peer cigarette use & 0 & 1 & 0 \\
       4 & Peer alcohol use & 0 & 1 & 0 \\
       5 & Neighborhood unemployment & 0 & 0 & 1 \\
       6 & Neighborhood poverty & 0 & 0 & 1 \\
       \bottomrule
    \end{tabular}
    }
\end{table}

\begin{table}[h!]
    \caption{Prevention science survey example reformulated using the BLESS model. Five latent classes obtained and explained in \cite{lanza2013}, and reformulated in the interpretable multidimensional-binary latent variable format.}
    \label{tab-prev2}
    
    \centering

\resizebox{\textwidth}{!}{
    \begin{tabular}{llccc}
    \toprule
     & \multirow{3}{*}{Latent Class Explanation} & & Fine-grained Latent Risks & \\
     \cmidrule(lr){3-5}
     & & $\alpha_1$ & $\alpha_2$ & $\alpha_3$ \\
     & & Household risk & Peer risk & Neighborhood risk \\
     \midrule
       1 & Overall low risk  & 0 & 0 & 0 \\
       2 & Peer risk & 1 & 0 & 0\\
       3 & Household \& neighborhood risk & 1 & 0 & 1 \\
       4 & Household \& peer risk & 1 & 1 & 0 \\
       5 & Overall high risk & 1 & 1 & 1 \\
    \bottomrule
    \end{tabular}
}
    
\end{table}

Because $\G$ shows that each latent risk has exactly two observed children characteristics, this example analysed in \cite{lanza2013} can be exactly regarded as satisfying the minimal conditions for generic identifiability of the BLESS model.
As \cite{lanza2013} did not include the original dataset that they analyzed which is extracted and sampled from the NLSAH survey, we do not perform the test here but point out the testing procedure is just the same as what we conducted in Section \ref{sec-prac} in the main text for the TIMSS data.
Specifically, one could simply test the hypothesis of identifiability by testing the marginal independence of the three groups of binary characteristics falling under the household risk, peer risk, and neighborhood risk, respectively.
One plausible conjecture is these three risks are likely interdependent due to the interactions of an adolescent's household, peers, and neighborhood.
In such a case, the BLESS model would be identifiable when applied to the survey dataset, and one could use the BLESS model as a more fine-grained and interpretable graphical modeling alternative to plain latent class analysis.



\end{document}